\documentclass[12pt]{article}
\usepackage{setspace}
\usepackage{authblk}

\RequirePackage[OT1]{fontenc}
\RequirePackage{amsthm}
\RequirePackage[sort]{natbib}
\RequirePackage{fix-cm}

\RequirePackage[colorlinks,citecolor=blue,urlcolor=blue]{hyperref}
\usepackage{amsfonts}
\usepackage{amsgen,amsmath,amstext,amsbsy,amsopn,amssymb,comment}
\usepackage[pdftex]{graphicx}

\usepackage[usenames]{color}

\usepackage{mathtools}
\usepackage{autonum}

\numberwithin{equation}{section}

 \newcommand{\sinf}{{\xi_p}}

\newcommand{\beaa}{\begin{eqnarray*}}
\newcommand{\eeaa}{\end{eqnarray*}}
\newcommand{\lbl}{\label}

\newcommand{\RR}{\mathbf R}

\newcommand{\E}{\mathbb{E}}
\newcommand{\PP}{\mathbb{P}}

\newcommand{\argmin}{\operatornamewithlimits{arg\,min}}

\newcommand{\tW}{\tilde{W}}
\newcommand{\tT}{\tilde{T}_{m,p}}
\newcommand{\tA}{\tilde{A}}
\newcommand{\tf}{\tilde{f}}
\newcommand{\tQ}{\tilde{Q}}

\newcommand{\fwig}[1]{h_{#1}}

\newtheorem{theorem}{Theorem}
\newtheorem{corollary}{Corollary}

\newtheorem{lemma}{Lemma}

\newtheorem{proposition}{Proposition}
\newtheorem{remark}{Remark}

\newcommand{\beq}{\begin{equation}}
\newcommand{\eeq}{\end{equation}}
\newcommand{\beas}{\begin{eqnarray*}}
\newcommand{\eeas}{\end{eqnarray*}}
\newcommand{\bea}{\begin{eqnarray}}
\newcommand{\eea}{\end{eqnarray}}
\newcommand{\bei}{\begin{itemize}}
\newcommand{\eei}{\end{itemize}}
\newcommand{\ben}{\begin{enumerate}}
\newcommand{\een}{\end{enumerate}}

\newcommand{\goto}{\rightarrow}

\begin{document}

%\begin{frontmatter}
\title{Asymptotic Analysis for Extreme Eigenvalues of Principal Minors of Random Matrices}

% \titlerunning{Extreme Eigenvalues of Principal Minors}
%\thankstext{T1}{Footnote to the title with the ``thankstext'' command.}

%\begin{aug}
{
  % \title{\bf Asymptotic Analysis for Extreme Eigenvalues of Principal Minors of Random Matrices}
 \author[1]{T. Tony Cai}
 \author[2]{Tiefeng Jiang}
 \author[2]{Xiaoou Li}
 \affil[1]{Department of Statistics, The Wharton School, University of Pennsylvania}
 \affil[2]{School of Statistics, University of Minnesota}

  \date{}
  \maketitle
} 
% \author{\fnms{Tiefeng} \snm{Jiang}\thanksref{m2}\ead[label=e2]{jiang040@umn.edu}},
% \and
% \author{\fnms{Xiaoou} \snm{Li}\thanksref{m2}
% \ead[label=e3]{lixx1766@umn.edu}
% %\ead[label=u1,url]{http://www.foo.com}
% }

% \thankstext{t1}{Some comment}
% \thankstext{t2}{First supporter of the project}
% \thankstext{t3}{Second supporter of the project}
%\runauthor{F. Author et al.}

% \affiliation{University of Pennsylvania\thanksmark{m1} and University of Minnesota\thanksmark{m2}}

% \address{T. T. Cai\\
% Department of Statistics\\
% The Wharton School\\
% University of Pennsylvania\\
% Philadelphia, PA 19104\\
% \printead{e1}}
% %\phantom{E-mail:\ }\printead*{e1}}

% \address{T. Jiang\\
% School of Statistics\\
% University of Minnesota\\
% Minneapolis, MN 55455\\
% \printead{e2}}

% \address{X. Li\\
% School of Statistics\\
% University of Minnesota\\
% Minneapolis, MN 55455\\
% \printead{e3}}
% %\end{aug}
\begin{abstract}
Consider a standard white Wishart matrix with parameters $n$ and $p$.
Motivated by applications in high-dimensional statistics and signal processing,
we perform asymptotic analysis on the {maxima and minima} of the eigenvalues of all the $m \times m$ principal minors, under the asymptotic regime that $n,p,m$ go to infinity. Asymptotic results concerning extreme eigenvalues of principal minors of real Wigner matrices are also obtained. In addition, we discuss an application of the theoretical results to the construction of compressed sensing matrices, which provides  insights to compressed sensing in signal processing and high dimensional linear regression in statistics.
\end{abstract}

\noindent%

% \begin{keyword}[class=MSC]
% \kwd[Primary ]{60B20,60F99}
% \kwd{60K35}
% %\kwd[; secondary ]{60K35}
% \end{keyword}

% \begin{keyword}
% \kwd{random matrix}
% \kwd{extremal eigenvalues}
% \kwd{maximum of random
% variables}
% \kwd{minimum of random variables}
% \end{keyword}

%\end{frontmatter}

\section{Introduction}

Random matrix theory is traditionally focused on the spectral analysis  of eigenvalues and eigenvectors of a single random matrix. See, for example,
\cite{wigner1955characteristic,wigner1958distribution,dyson1962statistical,dyson1962statisticalb,dyson1962statisticalc,Mehta:2004wq,tracy1994level,tracy1996orthogonal,tracy2000distribution,diaconis2001linear,johnstone2001distribution,johnstone2008multivariate,jiang2004limiting,jiang2004asymptotic,bryc2006spectral,bai2010spectral}. It is important in its own right and has been proved to be a powerful tool in a wide range of fields including  high-dimensional  statistics, quantum physics, electrical engineering, and  number theory.

The laws of large numbers and the limiting distributions for the extreme eigenvalues of the Wishart matrices are now well known, see, e.g., \cite{bai1999} and \cite{johnstone2001distribution,johnstone2008multivariate}.
%Bai (1999), Johnstone (2001, 2008) and El Karoui (2006).
Let ${X}={X}_{n\times p}$ be a random matrix with i.i.d. $N(0,1)$ entries and let $W={X}^{\intercal}{X}$. Let $\lambda_1(W)\ge \cdots \ge \lambda_p(W)$ be the eigenvalues of $W$.
The limiting distribution of the largest eigenvalue $\lambda_1(W)$  satisfies, for $n, p\goto \infty$ with $n/p \to \gamma$,
\beq
\label{TW}
\PP\Big(\frac{\lambda_{1}(W)- \mu_n}{\sigma_n}\leq x\Big) \to F_1(x)
\eeq
where $\mu_n=(\sqrt{n-1}+\sqrt{p})^2$ and $\sigma_n=(\sqrt{n-1} + \sqrt{p})(\frac{1}{\sqrt{n-1}} + \frac{1}{\sqrt{p}})^{1/3}$ and $F_1(x)$ is the distribution function of the Tracy-Widom law of type I. The results for the smallest eigenvalue $\lambda_p(W)$ can be found in, e.g., \cite{edelman1988eigenvalues} and \cite{bai1993limit}. These results have also been extended to generalized Wishart matrices, i.e., the entries of ${X}$ are i.i.d. but not necessarily normally distributed,  in, e.g., \cite{bai2010spectral,peche2009universality,tao2010random}.
%{\red Should we comment on the case $p\gg n$?}

Motivated by applications in high-dimensional statistics and signal processing, we study in this paper the extreme eigenvalues of the principal minors of a Wishart matrix $W$. Write ${X}=(x_{ij})_{n\times p}=({x}_1, \cdots,{x}_p).$  Let $S =\{i_1, \cdots, i_k\}\subset \{1,2,\cdots, p\}$ with the size of $S$ being $k$ and ${X}_S=({x}_{i_1},\cdots, {x}_{i_k})$. Then $W_S={X}_S^{\intercal}{X}_S$ is a $k\times k$ principal minor of $W.$
Denote by $\lambda_1(W_S)\geq \cdots \geq \lambda_k(W_S)$ the eigenvalues of $W_S$  in descending order.
% Denote by $\lambda_1(W_S)\geq \cdots \geq \lambda_k(W_S)$ the eigenvalues of $W_S$  in descending order and let
% \[
% \Lambda_S = \{\lambda_1(W_S), \cdots, \lambda_k(W_S) \}
% \]
%be the set of the eigenvalues of $W_S$.
% Denote by ${\cal J}(k,p)$ the collection of all subsets of $\{1, \ldots,p \}$ of size $k$.
We are interested in the largest and the smallest eigenvalues of all the $k\times k$ principal minors of $W$ in the setting that $n$, $p$, and $k$ are large but $k$ relatively smaller than $\min\{n, p\}$. More specifically, we are interested in the properties of the maximum of the eigenvalues of all $k\times k$ minors:
\begin{equation}\label{eq:lambda-max}
	\lambda_{\max}(k)=\max_{1\leq i\leq k, S\subset \{1,...,p\}, |S|=k}\lambda_i(W_S)
\end{equation}
% \[
% \lambda_{\max}(k)\equiv \max\left\{ \bigcup_{S \in {\cal J}(k,p)}   \Lambda_S \right\}
% \]
and the minimum of the eigenvalues of all $k\times k$ minors:
\begin{equation}\label{eq:lambda-min}
	\lambda_{\min}(k)  =\min_{1\leq i\leq k, S\subset \{1,...,p\}, |S|=k}\lambda_i(W_S),
\end{equation}
where $|S|$ denotes the cardinality of the set $S$.
% \[
% \lambda_{\min}(k)\equiv  \min\left\{ \bigcup_{S \in {\cal J}(k,p)}   \Lambda_S \right\}.
% \]

This is a problem of significant interest in its own right, and it has important applications in statistics and engineering. Before we  establish the properties for the extreme eigenvalues $\lambda_{\max}(k)$ and $\lambda_{\min}(k)$, of the $k\times k$ principal minors of a  Wishart matrix $W$, we first discuss an application in signal processing and statistics, namely the construction of the compressed sensing matrix, as the motivation for our study. The properties of the extreme eigenvalues $\lambda_{\max}(k)$ and $\lambda_{\min}(k)$ can also be used in other applications, including testing for the covariance structure of a high-dimensional Gaussian distribution, which is an important problem in statistics.

%%%%%%%%%%%%%%%%%%%%%%%%%%%%%%%%%%%%%%%%%%
\subsection{Construction of Compressed Sensing Matrices}
\label{CCSM}
%%%%%%%%%%%%%%%%%%%%%%%%%%%%%%%%%%%%%%%%%%

Compressed sensing, which aims to develop efficient data acquisition techniques that allow accurate reconstruction of highly undersampled sparse signals, has  received much attention recently in several fields, including signal processing, applied mathematics and statistics.  The development of the compressed sensing theory also provides crucial insights into inference for high dimensional linear regression in statistics.  It is now well understood that the  constrained $\ell_1$ minimization method provides an effective way for recovering sparse signals. See, e.g.,  \cite{candes2005decoding, candes2007dantzig}, \cite{donoho2006compressed}, and \cite{donoho2006stable}.
More specifically, in compressed sensing, one observes $(X, y)$ with
\begin{equation}\label{eq:modelsignal}
y=X\beta+z
\end{equation}
where $y\in \mathbb{R}^n$,  $X\in \mathbb{R}^{n\times p}$ with $n$ being much smaller than $p$, $\beta\in \mathbb{R}^p$ is a sparse signal of interest, and $z\in\mathbb{R}^n$ is a vector of measurement errors. One wishes to recover the unknown sparse signal $\beta\in\mathbb{R}^p$ based on $(X, y)$ using an efficient algorithm.

Since the number of measurements $n$ is much smaller than the dimension $p$,  without structural assumptions, the signal $\beta$ is under-determined, even in the noiseless case. A usual assumption in compressed sensing is that $\beta$ is sparse  and one of the most commonly used frameworks for sparse signal recovery is the \emph{Restricted Isometry Property} (RIP). See \cite{candes2005decoding}. A vector is  said to be $k$-sparse if $|{\rm supp}(v)|\leq k$, where ${\rm supp}(v)=\{i:v_i\neq 0\}$ is the support of $v$.
In compressed sensing, the RIP  requires subsets of certain cardinality of the columns of $X$ to be close to an orthonormal system.
For an integer $1\leq k\leq p$, define the restricted isometry constant $\delta_k$ to be the smallest non-negative numbers  such that for all $k$-sparse vectors $\beta$,
\begin{equation}\label{RIP}
(1-\delta_k)\|\beta\|^2_2\; \leq \|X \beta\|^2_2 \leq \; (1+\delta_k)\|\beta\|_2^2.
\end{equation}
There are a variety of sufficient conditions on the RIP  for the exact/stable recovery of $k$-sparse signals. A sharp condition was established in \cite{cai2014sparse} and a conjecture was proved in  \cite{zhang2018proof}. Let
\beq
\label{bt}
b_*(t)=\left\{
\begin{array}{ll}
\frac{t}{4-t} \quad & 0<t< {4\over 3} \\
\sqrt{t-1\over t}  \quad & t\ge {4\over 3}
\end{array}
.\right.
\eeq
For any given $t>0$,  the condition $\delta_{tk} < b_*(t)$ guarantees the exact recovery of all $k$ sparse signals in the noiseless case through the constrained $\ell_1$ minimization
\[
\hat \beta = \argmin\{\|\gamma\|_1: y=X\gamma, \gamma\in \RR^p\}.
\]
Moreover, for any $\varepsilon>0$, $\delta_{tk}< b_*(t)+\varepsilon$ is not sufficient to guarantee the exact recovery of all $k$-sparse signals for large $k$.  In addition, the conditions $\delta_{tk} < b_*(t)$  is also shown to be sufficient for stable recovery of approximately sparse signals in the noisy case.

One of the major goals of compressed sensing  is the construction of the measurement matrix ${X}_{n\times p}$, with the number of measurements $n$ as small as possible relative to $p$, such that all $k$-sparse signals can be accurately recovered.  Deterministic construction of large measurement matrices that satisfy  the RIP  is known to be difficult. Instead, random matrices are commonly used. Certain random matrices have been shown to satisfy the RIP conditions with high probability. See, e.g.,  \cite{baraniuk2008simple}. When the measurement matrix $X$ is a Gaussian matrix with i.i.d. $N(0, {1\over n})$ entries,  for any given $t$, the condition $\delta_{tk} < b_*(t)$  is equivalent to that the extreme eigenvalues, $\lambda_{\max}(tk)$ and $\lambda_{\min}(tk)$, of the $tk\times tk$ principal minors of the  Wishart matrix $W={X}^{\intercal}{X}$ satisfy
\[
1-b_*(t) <  \lambda_{\min}(tk)\le \lambda_{\max}(tk) <  1+b_*(t).
\]
Hence the condition (\ref{RIP}) can be viewed as a condition on $\lambda_{\min}(tk)$ and $\lambda_{\max}(tk)$ as defined in \eqref{eq:lambda-max} and \eqref{eq:lambda-min}, respectively.
%Between $\lambda_{\min}(k)$ and $\lambda_{\max}(k)$, the minimum eigenvalue $\lambda_{\min}(k)$ plays a more important role in compressed sensing.

%%%%%%%%%%%%%%%%%%%%%
\subsection{Main results and organization of the paper}
%%%%%%%%%%%%%%%%%%%%%

In this paper, we investigate the asymptotic behavior of the extreme eigenvalues $\lambda_{\max}(m)$ and $\lambda_{\min}(m)$ defined in \eqref{eq:lambda-max} and \eqref{eq:lambda-min}. We also consider the extreme eigenvalues of a related Wigner matrix. We then discuss the application of the results in the construction of compressed sensing matrices.

The rest of the paper is organized as follows. Section \ref{sec:settings} describes the precise setting of the problem. The main results are stated in Section \ref{sec:main-results}.  The proofs of the main theorems are given in Section  \ref{sec:proof}. 
The proofs of all the supporting lemma are given in the Appendix. The proof strategy for the main results is given in Section~\ref{sec:proof-strat}.
%{\red Should we combine Sections \ref{sec:proof1} and \ref{sec:proof}? Or convert a concise version of Section \ref{sec:proof1} into a subsection of Section \ref{sec:main-results}?}

%%%%%%%%%%%%%%%%%%%%%
\section{Problem settings}
\label{sec:settings}
%%%%%%%%%%%%%%%%%%%%%
In this paper, we consider a white Wishart matrix  $W=(w_{ij})_{1\leq i,j\leq p}=X^{\intercal} X$, where $X=(x_{ij})_{1\leq i\leq n,1\leq j\leq p}$ and $x_{ij}$ are independent $N(0,1)$-distributed random variables. For $S\subset\{1,...,p\}$, set the principal minor $W_S=(w_{ij})_{i,j\in S}$.
For an $m\times m$ symmetric matrix $A$, let $\lambda_1(A)$ and $\lambda_m(A)$ denote the largest and the smallest eigenvalues of $A$, respectively. Let
\begin{equation}\label{eq:t-stat}
	T_{m,n,p} = \max_{S\subset\{1,...,p\}, |S|=m}\lambda_1(W_S),
\end{equation}
 and $|S|$ denotes the cardinality of the set $S$. We also define
\begin{equation}\label{eq:v-stat}
{V}_{m,n,p}=  \min_{S\subset\{1,...,p\}, |S|=m}\lambda_m(W_S).	
\end{equation}
 Of interest is the asymptotic behavior of $T_{m,n,p}$ and ${V}_{m,n,p}$ when both $n$ and $p$ grow large.

Notice $W_{ij}$ is the sum of $n$ independent and identically distributed (i.i.d.) random variables. By the standard central limit theorem, for given $i\geq 1$ and $j\geq 1$, we have
\begin{equation}
    \frac{w_{ij}-n}{\sqrt{n}}\Longrightarrow N(0,2)\text{ if } i=j,\text{ and }
    \frac{w_{ij}}{\sqrt{n}} \Longrightarrow N(0,1) \text{ if } i\neq j,
\end{equation}
as $n \to\infty$, where we use ``$\Longrightarrow$'' to indicate convergence in distribution.
Motivated by this limiting distribution, we also consider the Wigner matrix $\tW=(\tilde{w}_{ij})_{1\leq i,j\leq p}$, which is a symmetric matrix whose upper triangular entries are independent Gaussian variables with the  following distribution
% Roughly, when $p$ grows to infinity in a relatively slow speed compared to $n$ (the precise asymptotic regime will be given in the sequel), we may expect that $(W-n)/\sqrt{n}\to \tW$ in distribution, where $\tW=(\tilde{w}_{ij})_{1\leq i,j\leq p}$ is a Wigner matrix---a symmetric matrix,  ---
\begin{equation}\label{eq:dist-wig}
\tilde{w}_{ij}\sim
\begin{cases}
  N(0,2) \text{ if } i=j;\\
  N(0,1) \text{ if } i <j.
\end{cases}
\end{equation}
For $S\subset\{1,...,p\}$, set $\tilde{W}_S=(\tilde{w}_{ij})_{i,j\in S}$.
We will work on the corresponding statistics
\beq\label{eq:tt-stat}
\tilde{T}_{m,p}= \max_{S\subset\{1,...,p\}, |S|=m}\lambda_1(\tilde{W}_S)
\eeq
 and
 \beq\label{eq:tv-stat}
 \tilde{V}_{m,p}=  \min_{S\subset\{1,...,p\}, |S|=m}\lambda_m(\tilde{W}_S).
 \eeq
In this paper, we study asymptotic results regarding the four statistics  $T_{m,n,p}$, ${V}_{m,n,p}$, $\tilde{V}_{m,p}$ and $\tT$.
%%%%%%%%%%%%%%%%%%%%%
\section{Main results}
\label{sec:main-results}
%%%%%%%%%%%%%%%%%%%%%
Throughout the paper, we will let $n\to\infty$ and let $p=p_n\to\infty$ with a speed depending on $n$. The following technical assumptions will be used in our main results.

\medskip\noindent
{\bf Assumption 1.}   The integer $m\geq 2$ is fixed and $\log p=o(n^{1/2})$; or $m\to\infty$   with
\bea\lbl{nice_birth_1}
m = o\left(\min\left\{ \frac{(\log p)^{1/3} }{\log\log p },\; \frac{n^{1/4}}{(\log n)^{3/2}(\log p)^{1/2}}\right\}\right).
 % \left({n \over (\log p) (\log n)^6}\right)^{1\over 5}\right\}\right).
\eea
Notice the second part of Assumption 1 implies that $\log p=o(n^{1/2}(\log n^{-3}))$. It says the population dimension $p$ can be very large and it can be as large as $\exp\{o(n^{1/2}/\log n^{3})\}$.  This assumption is used in the analysis of $T_{m,n,p}$ and ${V}_{m,n,p}$.  The requirement  $m = o((\log p)^{1/3}/\log\log p)$ is used in the last step in \eqref{eq:tq-rough-lower}. The second part of the condition $m = o(n^{1/4}(\log n)^{-3/2}(\log p)^{-1/2})$ is needed in a few places including \eqref{through_water}. The key scales $(\log p)^{1/3}$ and $n^{1/4}$ in condition \eqref{nice_birth_1} are tight, the terms of lower order $\log\log p$ and  $(\log n)^{3/2}$ can be improved to be relatively smaller.

The next assumption is needed for studying the properties of $\tilde{V}_{m,p}$ and $\tT$.

\medskip\noindent
{\bf Assumption 2.}  The integer $m$ satisfies that
\bea\lbl{nice_birth_2}
\text{$m\geq 2$ is fixed, or $m\to\infty$ with $m=o\Big(\frac{(\log p)^{1/3}}{\log\log p}\Big)$}.
\eea
%{\color{red} changed from $o(\sqrt{\log p/\log\log p})$}
This condition is the same as the first part of \eqref{nice_birth_1}. We start with asymptotic results for $T_{m,n,p}$ in \eqref{eq:t-stat} and ${V}_{m,n,p}$ in \eqref{eq:v-stat}.
\begin{theorem}\label{thm:wishart}
Suppose Assumption 1 in \eqref{nice_birth_1} holds. Recall $T_{m,n,p}$ defined as in \eqref{eq:t-stat}. Then,
\begin{equation}
	Z_n:=\frac{T_{m,n,p}-n}{\sqrt{n}}-2\sqrt{m\log p} \to 0
\end{equation}
in probability as $n\to\infty.$
Furthermore,
\begin{equation}\label{eq:mgf-thm-wish}
  	 \lim_{n\to\infty}\E\left[e^{\alpha|Z_n|} \mathbf{1}_{\{
  |Z_n|\geq \delta
  \}}\right]=0
  \end{equation}
  for all $\alpha>0$ and $\delta>0$.
\end{theorem}

\begin{remark}\label{thm:wishart-min}
Suppose Assumption 1 in \eqref{nice_birth_1} holds. Recall ${V}_{m,n,p}$ defined as in \eqref{eq:v-stat}.
Similar to the proof of Theorem \ref{thm:wishart} it can be shown that
\begin{equation}
	Z'_n:=\frac{{V}_{m,n,p}-n}{\sqrt{n}}+2\sqrt{m\log p}\to 0
%\text{ in probability as }n \to\infty.
\end{equation}
in probability as $n \to\infty$, and  furthermore,
  \begin{equation}\label{eq:mgf-thm-wish-min}
  	 \lim_{n\to\infty}\E\left[e^{\alpha|Z_n'|} \mathbf{1}_{\{
  |Z_n'|\geq \delta
  \}}\right]=0
  \end{equation}
  for all $\alpha>0$ and $\delta>0$. For reasons of space, we omit the details here.
\end{remark}

We now turn to  the asymptotic analysis for $\tilde{T}_{m,p}$ and $\tilde{V}_{m,p}$.
\begin{theorem}\label{thm:wigner}
Suppose  Assumption 2 in \eqref{nice_birth_2} is satisfied. Recall $\tilde{T}_{m,p}$ defined as in \eqref{eq:tt-stat}.
Then,
\begin{equation}
	\tilde{Z}_p:=\tilde{T}_{m,p}-2\sqrt{m\log p}\to 0
\end{equation}
in probability as $n\to\infty.$
Furthermore,
  \begin{equation}\label{eq:mgf-thm-wig}
    	\lim_{p\to\infty}\E\left[
   e^{\alpha |\tilde{Z}_p|}\mathbf{1}_{\{|\tilde{Z}_p|\geq \delta\} }
  \right]=0
    \end{equation}
for all $\alpha> 0$ and $\delta>0$.
\end{theorem}

\begin{remark}\label{thm:wigner-min}
Suppose  Assumption 2 in \eqref{nice_birth_2} is satisfied.  Review $\tW=(\tilde{w}_{ij})_{1\leq i,j\leq p}$ above \eqref{eq:dist-wig}, we know $\tW$ and $-\tW$ have the same distribution. Let $\tilde{V}_{m,p}$ be defined as in \eqref{eq:tv-stat}. It follows that $-\tilde{T}_{m,p}$ and  $\tilde{V}_{m,p}$ have the same distribution. Then, by Theorem \ref{thm:wigner},
\begin{equation}
	\tilde{Z}'_p:=\tilde{V}_{m,p}+2\sqrt{m\log p}\to 0
\end{equation}
in probability as  $n\to\infty.$ Furthermore,
  \begin{equation}\label{eq:mgf-thm-wig-min}
  	\lim_{p\to\infty}\E\left[
  e^{\alpha|\tilde{Z}'_p|}\mathbf{1}_{\{|\tilde{Z}'_p|\geq \delta\} }
  \right]=0
  \end{equation}
  for all $\alpha> 0$ and $\delta>0$.
\end{remark}

To better explain the convergence results in the \eqref{eq:mgf-thm-wish} -- \eqref{eq:mgf-thm-wig-min}, we give the following comments.
\begin{remark}\label{remark:implications}
Equation \eqref{eq:mgf-thm-wish} has the following implications, whose rigorous justification is given in Section~\ref{sec:proof}.
	\begin{enumerate}
		\item  $\lim\limits_{n\to\infty}\E\left[
   e^{\alpha |Z_n|}
  \right]=1$
  for all $\alpha>0$;
  \item  $
			\lim\limits_{n\to\infty}\E(
   |Z_n|^{\alpha}  )=0
			$
			for all $\alpha>0$;
			\item $\lim\limits_{n\to\infty }{\rm Var}(Z_n)=0$.
	\end{enumerate}
\end{remark}
We now elaborate on the above results. First, the moment generating function of $|Z_n|$ exists and is close to $1$ when $n$ is large. As a result, $|Z_n|$ has a sub-exponential tail probability for large $n$. Second, $Z_n$ converges to $0$ in $L_q$ for all $q>0$. Third, the variance of $Z_n$ vanishes for large $n$, indicating that ${\rm Var}(T_{m,n,p})=o(n)$ as $n\to\infty$. Overall, we can see \eqref{eq:mgf-thm-wish} is stronger than the typical convergence in probability. This provides information on the behavior of the  tail probability. Similar interpretations can also be made for \eqref{eq:mgf-thm-wish-min}, \eqref{eq:mgf-thm-wig} and \eqref{eq:mgf-thm-wig-min}, respectively.

\subsection{Extensions}
In this section, we discuss extensions of Theorems \ref{thm:wishart} and \ref{thm:wigner}. Similar extensions can also be made to Remarks \ref{thm:wishart-min} and \ref{thm:wigner-min}. They are omitted for the clarity of presentation.

First, we point out that Theorems \ref{thm:wishart} and \ref{thm:wigner}  still hold if we replace the size-$m$ principal minors by the principal minors with the size no larger than $m$ in the definition of $\tilde{T}_{m,p}$ and $T_{m,n,p}$, by the eigenvalue interlacing theorem [see, e.g., \cite{horn2012matrix}]. We then have the following corollary.
\begin{corollary}\label{coro:smaller-size}
\sloppy Define $ \hat{T}_{m,n,p} = \max_{S\subset\{1,...,p\}, |S|\leq m}\lambda_1(W_S)$ and $\hat{T}_{m,p}= \max_{S\subset\{1,...,p\}, |S|\leq m}\lambda_1(\tilde{W}_S)$. Then, Theorems \ref{thm:wishart} and \ref{thm:wigner} still hold if ``$T_{m,n,p}$" and ``$\tilde{T}_{m,p}$" are replaced by ``$\hat{T}_{m,n,p}$" and ``$\hat{T}_{m,p}$", respectively.
\end{corollary}

Next, we extend Theorem~\ref{thm:wigner} to allow other values of variance for the Wigner matrix. Here, we assume that the matrix $\tilde{W}$ to have the following distribution, instead of that in \eqref{eq:dist-wig}. For some $\eta\geq 0$,
\begin{equation}\label{eq:dist-general-wig}
  \tilde{w}_{ij}\sim
  \begin{cases}
    N(0,\eta) \text{ if } i=j;\\
    N(0,1) \text{ if } i<j.
  \end{cases}
\end{equation}
In addition, assume that $\tilde{W}$ is symmetric and $\tilde{w}_{ij}$ are independent for $i\leq j$. Note that if $\eta=2$, then the above distribution is the same as that defined  in \eqref{eq:dist-wig}. For $\tilde{W}$ defined in \eqref{eq:dist-general-wig}, we consider the statistic $\tilde{T}_{m,p}$. The following law of large numbers is obtained.
\begin{theorem}\label{thm:general-wig}
  Suppose $p\to\infty$ and that Assumption 2 in \eqref{nice_birth_2} is satisfied. In addition, assume $\tilde{W}$ has the distribution as in \eqref{eq:dist-general-wig} with $0\leq \eta\leq 2$.
  Then,
  $$
  \frac{\tilde{T}_{m,p}}{\sqrt{[4(m-1)+2\eta]\log p}}\to 1
%  \text{ in probability}.
  $$
in probability as $n\to\infty$.
% {\red Note that $m=o(\sqrt{\log p/\log\log p})$, so the term inside the squared root in the denominator is negative when $\eta<2$. Do we need $\eta\ge 2$?}
\end{theorem}

\begin{remark}{\rm
A related open question is whether Theorem~\ref{thm:wishart} can be extended to other distribution of $x_{ij}$ for the Wishart distribution. We conjecture that with certain assumptions on the moments of $x_{ij}$ and under the asymptotic regime that $n$ is sufficiently large compared to $\log p$ and $m$, and $\frac{{\rm Var}(x_{11}^2)}{{\rm Var}(x_{11}x_{12})}\leq 2$, the asymptotic behavior of $\frac{T_{m,n,p}-n}{\sqrt{n}}$ will be similar to that of $\tilde{T}_{m,p}$ as is discussed in Theorem~\ref{thm:general-wig}. We leave this question for future research, because it requires development of some technical tools that are beyond the scope of the current paper.

Some special cases for this question have been answered in the literature for Wishart matrices with non-Gaussian entries. For example, if $m=2$, and $x_{ij}$ follows an asymmetric Rademacher distribution   $\PP(x_{ij}=1)=p$ and $\PP(x_{ij}=-1)=1-p$, then it is easy to check
\beaa
W_{\{i,j\}}=\begin{pmatrix}
	n & \sum_{k=1}^n x_{ki}x_{kj}\\
	\sum_{k=1}^n x_{ki}x_{kj} & n
\end{pmatrix}
\eeaa
and $\lambda_1(W_{[i,j]})= n+ |\sum_{k=1}^n x_{ki}x_{kj}|$. As a result, $T_{m,n,p}=\max_{1\leq i<j\leq p}\lambda_1(W_{[i,j]}])=n+\max_{1\leq i<j\leq p}|\sum_{k=1}^n x_{ki}x_{kj}|$. Analysis on similar quantities has been studied extensively in the literature including
\cite{jiang2004asymptotic,cai2012phase,zhou2007asymptotic,shao2014necessary,li2012jiang,
li2006some,li2010necessary,fan2018discoveries,cai2013distributions}.} The limiting distributions of $T_{m,n,p}$ are the Gumbel distribution.
\end{remark}

\subsection{Application to Construction of Compressed Sensing Matrices}
\label{sec:CS}
%%%%%%%%%%%%%%%%%%%%%%%%%%%%%%%%%%%%%%%%

The main results given above have direct implications for the construction of compressed sensing  matrix $X_{n\times p}$ whose entries are  i.i.d. $N(0, {1\over n})$. As discussed in the introduction, the goal is to construct the measurement matrix $X$ with the number of measurements $n$ as small as possible relative to $p$, such that $k$-sparse signals $\beta$ can be accurately recovered.
For any given $t$, the RIP framework guarantees accurate recover of all $k$-sparse signals $\beta$ if the extreme eigenvalues, $\lambda_{\max}(tk)$ and $\lambda_{\min}(tk)$, of the $tk\times tk$ principal minors of the  Wishart matrix $W={X}^{\intercal}{X}$ satisfy
\beq
\label{rip.bound}
1-b_*(t) <  \lambda_{\min}(tk)\le \lambda_{\max}(tk) <  1+b_*(t)
\eeq
where $b_*(t)$ is given in \eqref{bt}.

By setting $m=tk$, $\lambda_{\max}(tk) = T_{m,n,p}/n$, and $\lambda_{\min}(tk) = {V}_{m,n,p}/n$,  it follows from Theorems \ref{thm:wishart} and Remark \ref{thm:wishart-min} that, under Assumption 1 in \eqref{nice_birth_1},
\beaa
\lambda_{\max}(tk) = 1 + 2\sqrt{tk \log p\over n} (1+ o_p(1))
\eeaa
and
\beaa
\lambda_{\min}(tk) = 1 - 2\sqrt{tk \log p\over n} (1+ o_p(1)).
\eeaa
On the other hand,  Assumption 1 implies that $\sqrt{\frac{m\log p}{n}}=\sqrt{\frac{tk\log p}{n}}=o(1)$. So the above asymptotic approximation gives $\lambda_{\max}(tk)=1+o_p(1)$ and $\lambda_{\min}(tk)=1+o_p(1)$, and hence \eqref{rip.bound} is satisfied. That is, Assumption 1 guarantees the exact recovery of all $k$ sparse signals in the noiseless case through the constrained $\ell_1$ minimization as explained in (\ref{RIP}) and (\ref{bt}).

\section{Technical Proofs}
\label{sec:proof}
%%%%%%%%%%%%%%%%%%%%%

Throughout the proof, as mentioned earlier, we will let $n\to\infty$ and $p=p_n\to \infty$; the integer $m\geq 2$ is either fixed or $m=m_n\to\infty$. The following notation will be adopted. We write $a_n=O(b_n)$ if there is a constant $\kappa$  {independent of} $n,p$ and $m$ (unless otherwise indicated) such that $|a_n|\leq \kappa b_n$. Moreover, we write $a_n=o(b_n)$, if there is a sequence $c_n$ independent of $n,p$ and $m$ such that $c_n\to 0$ and $|a_n|\leq c_n b_n$. Define $\sinf=\log\log\log p$. This is a sequence growing to infinity with a very slow speed compared to $n$ and $p$.
 % for a sequence growing to infinity in a very slow speed.
% {\red Xiaoou, do you mean to say ``Define $\sinf=\log\log\log p$" for the underlined part?}

This section is organized as follows. We first introduce the main steps in proving Theorems \ref{thm:wishart} and \ref{thm:wigner} in Section \ref{sec:proof-strat}.  In Section~\ref{sec:proof-thm}, we present the proofs for Theorems \ref{thm:wishart}-\ref{thm:general-wig}, Corollary~\ref{coro:smaller-size}, and Remark~\ref{remark:implications}. The proofs for all technical lemmas are given in the Appendix. For reader's convenience, we list the content of each section below.

\medskip

\noindent {\bf Section \ref{sec:proof-strat}}. The Strategy of the Proofs for Theorems \ref{thm:wishart} and \ref{thm:wigner}.

\noindent {\bf Section \ref{sec:proof-thm}}. Proof of the results in Section~\ref{sec:main-results}.

 {\bf Section \ref{sec:proof-week}}. Proof of Theorem~\ref{thm:wigner}.

 {\bf Section \ref{Sec_thm:wishart}}. Proof of Theorem~\ref{thm:wishart}.

 {\bf Section \ref{Sec_coro:smaller}}. Proofs of  Theorem \ref{thm:general-wig} and Remark~\ref{remark:implications}.
 %Proofs of Theorems~\ref{thm:wigner-min} and \ref{thm:general-wig}, Corollary~\ref{coro:smaller-size} and Remark~\ref{remark:implications}.

\subsection{The Strategy of the Proofs for Theorems \ref{thm:wishart} and \ref{thm:wigner}}\label{sec:proof-strat}
We first explain the proof strategy for Theorem~\ref{thm:wigner} and then explain that for Theorem~\ref{thm:wishart}, since Wigner matrices have simpler structure than Wishart matrices.
The proof of Theorem~\ref{thm:wigner} consists of three steps. The first step is to find an upper bound on the right tail probability $\PP(\tilde{T}_{m,p}\geq 2\sqrt{m\log p}+t)$ for $t\geq \delta$. Our method here is to first develop a moderate deviation bound of $\PP(\lambda_1(\tW_{S})\geq 2\sqrt{m\log p}+t)$ for each $S\subset \{1,...,p\}$ and $|S|=m$, and then use the union bound to control $\PP(\tilde{T}_{m,p}\geq 2\sqrt{m\log p}+t)$. The second step is to find an upper bound on the left tail probability $\PP(\tilde{T}_{m,p}\leq 2\sqrt{m\log p}-t)$ for $t\geq \delta$. Our approach is to construct a sequence of events $E_{p,m}$ with high probability, such that when $E_{p,m}$ occurs, there exists $S\subset\{1,...,p\}$ satisfying $|S|=m$ and $\lambda_1(\tW_{S})\geq 2\sqrt{m\log p}-t$. The third step is to combine the left and right tail bounds obtained from the previous two steps to show \eqref{eq:mgf-thm-wig}.

The proof of Theorem~\ref{thm:wishart} is based on a similar strategy to that of Theorem~\ref{thm:wigner}. A new and key ingredient is to control the approximation speed of the Wishart matrix to the Wigner matrix (after normalization). Change-of-measure arguments are used to quantify the approximation speed in the  moderate deviation domain.

We point out that the proof for the asymptotic lower bound of $\tilde{T}_{m,p}$ in this paper is different from the standard technique for analyzing the maximum/minimum statistic for a large random matrix (see, e.g. \cite{jiang2004asymptotic}). In particular, the proof in \cite{jiang2004asymptotic}  employs the Chen-Stein's Poisson approximation method [see, e.g., \cite{Arratia:1990ff}] and the asymptotic independence. However, this method does not fit our problem. For this reason, new technique are developed and, in particular, we {\em construct} an event on which $\tilde{T}_{m,p}$ achieves the asymptotic lower bound.

\subsection{Proof of the results in Section~\ref{sec:main-results}}\label{sec:proof-thm}
As mentioned earlier, Wigner matrices have simpler structure than Wishart matrices. Thus, we first present the proof of Theorem~\ref{thm:wigner}, followed by the proof of Theorem~\ref{thm:wishart}.
% Then, we give the proof of Theorem~\ref{thm:wigner-min}, where it is observed that the asymptotic behavior of $\tilde{V}_{m,p}$ is similar to that of $-\tilde{T}_{m,p}$. The proof of Theorem~\ref{thm:wishart-min} follows a similar line of reasoning, and is thus  omitted.
At the end of the section, the proofs of Corollary~\ref{coro:smaller-size}, Theorem~\ref{thm:general-wig} and Remark~\ref{remark:implications} are presented.

In each proof we will need auxiliary results. To make the proof clearer, we place the proofs of the auxiliary results in the Appendix. Sometimes a statement or a formula holds as $n$ is sufficiently large. We will not say ``as $n$ is sufficiently  large" if the context is apparent.

\subsubsection{Proof of Theorem~\ref{thm:wigner}}\label{sec:proof-week}

To prove Theorem~\ref{thm:wigner}, we need the following two key results.
\begin{proposition}\label{east_wind} Suppose  Assumption 2 in \eqref{nice_birth_2} is satisfied. Recall $\tilde{T}_{m,p}$ defined as in \eqref{eq:tt-stat}.
Then,
\begin{equation}\label{eq:final-upper}
	\lim_{p\to\infty}\sup_{t\geq\delta}e^{\alpha t}t^2 \PP\left(\tilde{T}_{m,p}\geq 2\sqrt{m\log p}+t\right) = 0
\end{equation}
for every $\alpha>0$ and every $\delta>0$.
\end{proposition}

\begin{proposition}\label{west_wind} Suppose  Assumption 2 in \eqref{nice_birth_2} is satisfied. Recall $\tilde{T}_{m,p}$ defined as in \eqref{eq:tt-stat}.
Then,
\begin{equation}\label{eq:final-upper}
	\lim_{p\to\infty}\sup_{t\geq\delta}e^{\alpha t}t^2 \PP\left(\tilde{T}_{m,p}\leq 2\sqrt{m\log p}-t\right) = 0
\end{equation}
for every  $\alpha>0$ and every $\delta>0$.
\end{proposition}

Another auxiliary lemma is need. Its proof is put in the Appendix.
\begin{lemma}\label{lemma:integral}
Let $Z\geq 0$ be a random variable with $\E[e^{\alpha Z}]<\infty$ for all $\alpha>0$. Then
%  For a non-negative random variable $Z$ and positive numbers $\alpha,\delta$, we have
  \begin{equation}
    \E\big[e^{\alpha Z}\mathbf{1}_{\{Z\geq \delta\}}\big]
    = e^{\alpha\delta}\PP(Z\geq \delta)+\alpha \int_{\delta}^{\infty} e^{\alpha t}\PP(Z>t)dt
  \end{equation}
for every  $\alpha>0$ and every $\delta>0.$
\end{lemma}

\begin{proof}[Proof of Theorem~\ref{thm:wigner}]
%We will need a few of lemmas in the process. To make the whole proof more concise, we place their proofs in the Appendix. Also, sometimes  a statement or a formula holds only as $n$ is sufficiently large. We will not say ``as $n$ is sufficiently  large" if the context is apparent.
%
%Notice that \eqref{eq:mgf-thm-wig} implies $\tilde{T}_{m,p}-2\sqrt{m\log p}\to 0$ in probability (see Remark~\ref{remark:implications} appeared below the statement of Theorem \ref{thm:wigner-min}, whose proof is independent of the current one). Thus, it is sufficient to show \eqref{eq:mgf-thm-wig}. Review the notation $\sinf=\log\log\log p$.
%
%The proof of the \eqref{eq:mgf-thm-wig} consists of three steps. First, an upper bound for $\PP(\tilde{T}_{m,p}\geq 2\sqrt{m\log p}+t)$ will be derived as follows.
%\begin{equation}\label{eq:final-upper}
%	\lim_{p\to\infty}\sup_{t\geq\delta}e^{\alpha t}t^2 \PP\left(\tilde{T}_{m,p}\geq 2\sqrt{m\log p}+t\right) = 0
%\end{equation}
%for every $\delta>0$ and $\alpha>0$.
%Second, we will develop an  upper bound for $\PP(\tilde{T}_{m,p}\leq 2\sqrt{m\log p}-t)$ by considering the following two cases.
%\begin{equation}\label{eq:final-case1}
%	\lim_{p\to\infty}\sup_{\delta\leq t \leq 2 \sqrt{m\log p}- m\sinf}e^{\alpha t}t^2\PP\left(\tilde{T}_{m,p}\leq 2\sqrt{m\log p}-t\right) = 0
%\end{equation}
%for $\alpha>0$ and $\delta>0$, and
%\begin{equation}\label{sweet_home}
%		\lim_{p\to\infty}\sup_{t\geq2\sqrt{m\log p}-m\sinf }e^{\alpha t}t^2\PP\left(\tilde{T}_{m,p}\leq 2\sqrt{m\log p}-t\right)=0.
%\end{equation}
%With the three limits above
By Propositions \ref{east_wind} and \ref{west_wind}, we have
%This together with \eqref{eq:final-upper} and \eqref{eq:final-case1} concludes
\begin{equation}\label{eq:decay-rate-wig-final}
			\lim_{p\to\infty}\sup_{t\geq\delta }e^{\alpha t}t^2\PP\left(\big|\tilde{T}_{m,p}- 2\sqrt{m\log p}\big|\geq \delta\right)=0
\end{equation}
for any $\alpha >0$ and $\delta>0$.
% Summarizing all the cases and scenarios, we have
% $\PP(\tilde{T}_{m,p}\leq 2\sqrt{m\log p}-t)$ is upper bounded by $\exp\{- 2\sqrt{m\log p} t + O(m^2\log\log p)\}$ if $0<t\leq 2\sqrt{m\log p}- m\sinf$; by $\exp(-1/4\log p)$ if $0<t\leq 2\sqrt{m\log p}- m\sinf$; and by $\exp\{-\frac{pt^2}{16} \}$ if $t>4\sqrt{m\log p}.$ In all of these cases, we can see that
% \begin{equation}
%   \lim_{p\to\infty}\sup_{t\geq m^{1.5}(\log p)^{-0.5}\log\log p \sinf }e^{\alpha t}t^2\PP(\tilde{T}_{m,p}\leq 2\sqrt{m\log p}-t)=0.
% \end{equation}
% In particular, under Assumption 2, $m^{1.5}(\log p)^{-0.5}\log\log p \sinf\to 0$, so we have
% \begin{equation}
%   \lim_{p\to\infty}\sup_{t\geq \delta}e^{\alpha t}t^2\PP(\tilde{T}_{m,p}\leq 2\sqrt{m\log p}-t)=0.
% \end{equation}
% In addition, according to \eqref{eq:tp-upper}, we can also see that
% \begin{equation}
%   \lim_{p\to\infty}\sup_{t\geq \delta}e^{\alpha t}t^2\PP(\tilde{T}_{m,p}\leq 2\sqrt{m\log p}+t)=0.
% \end{equation}
% Combining the above two equations, we have
% \begin{equation}\label{eq:decay-rate-wig-final}
%   \lim_{p\to\infty}\sup_{t\geq \delta}e^{\alpha t}t^2\PP(|\tilde{T}_{m,p}- 2\sqrt{m\log p}|\geq t)=0.
% \end{equation}
Consequently, for given $\alpha>0$, there exists a sequence of positive numbers  $a_p\to 0$ such that
\begin{equation}
  e^{\alpha t}t^2\PP\left(\big|\tilde{T}_{m,p}- 2\sqrt{m\log p}\big|\geq t\right)\leq a_p
\end{equation}
for all $t\geq \delta$ as $p$ is sufficiently large. Now we estimate
\beaa
\E(e^{\alpha|\tilde{T}-2\sqrt{m\log p}|}\mathbf{1}_{\{|\tilde{T}-2\sqrt{m\log p}|\geq \delta\} }).
\eeaa
%through the following identity, which will be shown in the Appendix.
%\begin{lemma}\label{lemma:integral}
%Let $Z\geq 0$ be a random variable with $\E[e^{\alpha Z}]<\infty$ for all $\alpha>0$. Then
%%  For a non-negative random variable $Z$ and positive numbers $\alpha,\delta$, we have
%  \begin{equation}
%    \E\big[e^{\alpha Z}\mathbf{1}_{\{Z\geq \delta\}}\big]
%    = e^{\alpha\delta}\PP(Z\geq \delta)+\alpha \int_{\delta}^{\infty} e^{\alpha t}\PP(Z>t)dt
%  \end{equation}
%for any  $\alpha>0$ and $\delta>0.$
%\end{lemma}
By applying Lemma \ref{lemma:integral} to $Z_{p,m}=|\tilde{T}-2\sqrt{m\log p}|$, we see
\begin{equation}
\begin{split}
  &\E\left[
  e^{\alpha |\tilde{T}_{m,p}-2\sqrt{m\log p}|}\mathbf{1}_{|\tilde{T}_{m,p}-2\sqrt{m\log p}|\geq \delta}
  \right]\\
	= & \E\left[
  e^{\alpha Z_{p,m}}\mathbf{1}_{\{Z_{p,m}\geq \delta\} }
  \right]\\
  =& e^{\alpha\delta}\PP(Z_{p,m}\geq \delta)
  + \int_{\delta}^{\infty}
  e^{\alpha t}\PP(Z_{p,m}\geq t)dt.
\end{split}
\end{equation}
According to \eqref{eq:decay-rate-wig-final}, the above display can be bounded from above by
\begin{equation}
%  \E\left[
%  e^{\alpha |\tilde{T}_{m,p}-2\sqrt{m\log p}|}\mathbf{1}_{|\tilde{T}_{m,p}-2\sqrt{m\log p}|\geq \delta}
%  \right]
%  \leq
  \delta^{-2}a_p + a_p\int_{\delta}^{\infty}t^{-2}dt,
\end{equation}
which tends to $0$ as $p\to\infty$. The proof is then complete.
\end{proof}

\medskip

Now we proceed to prove Propositions \ref{east_wind} and \ref{west_wind}.

\begin{proof}[Proof of Proposition \ref{east_wind}]
%\begin{proof}The proof of (\ref{eq:final-upper})}.
%We start with an upper bound for $\PP(\tilde{T}_{m,p}\geq 2\sqrt{m\log p}+t)$.
For any $t>0$, we have from the definition of $\tT$ that
\begin{equation}
\begin{split}
  &\PP\left(
\tT\geq 2\sqrt{m\log p} + t
  \right)\\
  =&\PP\Big(
\max_{S\subset\{1,...,p\},|S|=m}\lambda_1(\tW_S)\geq 2\sqrt{m\log p}+t
  \Big)\\
  =& \PP\Big(
\bigcup_{S\subset\{1,...,p\},|S|=m}\big\{\lambda_1(\tW_S)\geq2\sqrt{m\log p} + t\big\}
  \Big)\\
  \leq &\sum_{S\subset\{1,...,p\},|S|=m}
  \PP\Big(\lambda_1(\tW_S)\geq 2\sqrt{m\log p} +t \Big)\\
  \leq & p^m \PP\Big(\lambda_1(\tW_{\{1,...,m\}})\geq2\sqrt{m\log p} + t\Big),\label{eq:upper-split}
\end{split}
\end{equation}
where in the last inequality we use the fact that $W_S$ are identically distributed for all different $S$ with $|S|=m$. The following result  enables us to bound the last probability.
\begin{lemma}\label{eq:wig-margin-tail}
%[Equation (4.8) from \cite{Jiang:2015ir}]\label{lemma:wigner-margin-tail}
Let $\tW_{\{1,...,m\}}$ be defined as above \eqref{eq:tt-stat} with $S=\{1,...,m\}$. Then there is a constant $\kappa>0$ such that
\begin{equation}
	P\Big(\lambda_1(W_{\{1,...,m\}})\geq x \text{ or } \lambda_{m}(W_{\{1,...,m\}})\leq -x\Big)
\leq  e^{-(x^2/4)+\kappa m\log x }	
%\leq  \exp\Big\{-\frac{x^2}{4}+\kappa m\log x  \Big\}
%	 P\Big(\lambda_1(\tW_{\{1,...,m\}})\geq x \text{ or } \lambda_{m}(\tW_{\{1,...,m\}})\leq -x\Big)
%	 \leq \kappa\cdot e^{-\frac{x^2}{4} + \kappa \sqrt{m}x}
\end{equation}
for all $x>4\sqrt{m}$ and all $m\geq 2$.
\end{lemma}
%  lemma, which is the first line of equation (4.8) from \cite{Jiang:2015ir} gives an upper bound of the probability on the last line of the above display.
%We remark that the inequality \eqref{eq:wig-margin-tail} takes a slightly different form from equation (4.8) from \cite{Jiang:2015ir}, but they are essentially the same. The reason for the difference is due to the different scaling of the Wigner matrix  considered in our paper and that in \cite{Jiang:2015ir}.
Taking $x:=2\sqrt{m\log p}+t$ in the above lemma, we know $x>4\sqrt{m}$ as $n$ is large enough, and hence
\begin{equation}
\begin{split}
    &\log\left[e^{\alpha t}p^m \PP\left(\lambda_1(\tW_{\{1,...,m\}})\geq2\sqrt{m\log p} + t\right)\right]\\
  \leq& \alpha t + m\log p - \frac{1}{4}\left(2\sqrt{m\log p}+t\right)^2+\kappa{m}\log\left(2\sqrt{m\log p}+t\right)\\
  = & \alpha t -t\sqrt{m\log p} -\frac{1}{4}t^2\\
  & ~~~  + \kappa m\log\left(2\sqrt{m\log p}\right)+
  \kappa m\log\left(1+\frac{t}{2\sqrt{m\log p}}\right).
\end{split}
\end{equation}
%  \leq & - (\sqrt{m\log p} -\alpha - \frac{1}{2\sqrt{\log p}})t + O(\sqrt{m}\log\log p).
Note that $-\frac{1}{4}t^2\leq 0$, $ \kappa m\log(2\sqrt{m\log p}) = O({m}\log\log p) $, and $\kappa m\log(1+\frac{t}{2\sqrt{m\log p}})=O(\frac{\sqrt{m}}{\sqrt{\log p}}t)< t$ as $p$ is sufficiently large. Thus, the above inequality further implies
\bea\label{eq:wig-tail-p}
%	\begin{split}
& &    \log\left[e^{\alpha t}p^m \PP\left(\lambda_1(\tW_{\{1,...,m\}})\geq2\sqrt{m\log p} + t\right)\right]\nonumber\\
& \leq & - \left(\sqrt{m\log p} -\alpha - 1\right)t + O\left(m\log\log p\right)\nonumber\\
%\end{split}
%\end{equation}
%%To obtain the last inequality in the above display, we used $\log m = O(\log\log p)$ under Assumption 2.
%Note that $\frac{1}{\sqrt{\log p}}=o(1)$ and $\alpha$ is fixed, we further have
%\begin{equation}\label{eq:tp-upper}
%%\begin{split}
%  & \log\left(e^{\alpha t}\PP\left(\tilde{T}_{m,p}\geq 2\sqrt{m\log p}+t\right)\right)\\
%  \leq &  \log\left(e^{\alpha t}p^m \PP\left(\lambda_1(\tW_{\{1,...,m\}})\geq2\sqrt{m\log p} + t\right)\right)\\
  & \leq &  -\frac{t}{2}\sqrt{m\log p}  + O\left(m\log\log p\right)
%\end{split}
\eea
uniformly for all $t\geq 0$ as $p$ sufficiently large, where $\alpha>0$ is fixed.
With the above inequality, we complete the proof.

\end{proof}

\begin{proof}[Proof of Proposition \ref{west_wind}] Recall $\sinf=\log\log\log p$. The proof will be evidently finished if the following two limits hold. For each $\alpha>0$ and each $\delta>0$,
\begin{equation}\label{eq:final-case1}
	\lim_{p\to\infty}\sup_{\delta\leq t \leq 2 \sqrt{m\log p}- m\sinf}e^{\alpha t}t^2\PP\left(\tilde{T}_{m,p}\leq 2\sqrt{m\log p}-t\right) = 0
\end{equation}
 and
\begin{equation}\label{sweet_home}
		\lim_{p\to\infty}\sup_{t\geq2\sqrt{m\log p}-m\sinf }e^{\alpha t}t^2\PP\left(\tilde{T}_{m,p}\leq 2\sqrt{m\log p}-t\right)=0.
\end{equation}
We now verify the above two limits.

\medskip

\noindent{\bf The proof of (\ref{eq:final-case1})}. Recall
%\sloppy We proceed to an upper bound on the left tail probability $\PP(\tilde{T}_{m,p}\leq 2\sqrt{m\log p}-t)$. We discuss two cases: (1) $0<t\leq 2\sqrt{m\log p}- m\sinf$, and (2) $ t>2\sqrt{m\log p}- m\sinf$, where $\sinf=\log\log\log p$. We will use different proof strategies for the two cases.
%{\bf Case 1: $0<t\leq 2\sqrt{m\log p}- m\sinf$.}
%In this step, we assume $0<t\leq 2\sqrt{m\log p}- m\sinf$. Recall
%For this case, our proof is based on the following inequality that
 \begin{equation}\label{eq:lb-spec}
  \lambda_1(A)\geq \frac{1}{k}\sum_{i=1}^k\sum_{j=1}^k a_{ij}
 \end{equation}
 for any $k\times k$ square and symmetric matrix $A=(a_{ij})_{1\leq i,j\leq k}$, where $\lambda_1(A)$ is the largest eigenvalue of $A.$

For each $S\subset\{1,...,p\}$ such that $|S|=m$ and $\tW_{S}=(\tW_{ij})$, set
\begin{equation}\label{eq:tas}
\tA_{S}=\left\{  \tW_{ij}\geq (1-
\varepsilon_{m,p,t})\sqrt{\frac{4\log p}{m}} \text{ for all } i,j\in S \text{ and } i\leq j
\right\},
\end{equation}
where
\begin{equation}\lbl{talk_hire}
\varepsilon_{m,p,t}:=(4m\log p)^{-1/2}t.
\end{equation}
If $0<t\leq 2\sqrt{m\log p}- m\sinf$ then $0<\varepsilon_{m,p,t}<1$.
 According to \eqref{eq:lb-spec}, if there exists $S_0\subset\{1,...,p\}$ such that  $|S_0|=m$ and $\tA_{S_0}$ occurs, then
\begin{equation}
  \tT\geq \lambda_1(\tW_{S_0})\geq m (1-\varepsilon_{m,p,t})\sqrt{\frac{4\log p}{m}} =2\sqrt{m \log p}-t.
\end{equation}
Define
\bea\label{eq:def-q}
  \tQ_{m,p}=\sum_{S\subset\{1,...,p\}:\, |S|=m} \mathbf{1}_{\tA_S},
\eea
where $\mathbf{1}_{\tA_S}$ is the indicator function of $\tA_S$. Then,
\begin{equation}\label{eq:tt-tq}
  \PP\left(\tT< 2\sqrt{m \log p}-t\right)\leq
%\PP\left(\text{None of } \tA_S \text{ happens}\right)=
  \PP\left(\tQ_{m,p}=0\right).
\end{equation}
%Using the Chebyshev's inequality, it is straightforward to derive the following inequality:
For any random variable $Y$ with $\E Y>0$ and $\E (Y^2)<\infty$, we have
%a finite second moment and a positive mean,
\bea\label{sun_set_flower}
\PP(Y\leq 0) &\leq & \PP(Y-\E Y\leq - \E Y)\nonumber\\
&\leq & \PP( (Y-\E Y)^2\geq (\E Y)^2)\nonumber\\
& \leq & \frac{Var(Y)}{( \E Y)^2}.
\eea
Applying this inequality to $\tQ_{m,p}$, we obtain
\begin{equation}\label{eq:lower-final}
  \PP\left(\tQ_{m,p}=0\right)=\PP\left(\tQ_{m,p}\leq 0\right)\leq \frac{Var(\tQ_{m,p})}{(\E\tQ_{m,p})^2}.
\end{equation}
We proceed to find a lower bound on $\E\tQ_{m,p}$ and an upper bound on $Var(\tQ_{m,p})$ in two steps.

\noindent{\bf Step 1: the estimate of $\E\tQ_{m,p}$}. Note that $\mathbf{1}_{\tA_S}$ are identically (not independently) distributed Bernoulli variables for different $S$ with success rate $\PP(\tA_{\{1,...,m\}})$. Thus, we have
\begin{equation}\label{eq:exp-Q}
  \E\tQ_{m,p}= \binom{p}{m} \PP(\tA_{S_0}),
\end{equation}
where we choose $S_0=\{1,...,m\}$ with a bit abuse of notation. For convenience, write
\bea\lbl{aha_11}
\tau_{m,p,t}=(1-\varepsilon_{m,p,t})\sqrt{\frac{4\log p}{m}}=\sqrt{\frac{4\log p}{m}}-\frac{t}{m}.
\eea
%Consequently,
%\bea\lbl{blue_vase}
%\tau_{m,p,t}^2=\frac{4\log p}{m}\ \ \mbox{and}\ \ \log \tau_{m,p,t}=O(\log\log p)
%\eea
%uniformly over all $0<t\leq 2\sqrt{m\log p}- m\sinf$.
Since the upper triangular entries of $\tW$ are independent Gaussian variables, we have from \eqref{eq:tas} that
%the following closed form expression for $\PP(\tA_{S_0})$.
\begin{equation}
%\begin{split}
	 \PP(\tA_{S_0})
	 %=&\PP\left(
%	 \tW_{kk}\geq \tau_{m,p,t}, \tW_{ij}\geq \tau_{m,p,t} \text{ for } 1\leq i<j\leq m
%	 \right)\\
	 =\prod_{k=1}^m \PP\left(\tW_{kk}\geq \tau_{m,p,t}\right) \prod_{1\leq i<j\leq m} \PP\left(
  \tW_{ij}\geq \tau_{m,p,t}
  \right).
%\end{split}
 \end{equation}
 % where we write $\varepsilon_{m,p,t}=(4m\log p)^{-1/2}t$.
Recall that $\tW_{kk}\sim N(0,2)$ and $\tW_{ij}\sim N(0,1)$ for $i\neq j$. Hence
\begin{equation}\label{eq:pas}
  \PP(\tA_{S_0})= \bar{\Phi}\left(\frac{1}{\sqrt{2}}\tau_{m,p,t} \right)^m \bar{\Phi}\left(
\tau_{m,p,t}
\right)^{\frac{m(m-1)}{2}},
\end{equation}
where $\bar{\Phi}(z)=\int_{z}^{\infty} \frac{1}{\sqrt{2\pi}}e^{-\frac{w^2}{2}}dw$. It is well known that
%is the right-tail probability of a standard Gaussian random variable.
%We use the tail probability approximation for standard Gaussian variable,
\begin{equation}\label{eq:tail-sdn}
  \log \bar{\Phi}(x)=-  \frac{x^2}{2}-\log(x)-\log \sqrt{2\pi}+o(1)
\end{equation}
as $x\to\infty$.
Recall the assumption that $t\leq 2\sqrt{m\log p}-m\sinf$, so $\tau_{m,p,t}=\sqrt{\frac{4\log p}{m}}-\frac{t}{m}\geq \sinf\to\infty$. Thus, by \eqref{eq:pas} and \eqref{eq:tail-sdn},
%to \eqref{eq:pas} and arrive at
\begin{equation}
\begin{split}
	& \log \PP(\tA_{S_0})\\
  = & -\frac{1}{2}m \cdot\frac{1}{2}\cdot\tau_{m,p,t}^2-m\log \frac{\tau_{m,p,t}}{\sqrt{2}}\\
  &-\frac{1}{2}\cdot\frac{m(m-1)}{2}\left(\tau_{m,p,t}\right)^2 - \frac{m(m-1)}{2}\log\left(
\tau_{m,p,t}
\right) + O(m^2).
\end{split}
 % - (1-\varepsilon_{m,p,t})^2 m\log p + O(m^2(\log\log p-\log m) ).
\end{equation}
Note that $1>1-\varepsilon_{m,p,t}\geq \xi_p\sqrt{\frac{m}{4\log p}}$ since $0<t\leq 2\sqrt{m\log p}- m\sinf$. It follows that
$|\log(1-\varepsilon_{m,p,t})|=O\left(\log \sqrt{\frac{\log p}{m\xi_p^2}}\right)=O(\log\log p)$.
%under Assumption 2 and Case 1.
Also, $\log\sqrt{\frac{4\log p}{m}} =O(\log\log p)$. As a result, from (\ref{aha_11}) we have
\beaa
 \tau_{m,p,t}^2=(1-\varepsilon_{m,p,t})^2\cdot \frac{4\log p}{m}\
\ \ \mbox{and}\ \  \log\left(\tau_{m,p,t}
\right)=O(\log\log p).
\eeaa
It follows that
%the above approximation can be simplified as
\bea\label{eq:gaussian-approx}
%\begin{split}
	\log \PP(\tA_{S_0})
&=& -\frac{1}{4}m^2\left(\tau_{m,p,t}\right)^2+ O(m^2\log\log p )\notag\\
& =& - (1-\varepsilon_{m,p,t})^2 m\log p + O( m^2\log\log p ).
%\end{split}
\eea
%{\red [note that $m(\tau_{m,p,t})^2=O(\log p)$]}.
% To obtain the above approximation, we used Assumption 2 and
Combining this with \eqref{eq:exp-Q}, we see
\begin{equation}\label{eq:intermediate}
	 \log (\E\tQ_{m,p})=\log \binom{p}{m}- (1-\varepsilon_{m,p,t})^2 m\log p + O(m^2\log\log p ).
\end{equation}
To control $\binom{p}{m}$, we need the next result, which will be proved in the Appendix.
\begin{lemma}\label{lemma:p-choose-m}
For all $m\geq p\geq 1$, we have
\begin{equation}
  m\log p-m\log m\leq\log\binom{p}{m}\leq m\log p + m - m\log m.
\end{equation}
\end{lemma}
Using the above lemma, \eqref{eq:intermediate}, and note that $ m\log m=O(m^2\log\log p)$, we have
\begin{equation}\label{eq:etq-approx}
  \log (\E\tQ_{m,p})
  = \left[1- (1-\varepsilon_{m,p,t})^2\right] m\log p + O(m^2\log\log p ).
\end{equation}
%This gives an approximation for $\E(\tQ_{m,p})$.

\noindent{\bf Step 2: the estimate of $Var(\tQ_{m,p})$}. Reviewing $\tQ_{m,p}$ in (\ref{eq:def-q}), we have
\bea\label{eq:var-q1}
  Var(\tQ_{m,p}) &=& \E \tQ_{m,p}^2 -(\E\tQ_{m,p})^2\nonumber\\
   &=& \sum_{S_1,S_2\subset\{1,..,p\}, |S_1|=|S_2|=m} \PP(\tA_{S_1}\cap\tA_{S_2})- (\E\tQ_{m,p})^2.
\eea
Note that $\PP(\tA_{S_1}\cap\tA_{S_2})$ is determined by $|S_1\cap S_2|$ and $m$. By (\ref{eq:tas}),
\begin{equation}
\begin{split}
  &\sum_{S_1,S_2\subset\{1,..,p\}, |S_1|=|S_2|=m} \PP\left(\tA_{S_1}\cap\tA_{S_2}\right)\\
  =&\sum_{l=0}^m \sum_{|S_1\cap S_2|=l, |S_1|=|S_2|=m}\PP\left(\tA_{S_1}\cap\tA_{S_2}\right)\\
  =&
  \sum_{l=0}^{m} \binom{p}{l} \binom{p-l}{m-l}\binom{p-m}{ m-l  }\PP\left( \tA_{ \{1,...,m\}}\cap \tA_{ \{1,...,l,m+1,...,2m-l\}}\right)\\
  =& \sum_{l=0}^{m}\frac{p!}{l!(m-l)!(m-l)!(p-2m+l)!} \PP\left( \tA_{ \{1,...,m\}}\cap \tA_{ \{1,...,l,m+1,...,2m-l\}}\right)
\end{split}
\end{equation}
Single out the terms where $l=0$ and $l=m$, we further have
\begin{equation}\label{eq:var-q}
	\begin{split}
	&\sum_{S_1,S_2\subset\{1,..,p\}, |S_1|=|S_2|=m} \PP(\tA_{S_1}\cap\tA_{S_2})\\
  = &\frac{p!}{m!m!(p-2m)!}\PP\left( \tA_{ \{1,...,m\}}\right)^2+ \binom{p}{m}\PP\left( \tA_{ \{1,...,m\}}\right)\\
  &+ \sum_{l=1}^{m-1}\frac{p!}{l!(m-l)!(m-l)!(p-2m+l)!} \PP\left( \tA_{ \{1,...,m\}}\cap \tA_{ \{1,...,l,m+1,...,2m-l\}}\right).
	\end{split}
\end{equation}
On the other hand,  $\E\tQ_{m,p}= \binom{p}{m}P\big( \tA_{ \{1,...,m\}}\big)$ and hence
%\beaa
\begin{eqnarray}
\label{eq:var-q2}
%\begin{split}
	  (\E\tQ_{m,p})^2
	  &=& \frac{p!^2}{m!^2(p-m)!^2}P\left( \tA_{ \{1,...,m\}}\right)^2 \notag\\
 & = & \frac{p!}{m!m!(p-2m)!}P\left( \tA_{ \{1,...,m\}}\right)^2\cdot \frac{p!(p-2m)!}{(p-m)!^2}.
 \end{eqnarray}
%\end{split}
%\eeaa
Combining \eqref{eq:var-q1}, \eqref{eq:var-q} and \eqref{eq:var-q2}, we arrive at
\begin{equation}
\begin{split}
  &Var(\tQ_{m,p})\\
  =&   (\E\tQ_{m,p})^2 \left(\frac{(p-m)!^2}{p!(p-2m)!}-1\right)+\E\tQ_{m,p}\\
  &+ \sum_{l=1}^{m-1}\frac{p!}{l!(m-l)!(m-l)!(p-2m+l)!} \PP\left( \tA_{ \{1,...,m\}}\cap \tA_{ \{1,...,l,m+1,...,2m-l\}}\right).
\end{split}
\end{equation}
Observe that $\frac{p!}{(p-2m+l)!}=p(p-1)\cdots (p-2m+l-1)\leq p^{2m-l}$ and $\frac{1}{l!(m-l)!(m-l)!}\leq 1$. It follows that
\begin{equation}\label{eq:var-q-split}
  \begin{split}
  &Var(\tQ_{m,p})\\
  \leq &  \E\tQ_{m,p}+ (\E\tQ_{m,p})^2 \left(\frac{(p-m)!^2}{p!(p-2m)!}-1\right)\\
  &+ m \max_{l=1,...,m-1}p^{2m-l} \PP\left( \tA_{ \{1,...,m\}}\cap \tA_{ \{1,...,l,m+1,...,2m-l\}}\right).
\end{split}
\end{equation}
% \begin{equation}
%    E Q_p + O(p^{2m-l})\PP\left( \tA_{ \{1,...,m\}}\cap \tA_{ \{1,...,l,m+1,...,2m-l\}}\right).
% \end{equation}
Similar to \eqref{eq:pas} we have
\begin{equation}\label{eq:intersection}
\begin{split}
  &\PP\left( \tA_{ \{1,...,m\}}\cap \tA_{ \{1,...,l,m+1,...,2m-l\}}\right)\\
= &\bar{\Phi}\left(\frac{1}{\sqrt{2}}\tau_{m,p,t}\right)^{2m-l}\bar{\Phi}\left(\tau_{m,p,t}\right)^{\frac{m(m-1)}{2}\cdot 2-\frac{l(l-1)}{2}  }.
\end{split}
\end{equation}
Again, we find an approximation for the above display by using \eqref{eq:tail-sdn} and simplifying it. We arrive at
\begin{equation}
\begin{split}
    &\log \PP\left( \tA_{ \{1,...,m\}}\cap \tA_{ \{1,...,l,m+1,...,2m-l\}}\right)\\
\leq &-\frac{1}{m}(2m^2-l^2){(1-\varepsilon_{m,p,t})^2\log p}+ O(m^2\log\log p).
\end{split}
\end{equation}
Therefore, for the last term in (\ref{eq:var-q-split}), we see
\begin{equation}\label{eq:inter-upper}
  \begin{split}
& \log\left[mp^{2m-l}\PP\left( \tA_{ \{1,...,m\}}\cap \tA_{ \{1,...,l,m+1,...,2m-l\}}\right)\right]\\
\leq & \log m+(2m-l)\log p- \frac{1}{m}(2m^2-l^2) (1-\varepsilon_{m,p,t})^2{\log p}+O(m^2\log\log p)\\
=&
\left[2m-l - \frac{1}{m}(1-\varepsilon_{m,p,t})^2(2m^2-l^2) \right]\log p+O(m^2\log\log p)
.
\end{split}
\end{equation}
The following lemma enables us to evaluate the coefficient of $\log p.$
\begin{lemma}\label{lemma:max-l}
For any $0<\varepsilon<1$ and $m\geq 2$, we have
  \begin{equation}
  \begin{split}
    & \max_{l=1,...,m-1}\Big\{(2m-l)- \frac{2m^2-l^2}{m}(1-\varepsilon)^2\Big\}\\
    = &(2m-1)-\big(2m -\frac{1}{m}\big)(1-\varepsilon)^2\\
    = & 2m \big[1-(1-\varepsilon)^2 \big] -\Big[1-\frac{1}{m}(1-\varepsilon)^2\Big].
  \end{split}
      \end{equation}
\end{lemma}
Applying the above lemma to \eqref{eq:inter-upper}, we get
\begin{equation}\label{eq:inter-upper-final}
\begin{split}
    &m\max_{l=1,...,m-1}p^{2m-l}\PP\left( \tA_{ \{1,...,m\}}\cap \tA_{ \{1,...,l,m+1,...,2m-l\}}\right)\\
  \leq & \exp\Big\{2m \big[1-(1-\varepsilon_{m,p,t})^2 \big]\log p\\ &-\big[1-\frac{1}{m}(1-\varepsilon_{m,p,t})^2\big]\log p+ O(m^2\log\log p)\Big\}.
%  \leq & \exp\Big\{2m [1-(1-\varepsilon_{m,p,t})^2 ]\log p+ O(m^2\log\log p)\Big\}.
\end{split}
\end{equation}
This inequality together with \eqref{eq:etq-approx} implies that
\begin{equation}\label{eq:ratio-intermed}
	\begin{split}
    &{\left(\E \tilde{Q}_{m,p}\right)^{-2}}m\max_{l=1,...,m-1}p^{2m-l}\PP\left( \tA_{ \{1,...,m\}}\cap \tA_{ \{1,...,l,m+1,...,2m-l\}}\right)\\
  \leq & \exp\left\{-\left[1-\frac{1}{m}(1-\varepsilon_{m,p,t})^2\right]\log p+O(m^2\log\log p)\right\}.
\end{split}
\end{equation}
Combining the above display with \eqref{eq:var-q-split}, we arrive at
% \begin{equation}\label{eq:var-tq-final}
%   \begin{split}
%   &Var(\tQ_{m,p})\\
%   \leq & \exp\Big\{(2m [1-(1-\varepsilon_{m,p,t})^2 ]\log p\\
%   & -[1-(1-\varepsilon_{m,p,t})^2/m]\log p+ O(m^2(\log\log p-\log m))\Big\}\\
%   &+(\E\tQ_{m,p})+ (\E\tQ_{m,p})^2 \left(\frac{(p-m)!^2}{p!(p-2m)!}-1\right).
%   \end{split}
% \end{equation}
% Combining \eqref{eq:etq-approx} and the above display, we have
\begin{equation}\label{eq:var-exp-tq}
\begin{split}
    &\frac{Var\left(\tQ_{m,p}\right)}{\left(\E\tQ_{m,p}\right)^2}\\
      \leq &  \exp\left\{-\left[1-\frac{1}{m}(1-\varepsilon_{m,p,t})^2\right]\log p+O(m^2\log\log p)\right\}\\
      &+\left(\E\tQ_{m,p}\right)^{-1}+ \frac{(p-m)!^2}{p!(p-2m)!}-1.
  %     & \exp\{
  % -[1- (1-\varepsilon_{m,p,t})^2] m\log p + O(m^2(\log\log p-\log m) )
  % \} \\
%%mean  % = & \exp\{
  % -[1- (1-\varepsilon_{m,p,t})^2] m\log p + O(m^2(\log\log p-\log m) )
  % \} \\
  % +& \exp\{2m [1-(1-\varepsilon_{m,p,t})^2 ]\log p -[1-(1-\varepsilon_{m,p,t})^2/m]\log p\\
  % &+O(m^2(\log\log p-\log m))-2\log \E\tQ_{m,p}\}+ \frac{(p-m)!^2}{p!(p-2m)!}-1\\
  % =& \exp\{
  % -[1- (1-\varepsilon_{m,p,t})^2] m\log p + O(m^2(\log\log p-\log m) )
  % \} \\
  % +& \exp\{2m [1-(1-\varepsilon_{m,p,t})^2 ]\log p -[1-(1-\varepsilon_{m,p,t})^2/m]\log p\\
  % &-[1- (1-\varepsilon_{m,p,t})^2] 2m\log p
  % +O(m^2(\log\log p-\log m))\}+ \frac{(p-m)!^2}{p!(p-2m)!}-1
\end{split}
\end{equation}
\begin{lemma}\label{lemma:combine}
For all integers $p\geq m\geq 1$ satisfying $2m<p$, we have
  \begin{equation}
    \frac{(p-m)!^2}{p!(p-2m)!}<1.
  \end{equation}
\end{lemma}
%Using the above lemma, we further bound \eqref{eq:var-exp-tq} as
Therefore,
\begin{equation}\label{eq:ratio-tq-simp}
	\begin{split}
		    &\frac{Var(\tQ_{m,p})}{(\E\tQ_{m,p})^2}\\
      \leq &  \exp\left\{-\left[1-\frac{1}{m}(1-\varepsilon_{m,p,t})^2\right]\log p+O(m^2\log\log p)\right\}+\left(\E\tQ_{m,p}\right)^{-1}.
	\end{split}
\end{equation}
% \begin{equation}
%   \begin{split}
%     &\frac{Var(\tQ_{m,p})}{(\E\tQ_{m,p})^2}\\
%   \leq & \exp\{
%   -[1- (1-\varepsilon_{m,p,t})^2] m\log p + O(m^2(\log\log p-\log m) )
%   \} \\
%   +& \exp\{2m [1-(1-\varepsilon_{m,p,t})^2 ]\log p -[1-(1-\varepsilon_{m,p,t})^2/m]\log p\\
%   &-[1- (1-\varepsilon_{m,p,t})^2] 2m\log p
%   +O(m^2(\log\log p-\log m))\}.
% \end{split}
% \end{equation}
% Simplifying it, we arrive at
% \begin{equation}\label{eq:ratio-tq-simp}
%   \begin{split}
%     &\frac{Var(\tQ_{m,p})}{(\E\tQ_{m,p})^2}\\
%   \leq & \exp\{
%   -[1- (1-\varepsilon_{m,p,t})^2] m\log p + O(m^2(\log\log p-\log m) )
%   \} \\
%   +& \exp\{-[1-(1-\varepsilon_{m,p,t})^2/m]\log p
%   +O(m^2(\log\log p-\log m))\}.
% \end{split}
% \end{equation}
We now study the last two terms one by one. For $m\geq 2$,
\begin{equation}\label{eq:ratio-simp-b2}
\begin{split}
&-\left[1-\frac{1}{m}(1-\varepsilon_{m,p,t})^2\right]\log p
  +O(m^2\log\log p)\\
  \leq &-\frac{1}{2}\log p +O(m^2\log\log p)\\
  \leq & -\frac{1}{4}\log p
\end{split}
\end{equation}
for $n$ sufficiently large under Assumption 2 in \eqref{nice_birth_2}. Recalling $\varepsilon_{m,p,t}=(4m\log p)^{-1/2}t$, we see from \eqref{eq:etq-approx}  that
\begin{equation}\label{eq:ratio-simp-b1}
\begin{split}
&\log \left(\E\tQ_{m,p}\right)^{-1}\\
  =
  & -\left[1- (1-\varepsilon_{m,p,t})^2\right] m\log p + O(m^2\log\log p)\\
  \leq & -\varepsilon_{m,p,t}m\log p + O(m^2\log\log p)\\
  \leq & - \frac{t}{2}\sqrt{m\log p}  + O(m^2\log\log p).
\end{split}
\end{equation}
Combining \eqref{eq:ratio-tq-simp}, \eqref{eq:ratio-simp-b2} and \eqref{eq:ratio-simp-b1}, we arrive at
\begin{equation}
  \frac{Var(\tQ_{m,p})}{(\E\tQ_{m,p})^2}
  \leq \exp\left\{- \frac{t}{2}\sqrt{m\log p}  + O(m^2\log\log p)\right\}+
  \exp\left\{-\frac{1}{4}\log p\right\}.
\end{equation}
This together with \eqref{eq:tt-tq} and \eqref{eq:lower-final} yields
\begin{equation}
\begin{split}
	  &\PP\left(\tilde{T}_{m,p}\leq 2\sqrt{m\log p}-t\right)\\
  \leq &
  \exp\left\{- \frac{t}{2}\sqrt{m\log p}  + O(m^2\log\log p)\right\}+\frac{1}{p^{1/4}}
\end{split}
\end{equation}
uniformly for all $\delta\leq t\leq 2\sqrt{m\log p}- m\sinf$. Consequently, we get (\ref{eq:final-case1}).

\medskip

\noindent{\bf The proof of (\ref{sweet_home})}.
%In this step, we assume $t\geq 2\sqrt{m\log p}-m\sinf$.
%{\bf Case 2: $t\geq 2\sqrt{m\log p}-m\sinf$.}
%Here, we will find an upper bound for $\PP(\tilde{T}_{m,p}\leq 2\sqrt{m\log p}-t)$
%  as $t\geq 2\sqrt{m\log p}-m\sinf$.
For any $S\subset\{1,...,p\}$ with $|S|=m$, write $\tW_S=(\tW_{ij})_{i,j\in S}$. Note that $\lambda_1(\tW_S)\geq \max_{i\in S}\tW_{ii}$. Thus,
\begin{equation}
  \tilde{T}_{m,p}\geq \max_{S\subset\{1,...,p\},|S|=m} \lambda_1(\tW_S)
  \geq \max_{1\leq i\leq p}\tW_{ii}.
\end{equation}
As a result,
\beaa
	 \PP\Big(\tilde{T}_{m,p}\leq 2\sqrt{m\log p}-t\Big)
  &\leq& \PP\Big(\max_{1\leq i\leq p}\tW_{ii}\leq 2\sqrt{m\log p}-t\Big)\\
 & =& {\Phi}\Big(\sqrt{2m\log p}-\frac{1}{\sqrt{2}}t\Big)^p,
\eeaa
 where the function $\Phi(z)=\int_{-\infty}^z \frac{1}{\sqrt{2\pi}}e^{-\frac{s^2}{2}}ds$ for $z\in \mathbb{R}$.
To proceed, we discuss two scenarios: $2\sqrt{m\log p}-m\sinf\leq t\leq 4\sqrt{m\log p}$ and $t> 4\sqrt{m\log p}$. For $2\sqrt{m\log p}-m\sinf\leq t\leq 4\sqrt{m\log p}$,  we have
\begin{equation}
  \begin{split}
    &{\Phi}\left(\sqrt{2m\log p}-\frac{t}{\sqrt{2}}\right)^p\\
    \leq & {\Phi}\left(\sqrt{2m\log p}-\frac{2\sqrt{m\log p}-m\sinf}{\sqrt{2}}\right)^p \\
    =&{\Phi}\left(\frac{m\sinf}{\sqrt{2}}\right)^p\\
    = & \exp\left\{
    p\log\left(1-\bar{\Phi}\left(\frac{m\sinf}{\sqrt{2}}\right)\right)
    \right\}\\
    \leq &\exp\left\{
    - p \bar{\Phi}\left(\frac{m\sinf}{\sqrt{2}}\right)
        \right\},
  \end{split}
\end{equation}
where $\bar{\Phi}(z)=1-\Phi(z)$ for any $z\in \mathbb{R}$ and the inequality $\log (1-x) \leq -x$ for any $x<1$ is used in the last step.
Note $\bar{\Phi}\left(\frac{1}{\sqrt{2}}m\sinf\right)=(1+o(1)) \frac{1}{\sqrt{4\pi}m\sinf}e^{-\frac{m^2\sinf^2}{4}}$ and $ p^{0.1}(\sinf)^{-1} e^{-\frac{m^2\sinf^2}{4}}\to\infty $ since $\sinf=\log\log\log p$. Thus,
\begin{equation}\label{eq:very-small-t-first-case}
	\begin{split}
		{\Phi}\Big(\sqrt{2m\log p}-\frac{t}{\sqrt{2}}\Big)^p\leq \exp\Big\{
     - p^{0.9} m
     \Big\},
	\end{split}
\end{equation}
for sufficiently large $p$. This further implies
\begin{equation}
\begin{split}
	&\lim_{p\to\infty}\sup_{2\sqrt{m\log p}-m\sinf\leq t\leq 4\sqrt{m\log p} }e^{\alpha t}t^2\PP\left(\tilde{T}_{m,p}\leq 2\sqrt{m\log p}-t\right)\\
	\leq & \limsup_{p\to\infty}\exp\left\{
     - p^{0.9} m + \alpha\cdot 4\sqrt{m\log p}+ 2\log\left( 4\sqrt{m\log p}\right)
     \right\}\\
     =&0 \label{eq:final-scenario1}
\end{split}
\end{equation}
for any $\alpha>0$. Note that $\Phi(-x)=\bar{\Phi}(x)\leq \frac{1}{\sqrt{2\pi}\,x}e^{-x^2/2}\leq e^{-x^2/2}$ for any $x\geq 1$.
%That is,
%\begin{equation}\label{eq:final-scenario1}
%	\lim_{p\to\infty}\sup_{2\sqrt{m\log p}-m\sinf\leq t\leq 4\sqrt{m\log p} }e^{\alpha t}t^2\PP\left(\tilde{T}_{m,p}\leq 2\sqrt{m\log p}-t\right)=0.
%\end{equation}
Then, for the other scenario where $t\geq 4\sqrt{m\log p}$, we have
\begin{equation}
  {\Phi}\left(\sqrt{2m\log p}-\frac{t}{\sqrt{2}}\right)^p
  \leq {\Phi}\left(-\frac{t}{2\sqrt{2}}\right)^p
  \leq \exp\left\{-\frac{pt^2}{16} \right\}
\end{equation}
as $n$ is large enough. Thus,
\begin{equation}\label{eq:final-scenario2}
	\begin{split}
	&\lim_{p\to\infty}\sup_{t\geq 4\sqrt{m\log p}  }e^{\alpha t}t^2\PP\left(\tilde{T}_{m,p}\leq 2\sqrt{m\log p}-t\right)\\
	\leq & \limsup_{p\to\infty}\sup_{ t\geq 4\sqrt{m\log p} }\exp\left\{-\frac{pt^2}{16} +\alpha t+2\log t\right\}\\
     =& 0
\end{split}
\end{equation}
for any $\alpha>0.$ Joining \eqref{eq:final-scenario1} and \eqref{eq:final-scenario2}, we see (\ref{sweet_home}). This completes the whole proof.
\end{proof}

\subsubsection{Proof of Theorem~\ref{thm:wishart}}\label{Sec_thm:wishart}

To prove Theorem~\ref{thm:wishart}, we need the following two propositions.
\begin{proposition}\label{upper_bound_11} Suppose Assumption 1 in \eqref{nice_birth_1} holds. Recall $T_{m,n,p}$ defined as in \eqref{eq:t-stat}. Then,
\beaa
\lim_{n\to\infty}\sup_{t\geq \delta} e^{\alpha t} t^2\PP\left(\frac{1}{\sqrt{n}}(T_{m,n,p}-n)\geq 2\sqrt{m\log p} +t
  \right)= 0
\eeaa
for any $\alpha>0$ and $\delta>0$.
\end{proposition}

\begin{proposition}\label{lower_bound_11} Suppose Assumption 1 in \eqref{nice_birth_1} holds. Recall $T_{m,n,p}$ defined as in \eqref{eq:t-stat}. Then,
\beaa
\lim_{n\to\infty}\sup_{t\geq \delta} e^{\alpha t} t^2\PP\left(\frac{1}{\sqrt{n}}(T_{m,n,p}-n)\leq 2\sqrt{m\log p} -t
  \right)= 0
\eeaa
for any $\alpha>0$ and $\delta>0$.
\end{proposition}
\begin{proof}[Proof of Theorem~\ref{thm:wishart}] Similar to the proof of Theorem~\ref{thm:wigner}, it is sufficient to prove \eqref{eq:mgf-thm-wish}.
%Without loss of generality, we assume $\delta<1$ since the expectation  in \eqref{eq:mgf-thm-wish} is monotonically decreasing in $\delta$.
By the same argument as in
the proof of Theorem~\ref{thm:wigner}, with the upper bound for $\PP(\frac{T_{m,n,p}-n}{\sqrt{n}}\geq 2\sqrt{m\log p}+t)$ given in Proposition \ref{upper_bound_11} and the upper bound for $\PP(\frac{T_{m,n,p}-n}{\sqrt{n}}\leq 2\sqrt{m\log p}-t)$ for $t>\delta$ given in Proposition \ref{lower_bound_11}, we get \eqref{eq:mgf-thm-wish}.
\end{proof}

In the following we start to prove  Propositions \ref{upper_bound_11} and  \ref{lower_bound_11}.
\begin{proof}[Proof of Proposition~\ref{upper_bound_11}]
%We will need a few of lemmas in the process. To make the whole proof more concise, we place their proofs in the Appendix. Also, sometimes  a statement or a formula holds only as $n$ is sufficiently large. We will not say ``as $n$ is sufficiently  large" if the context is apparent.
Without loss of generality, we assume $\delta<1$ since the expectation  in \eqref{eq:mgf-thm-wish} is monotonically decreasing in $\delta$.

Let $W_{\{1,...,m\}}$ be as $W_S$ above \eqref{eq:t-stat} with $S=\{1,2,\cdots, m\}$. Analogous to \eqref{eq:upper-split}, we have
\begin{equation}\label{eq:split-wish}
\begin{split}
  &\PP\left(\frac{1}{\sqrt{n}}(T_{m,n,p}-n)\geq2\sqrt{m\log p}+t\right)\\
  \leq &p^m \PP\left(
  \frac{1}{\sqrt{n}}(\lambda_1(W_{\{1,...,m\}})-n)\geq 2\sqrt{m\log p} + t
  \right).
\end{split}
\end{equation}
We now bound the last probability.
%proceed to an upper bound of $\PP(
%  \frac{1}{\sqrt{n}}(\lambda_1(W_{\{1,...,m\}})-n)\geq 2\sqrt{m\log p} + t
%  )$.
  Since the above tail probability involve moderate bound and large deviation bound for different ranges of $t$, we will discuss three different cases and   use different proof strategies. Recall $\xi_p=\log\log\log p$. Set
\bea\lbl{green_fresh}
\omega_n=\Big(\frac{m}{\log p}\Big)^{1/2}\sinf\log n.
\eea
The three cases are: (1) $t>\frac{\delta\sqrt{n}}{100}$, (2) $\delta\vee \omega_n \leq t\leq \frac{\delta\sqrt{n}}{100}$, and (3) $\delta\leq t < \delta\vee \omega_n$.
They cover all situations for  $t\geq \delta$.
For the first two cases, the upper bound is based on the next lemma, which gives a moderate deviation bound for the spectrum of $\frac{1}{\sqrt{n}}W_{\{1,...,m\}}$ from the identity matrix $I_m$.
\begin{lemma}\label{lemma:moderate-bound}
There exists a constant $\kappa>0$ such that for all $n,p,m$,  $r\geq 1$, $0<d<1/2$ and $y>2dmr$, we have
% \begin{equation}\label{eq:max-lambda-bound}
%   \PP\left(\frac{\lambda_1(W_{\{1,...,m\}}-n)}{\sqrt{n}}\geq \sqrt{n}r\right)
%   \leq \exp\{-\frac{1}{2}(r-1-\log(r))mn\},
% \end{equation}
% and
\begin{equation}\label{mid_range}
\begin{split}
	& \PP\left(
  \frac{\lambda_1(W_{\{1,...,m\}})-n}{n}\geq y
  \right)\\
 % \leq & \exp\big\{-nI(y-2rdm)+\kappa m\log(1/d)\big\}
%  + \exp\Big\{-\frac{1}{2}(r-1-\log(r))mn\Big\},
 \leq & 2\cdot\exp\Big\{-nI(1+y-2dmr)+\kappa m\log\frac{1}{d}\Big\}
  +2\cdot e^{-mnI(r)}
\end{split}
 \end{equation}
and
\bea\label{mid_range_2}
\begin{split}
	 &\PP\left(
  \frac{\lambda_m(W_{\{1,...,m\}})-n}{n}\leq- y
  \right)\\
 % \leq & \exp\big\{-nI(y-2rdm)+\kappa m\log(1/d)\big\}
%  + \exp\Big\{-\frac{1}{2}(r-1-\log(r))mn\Big\},
 \leq & 2\cdot\exp\Big\{-nI(1-y+2dmr)+\kappa m\log\frac{1}{d}\Big\}
  + 2\cdot e^{-mnI(r)}
\end{split}
 \eea
where $I(s)=\frac{1}{2}(s-1-\log s)$ for $s> 0$ and  $I(s)=\infty$ for $s\leq 0$.
%$I_+(s)=\frac{1}{2}(s-\log(1+s))$ for $s>-1$ and $I_-(s)=\frac{1}{2}(-s-\log(1-s))$ for $s<1$ and $I_-(s)=\infty$ for $s\geq 1$.
\end{lemma}
{\bf Case 1: $t>\frac{\delta\sqrt{n}}{100}$.} Let $\alpha>0$ be given.
Choose $r=\max(2, 1+\frac{80\alpha t}{mn})$,  $d=\min(\frac{1}{2}, \frac{t}{4m\sqrt{n}r})$, and $y=\frac{2\sqrt{m\log p}+t}{\sqrt{n}}$ in Lemma~\ref{lemma:moderate-bound}. The choice of $r,d,$ and $y$ satisfies that  $2dmr\leq \frac{t}{2\sqrt{n}}$ and hence $y-2dmr\geq \frac{2\sqrt{m\log p}}{\sqrt{n}}+\frac{t}{2\sqrt{n}}$. Set $z=\frac{2\sqrt{m\log p}}{\sqrt{n}}+\frac{t}{2\sqrt{n}}$.  Notice that $I(s)$ from Lemma \ref{lemma:moderate-bound} is increasing for $s\geq 1$. Then, by the lemma,
\begin{equation}\label{eq:case1-main}
\begin{split}
&t^2e^{\alpha t}p^m\PP\left(
  \frac{\lambda_1(W_{\{1,...,m\}})-n}{\sqrt{n}}
  \geq 2\sqrt{m\log p}+t
  \right)\\
  = &t^2e^{\alpha t}p^m\PP\left(
  \frac{\lambda_1(W_{\{1,...,m\}})-n}{\sqrt{n}}
  \geq \sqrt{n}y
  \right)\\
%    \leq &t^2e^{\alpha t}p^m\PP\left(
%  \frac{\lambda_1(W_{\{1,...,m\}})-n}{\sqrt{n}}
%  \geq \sqrt{n}z
%  \right)\\
    \leq &
  2\cdot \exp\left\{
  -\frac{n}{2}[z-\log(1+z)]+\kappa m\log\frac{1}{d} +\alpha t+2\log t+ m \log p
  \right\}\\
  &+2\cdot \exp\left\{
  -\frac{1}{2}(r-1-\log r) mn+\alpha t+2\log t+m\log p
  \right\}.
\end{split}
\end{equation}
The following lemma says that both of the last two terms go to zero.
\begin{lemma}\label{lemma:case1}
Suppose Assumption 1 in \eqref{nice_birth_1} holds.	Let $\alpha>0$ and $\delta>0$ be given. For $r=\max(2, 1+\frac{80\alpha t}{mn})$, $d=\min(\frac{1}{2}, \frac{t}{4m\sqrt{n}r})$  and $z=\frac{2\sqrt{m\log p}}{\sqrt{n}}+\frac{t}{2\sqrt{n}}$,
% $r=\max(2, 1+\frac{80\alpha t}{mn})$ and  $d=\min(\frac{1}{2}, \frac{t}{4m\sqrt{n}r})$,
 we have
	\begin{equation}\label{eq:case1-first}
		\lim_{n\to\infty}\sup_{t>\frac{\delta\sqrt{n}}{100}}\exp\left\{
  -\frac{n}{2}[z-\log(1+z)]+\kappa m\log\frac{1}{d} +\alpha t+2\log t+ m \log p
  \right\}=0
	\end{equation}
	and
	\begin{equation}\label{eq:case1-second}
		\lim_{n\to\infty}\sup_{t>\frac{\delta\sqrt{n}}{100}} \exp\left\{
  -\frac{1}{2}(r-1-\log r) mn+\alpha t+2\log t+m\log p
  \right\}=0.
	\end{equation}

\end{lemma}
Combining \eqref{eq:split-wish}, \eqref{eq:case1-main}-\eqref{eq:case1-second}, we conclude
\begin{equation}\label{eq:case1-final}
  \lim_{n\to\infty}\sup_{t>\frac{\delta\sqrt{n}}{100}}
   t^2e^{\alpha t}\PP\left(\frac{1}{\sqrt{n}}(T_{m,n,p}-n)
  \geq 2\sqrt{m\log p}+t
  \right)=0.
\end{equation}
%This completes our analysis for Case 1.
%
%We proceed to the analysis to Case 2.

\noindent{\bf Case 2: $\delta\vee \omega_n \leq t\leq \frac{\delta\sqrt{n}}{100}$.} Review $\omega_n$ in \eqref{green_fresh}.
Now we choose $r=2$, $d=\frac{t}{8m\sqrt{n}}<\frac{1}{2}$ and $y=\frac{2\sqrt{m\log p}+t}{\sqrt{n}}$. Then $y>\frac{t}{2\sqrt{n}}=2dmr$. By \eqref{mid_range},  %Lemma~\ref{lemma:moderate-bound},
%
% in Lemma~\ref{lemma:moderate-bound}. Note that we abuse the notation a little and have different choices of $r$ and $d$ in Case 2 and Case 1. From Lemma~\ref{lemma:moderate-bound} and our choice of $r$ and $d$, we have
\begin{equation}\label{eq:case2-main}
  \begin{split}
    &t^2e^{\alpha t}p^m\PP\left(
  \frac{\lambda_1(W_{\{1,...,m\}})-n}{\sqrt{n}}
  \geq 2\sqrt{m\log p}+t
  \right)\\
  \leq &
  2\cdot\exp\left\{
  -\frac{n}{2}[z-\log(1+z)]+\kappa m\log\frac{1}{d} +\alpha t+2\log t+ m \log p
  \right\}\\
  &+2\cdot\exp\left\{
  -\frac{1}{2}(1-\log 2) mn+\alpha t+2\log t+m\log p
  \right\}
  \end{split}
\end{equation}
where  $z:=y-2dmr = \frac{2\sqrt{m\log p}}{\sqrt{n}}+\frac{t}{2\sqrt{n}}$.
The last two terms are analyzed in the next lemma.
\begin{lemma}\label{lemma:case2} Suppose Assumption 1 in \eqref{nice_birth_1} holds. Let $\omega_n$ be as in \eqref{green_fresh}.
For $\delta\vee \omega_n \leq t\leq \frac{\delta\sqrt{n}}{100}$, $z=\frac{2\sqrt{m\log p}}{\sqrt{n}}+\frac{t}{2\sqrt{n}}$ and $d=\frac{t}{8m\sqrt{n}}$, we have
	 \begin{equation}\label{eq:case2-first}
 \begin{split}
 	&\exp\left\{
  -\frac{n}{2}[z-\log(1+z)]+\kappa m\log\frac{1}{d} +\alpha t+2\log t+ m \log p
  \right\}\\
    \leq &\exp\left\{
  -\frac{1}{4}t\sqrt{m\log p}
  \right\}
 \end{split}
 \end{equation}
 as $n$ is sufficiently large. In addition,
 \bea\label{eq:case2-second}
 & &  -\frac{1}{2}(1-\log 2) mn+\alpha t+2\log t+m\log p\nonumber\\
  &=& -\frac{1-\log 2}{2}[1+o(1)] mn
\eea
as $n\to\infty$.
\end{lemma}
Joining \eqref{eq:case2-main}-\eqref{eq:case2-second}, we obtain
\begin{equation}\label{eq:case2-final}
\begin{split}
\lim_{n\to\infty}\sup_{\delta\vee \omega_n \leq t\leq \frac{\delta\sqrt{n}}{100}}
   p^mt^2e^{\alpha t}\PP\left(
  \frac{\lambda_1(W_{\{1,...,m\}})-n}{\sqrt{n}}
  \geq 2\sqrt{m\log p}+t
  \right)=0,
\end{split}
\end{equation}
which together with \eqref{eq:split-wish} implies that
\begin{equation}\label{eq:case2-final_5}
\begin{split}
\lim_{n\to\infty}\sup_{\delta\vee \omega_n \leq t\leq \frac{\delta\sqrt{n}}{100}}
   t^2e^{\alpha t}\PP\left(\frac{1}{\sqrt{n}}(T_{m,n,p}-n)
  \geq 2\sqrt{m\log p}+t
  \right)=0.
%
%  &\lim_{n\to\infty}\sup_{\delta\vee (m/\log p)^{0.5}\log n\sinf\leq t\leq \frac{\delta \sqrt{n}}{100}}t^2e^{\alpha t}p^m\PP\left(
%  \frac{\lambda_1(W_{\{1,...,m\}})-n}{\sqrt{n}}
%  \geq 2\sqrt{m\log p}+t
%  \right)\\
%  &=0.
%  % \leq& \exp\{
%  % -\frac{1}{4}\sqrt{m\log p}\delta
%  % \}
%  % +\exp\{-\frac{1}{2}(1-\log 2) mn + O(\sqrt{n})+m\log p\}\to 0,
\end{split}
\end{equation}
This completes our analysis for Case 2. By using the same argument as obtaining \eqref{eq:case2-final}, we have the following limit, which will be used later on.
 \bea\label{eq:case2-final_1}
& & \lim_{n\to\infty}\sup_{\delta\vee \omega_n \leq t\leq \frac{\delta\sqrt{n}}{100}}
   t^2e^{\alpha t}p^m\PP\left(
  \frac{\lambda_m(W_{\{1,...,m\}})-n}{\sqrt{n}}
  \leq -2\sqrt{m\log p}-t
  \right)=0.\nonumber\\
& &
\eea
We next study Case 3.

\noindent{\bf Case 3: $\delta\leq t < \delta\vee \omega_n$.}
Note that this case
is only possible
%when $n$ is large enough (compared to $\log p$) such that
if $n\geq \exp\{
((\log p)/m)^{1/2}\sinf^{-1}\delta
\}$.
We point out that Lemma~\ref{lemma:moderate-bound} is not a suitable approach for bounding the tail probability in this case because the term $m\log (1/d)$, which cannot be easily controlled, will dominate the other terms in the error bound  for very large $n$.
Instead, we will use another approach to obtain an upper bound of $\PP\left(
  \lambda_1(W_{\{1,...,m\}})\geq 2\sqrt{m\log p}+t
  \right)$. The main step here is to quantify the approximation of the extreme eigenvalue of a Wishart matrix to that of a Wigner matrix. We will analyze their density functions and leverage them with the results in the proof of Theorem~\ref{thm:wigner}.

Let $\mu=(\mu_1,...,\mu_m)$ be the order statistics of the eigenvalues of $W_{\{1,...,m\}}$ such that $\mu_1>\mu_2>...>\mu_m$. Write $\nu=(\nu_1,...,\nu_m)$ with $\nu_i=(\mu_i-n)/\sqrt{n}$. Let $\tW_{\{1,...,m\}}=(\tilde{w}_{ij})_{1\leq i,j\leq m}$ where $\tilde{w}_{ij}$'s are   as in \eqref{eq:dist-wig}. Let the  eigenvalues of $\tW_{\{1,...,m\}}$ be $\lambda_1>... >\lambda_m$. Set $\lambda=(\lambda_1,...,\lambda_m)$.  Intuitively, the law of $\nu$ is close to that of $\lambda$ when $n$ is large. The next lemma quantifies the approximation speed. Review $\|x\|_{\infty}=\max_{1\leq i \leq m}|x_i|$ for any $x=(x_1, \cdots, x_m) \in \mathbb{R}^m$.
\begin{lemma}\label{lemma:wig-eigen-approx}
  Let $g_{n,m}(\cdot)$ be the density function of $\nu$, and let $\fwig{m}(\cdot)$ be the density function of $\lambda$.  Assume $m^3=o(n)$.  Then,
  \beaa
 & & 	    \log g_{n,m}(v)-\log \fwig{m}(v)\\
 &=& {o(1)} +O\left(m^2n^{-1/2}\|v\|_{\infty}+ m^2n^{-1} \|v\|_{\infty}^2+ mn^{-1/2} \|v\|_{\infty}^3\right)
  \eeaa
for all $v\in \mathbb{R}^m$ with $\|v\|_{\infty}\leq {\frac{2}{3}}\sqrt{n}$.
\end{lemma}

Let $r_{m,n}=2\sqrt{m\log p}+\omega_n$, where $\omega_n$ is as in \eqref{green_fresh}.   Then for $t$ such that $\delta\leq t\leq \omega_n$,
\begin{equation}\label{eq:case3-main}
\begin{split}
    & \PP\left(\frac{1}{\sqrt{n}}
  \left(\lambda_1(W_{\{1,...,m\}})-n\right)\geq 2\sqrt{m\log p} +t
  \right)\\
  \leq & \PP\left(\frac{1}{\sqrt{n}}
  \left(\lambda_1(W_{\{1,...,m\}})-n\right)\geq 2\sqrt{m\log p} +t, \max_{1\leq i \leq m}|\nu_i|\leq r_{m,n}
  \right)\\
  &~~~~~~~~~~~~~~~~~~~~~~~~~~~~~~~~~~~~+ \PP\big(
\max_{1\leq i \leq m}|\nu_i|> r_{m,n}
  \big).
\end{split}
\end{equation}
There are three probabilities above, denote the second one  by  $H_{n}$.
For $H_{n}$,  we use the change-of-measure argument. In fact,
%s combined with Lemma~\ref{lemma:wig-eigen-approx}.
\begin{equation}
  \begin{split}
%  & H_{n,1}\\
  %&\PP\left(\frac{1}{\sqrt{n}}
%  \left(\lambda_1(W_{\{1,...,m\}})-n\right)\geq 2\sqrt{m\log p} +t, \max_{1\leq i \leq m}|\nu_i|\leq r_{m,n}
%  \right)\\
H_{n}  = & \int_{v_1\geq 2\sqrt{m\log p} +t, \|v\|_{\infty}\leq r_{m,n}} g_{n,m}(v)dv\\
    = & \int_{v_1\geq 2\sqrt{m\log p} +t,\|v\|_{\infty}\leq r_{m,n}} \exp\{ \log g_{n,m}(v)-  \log \fwig{m}(v) \}\fwig{m}(v) dv\\
    = & \exp\Big\{{o(1)+} O(m^2n^{-1/2}r_{m,n}) + O(m^2n^{-1} r_{m,n}^2)+O(m n^{-1/2} r_{m,n}^3)\Big\}\\
    & ~~~~~~~~~~~~~~~~~~~~~~~~~~~~~~~~~~~~~~~~\cdot\int_{v_1\geq 2\sqrt{m\log p} +t,\|v\|_{\infty}\leq r_{m,n}} \fwig{m}(v) dv\\
    % = & e^{O(nmr_n^3)}\PP\left(
    % \lambda_1(\tW_{[1,...m]})\geq x ;\mathcal{L}
    % \right)\\
    \leq &{2\cdot}\exp\Big\{O(m^2n^{-1/2}r_{m,n}) + O(m^2n^{-1} r_{m,n}^2)+O(m n^{-1/2} r_{m,n}^3)\Big\}\\
    &~~~~~~~~~~~~~~~~~~~~~~~~~~~~~~~~~~~~~~ \cdot\PP\left(
    \lambda_1(\tW_{\{1,...m\}})\geq 2\sqrt{m\log p} +t
    \right).
  \end{split}
\end{equation}
Now
\beaa
& & O(m^2n^{-1/2}r_{m,n}) + O(m^2n^{-1} r_{m,n}^2)+O(m n^{-1/2} r_{m,n}^3)\\
&=& \Big(\frac{m}{r_{m,n}^2}+ \frac{m}{\sqrt{n}r_{m,n}}+1\Big)\cdot O(m n^{-1/2} r_{m,n}^3)\\
& = & O(m n^{-1/2} r_{m,n}^3)
\eeaa
since $r_{m,n}>\sqrt{m\log p}$ and $m=o(n)$. By the definition of $\omega_n$  in \eqref{green_fresh},
\beaa
& & m n^{-1/2} r_{m,n}^3\\
&=& m n^{-1/2}\cdot O\Big((m\log p)^{3/2} + m^{3/2}(\log n)^3(\log\log\log p)^3(\log p)^{-3/2}\Big)\\
& = & \sqrt{m\log p}\cdot O\Big(\frac{m^2\log p}{\sqrt{n}}+\frac{m^2(\log n)^3(\log\log\log p)^3}{\sqrt{n}(\log p)^2}\Big)\\
& = & o(\sqrt{m\log p})
\eeaa
where Assumption 1 from (\ref{nice_birth_1}) is used. Therefore,
\begin{equation}
  H_{n}\leq
     \exp{\left(o\big(\sqrt{m\log p}\,\big)\right)} \cdot \PP\left(
    \lambda_1(\tW_{\{1,...m\}})\geq 2\sqrt{m\log p} +t
    \right).
\end{equation}
Note that $t\leq \frac{1}{\beta}e^{\beta t}$ for any $\beta>0$ and $t>0$. It follows from  \eqref{eq:wig-tail-p} that
%Combine the above inequality with \eqref{eq:tp-upper}, we further have
\begin{equation}
\begin{split}
  &\sup_{t\geq \delta}\big\{p^m e^{\alpha t} t^2H_{n}\big\}
%\PP\Big(
%  \frac{\lambda_1(W_{\{1,...,m\}})-n}{\sqrt{n}}\geq 2\sqrt{m\log p} +t, \max_{1\leq i \leq m}|\nu_i|\leq r_{m,n}
%  \Big)
\\
  \leq & \sup_{t\geq \delta}\exp\Big\{
  -\frac{1}{2}t\sqrt{m\log p}  + O(m\log\log p) + o(\sqrt{m\log p}\,)
  \Big\}\\
  \leq & \sup_{t\geq \delta}\exp\Big\{
  -\frac{1}{2}\sqrt{m\log p}\cdot (\delta+o(1))
 + o(\sqrt{m\log p}\,)
  \Big\}\\
  = & o(1)
\end{split}
\end{equation}
by the fact $t\geq \delta$ and Assumption 1.
% {\color{red} modify Assumption 1 such that
% \begin{itemize}
% % \item $m\log p= o(n)$
%   \item $m=o(\sqrt{n})$
%   \item $m=o(\sqrt{\log p})$
%   \item $m = o(\frac{n^{1/4}}{\log p^{1/2}})$
%   \item $\log p= o(\sqrt{n})$
%   \item $m=o(\frac{n^{1/4}\log p}{\log n^{1.5}\sinf^{3/2}})$
% \end{itemize}
% Overall, we need
% $\log p= o(\sqrt{n})$, $m=o(\frac{n^{1/4}}{(\log n)^{1.5}(\log p)^{1/2}})$
% }
% Thus,
% \begin{equation}
% \begin{split}
%   &\\sup_{t\geq \delta}p^m e^{\alpha t} t^2\PP\Big(\frac{\lambda_1(W_{\{1,...,m\}})-n}{\sqrt{n}}\geq 2\sqrt{m\log p} +t, \max_{1\leq i \leq m}|\nu_i|\leq r_{m,n}
%   \Big)\\
%   \leq &\exp\left\{
%   -\frac{1}{2}\sqrt{m\log p} (\delta+o(1))
%   \right\}\\
%   \leq & \exp\left\{
%   -\frac{\delta}{4}\sqrt{m\log p}
%   \right\},
% \end{split}
% \end{equation}
% for $n,p$ sufficiently large.
Combining this with \eqref{eq:case3-main}, we have
\begin{equation}\label{eq:to-bound-second}
  \begin{split}
  &\sup_{\delta\leq t\leq \delta\vee\omega_n}\left\{p^m e^{\alpha t} t^2\cdot\PP\Big(
  \frac{\lambda_1(W_{\{1,...,m\}})-n}{\sqrt{n}}\geq 2\sqrt{m\log p} +t
  \Big)\right\}\\
   \leq & o(1)+p^m e^{\alpha\omega_n+2\log \omega_n}\cdot  \PP\big(
  \max_{1\leq i \leq m}|\nu_i|\geq r_{m,n}
  \big).
  % \leq & \exp\left\{
  % -\frac{\delta}{4}\sqrt{m\log p}
  % \right\} +p^m e^{\alpha r_{m,n}+2\log r_{m,n}}  \PP\left(
  % \max_{1\leq i \leq m}|\nu_i|\geq r_{m,n}
  % \right)
\end{split}
\end{equation}\label{eq:tail-max}
We next analyze $\PP\big(\max_{1\leq i \leq m}|\nu_i|\geq r_{m,n}
  \big)$. Recall $r_{m,n}=2\sqrt{m\log p}+\omega_n$, where $\omega_n$ is as in \eqref{green_fresh}. Recall that we only discuss Case 3 when $\delta\leq t< \delta\vee\omega_n$, and this is only meaningful when $\omega_n>\delta$. Thus,  $\delta\vee\omega_n  = \omega_n\leq \frac{\sqrt{n}\delta}{100}$. Thus, from \eqref{eq:case2-final} we have
  %holds for $t=r_{m,n}$ because \eqref{eq:case2-first} and \eqref{eq:case2-second} both holds for $t=r_{m,n}$. As a result,
  \begin{equation}\label{eq:tail-wish-max}
  \lim_{n\to\infty}p^m e^{\alpha \omega_n+2\log \omega_n}  \PP\left(
  \frac{\lambda_1(W_{\{1,...,m\}})-n }{\sqrt{n}}\geq r_{m,n}
  \right)= 0.
\end{equation}
%Using a similar proof as that for the above inequality, we also have
By \eqref{eq:case2-final_1},
\begin{equation}\label{eq:tail-min}
  \lim_{n\to\infty}p^m e^{\alpha \omega_n+2\log \omega_n}  \PP\left(
  \frac{\lambda_m(W_{\{1,...,m\}})-n }{\sqrt{n}}\leq- r_{m,n}
  \right)= 0.
\end{equation}
%%%%%BELOW ARE MORE CLARIFICATION FOR THIS INEQUALITY%%%
%{We clarify that the only difference in the proof of the above two inequalities is that the term $(z-\log(1+z))$ in \eqref{eq:small-1} is replaced by $-z-\log(1-z)$ to account for the minimal eigenvalue of the Wishart matrix, and $-z-\log(1-z)\geq \frac{1}{2}z^2$ is used to derive results that are similar to \eqref{eq:small-3} and \eqref{eq:small-4}.}
Since $\max_{1\leq i \leq m}|\nu_i|=\max(\nu_1, -\nu_m)$, by combining \eqref{eq:tail-wish-max} and \eqref{eq:tail-min}, we see that
\begin{equation}\label{eq:spec-bound-rough}
  \lim_{n\to\infty}p^m e^{\alpha \omega_n+2\log \omega_n}  \PP\left(
  \max_{1\leq i \leq m}|\nu_i|\geq r_{m,n}
  \right)= 0.
\end{equation}
Combining this with \eqref{eq:to-bound-second}, we further have
\begin{equation}\label{eq:case3-final}
  \lim_{n\to\infty}\sup_{\delta\leq t\leq \delta\vee\omega_n}p^m e^{\alpha t} t^2\PP\left(\frac{1}{\sqrt{n}}
  (\lambda_1(W_{\{1,...,m\}})-n)\geq 2\sqrt{m\log p} +t
  \right)= 0.
\end{equation}
% for $n,p \to\infty$ under Assumption 1.
This completes our analysis for Case 3.

Now, we combine \eqref{eq:case1-final}, \eqref{eq:case2-final} and \eqref{eq:case3-final}, and arrive at
\begin{equation}
  \lim_{n\to\infty}\sup_{t\geq \delta}p^m e^{\alpha t} t^2\PP\left(
  \frac{1}{\sqrt{n}}(\lambda_1(W_{\{1,...,m\}})-n)\geq 2\sqrt{m\log p} +t
  \right)= 0.
\end{equation}
This and \eqref{eq:split-wish} conclude
\begin{equation}\label{eq:upper-final}
\lim_{n\to\infty}\sup_{t\geq \delta} e^{\alpha t} t^2\PP\left(\frac{1}{\sqrt{n}}(T_{m,n,p}-n)\geq 2\sqrt{m\log p} +t
  \right)= 0.
\end{equation}
\end{proof}
%This completes our analysis for the upper tail of $T_{m,n,p}$. We proceed to the analysis of the lower tail probability of $T_{m,n,p}$ in the form of $\PP\left(T_{m,n,p}\leq 2\sqrt{m\log p}-t\right)$ for $t\geq \delta$.

\begin{proof}[Proof of Proposition~\ref{lower_bound_11}] Noticing the expectation  in \eqref{eq:mgf-thm-wish} is non-increasing in $\delta$. Without loss of generality, we assume $\delta<1$.

Here we discuss two scenarios that are similar to those in the proof of Theorem~\ref{thm:wigner}. They are 1) $\delta\leq t\leq 2\sqrt{m\log p}- m\sinf$ and 2) $ t>2\sqrt{m\log p}- m\sinf$, where $\sinf=\log\log\log p.$

\noindent{\bf Scenario 1: $\delta \leq t\leq 2\sqrt{m\log p}- m\sinf$.}
Similar to the proof of Theorem~\ref{thm:wigner},
we define the event $A_S$ as follows. For each $S\subset\{1,...,p\}$ with $|S|=m$, set
\bea\label{eq:as}
%\begin{split}
& &	 A_{S}=\Big\{ \frac{1}{\sqrt{n}}(W_{kk}-n)\geq \tau_{m,p,t},
	  \frac{W_{ij}}{\sqrt{n}}\geq \tau_{m,p,t}\nonumber\\
& &  ~~~~~~~~~~~~~~~~~~~~~~~~~~~~~~~\text{ for all } i,j,k\in S \text{ and } i<j
\Big\},
%\end{split}
\eea
where $\tau_{m,p,t}=(1-\varepsilon_{m,p,t})\sqrt{\frac{4\log p}{m}}$ and $\varepsilon_{m,p,t}=(4m\log p)^{-1/2}t$.
We also define
\begin{equation}\label{eq:q}
  Q_{m,n,p}=\sum_{S\subset\{1,...,p\}:\, |S|=m} \mathbf{1}_{A_S}.
\end{equation}
Similar to the discussion between  \eqref{eq:lb-spec} and \eqref{eq:lower-final} in  the proof of Theorem~\ref{thm:wigner}, we have
\begin{equation}\label{eq:t-moment-bound}
\PP\left(T_{m,n,p}\leq 2\sqrt{m\log p}-t\right)\leq
\frac{Var(Q_{m,n,p})}{\E(Q_{m,n,p})^2}.
\end{equation}
In the rest of the discussion under Scenario 1, we will develop a lower bound for $\E(Q_{m,n,p})$ and an upper bound for $Var(Q_{m,n,p})$ in two steps.

\medskip

\noindent{\bf Step 1: the estimate of $\E(Q_{m,n,p})$}. For a $m\times m $ symmetric matrix $M$, {we use $\|M\|$ to denote its spectral norm.}
%we use $\|M\|$ to denote its Euclidean norm, that is, $\|M\|=[tr(M^2)]^{1/2}$.
Set $S_0=\{1,2,\cdots, m\}$. Review
$\omega_n$ in \eqref{green_fresh}. Since $\{\mathbf{1}_{A_S};\, S\subset\{1,...,p\}\ \mbox{with}\ |S|=m\}$ are identically distributed, we have
\begin{equation}\label{eq:eq-main}
  \E(Q_{m,n,p})=\binom{p}{m}\PP(A_{S_0})
  \geq \binom{p}{m}\PP(A_{S_0}\cap \mathcal{L}_{m,n,p}),
 \end{equation}
where
\beaa
& & \mathcal{L}_{m,n,p}:
 = \left\{\frac{\|W_{\{1,...,m\}}-nI_{m}\|}{\sqrt{n}}\leq s_{m,n,p}\right\}\ \ \mbox{and} \\
 & &s_{m,n,p}=\max\left\{10\sqrt{m\log p},\, 2\sqrt{m\log p}+\omega_n\right\}.
\eeaa
It is easy to check that Assumption 1 in \eqref{nice_birth_1} implies
\begin{equation}\lbl{through_water}
\frac{s_{m,n,p}}{\sqrt{n}}\to 0\ \ \mbox{and}\ \ \frac{\sqrt{m}s_{m,n,p}^3}{\sqrt{n\log p}}\to 0.
\end{equation}
Similar to Lemma \ref{lemma:wig-eigen-approx},  we need the following lemma, which quantifies the speed that a Wishart matrix converges to a Wigner matrix. The difference is that the spectral norm $\|\cdot\|$ is used  here instead of $\|\cdot\|_{\infty}$ in Lemma \ref{lemma:wig-eigen-approx}.

Write $W_{\{1,...,m\}}$ for $W_S$ above \eqref{eq:t-stat} with $S=\{1,2,\cdots, m\}$. Review that  the Wigner matrix $\tW_{\{1,...,m\}}=(\tilde{w}_{ij})_{m\times m}$, where $\tilde{w}_{ij}$'s are  as in \eqref{eq:dist-wig}.

\begin{lemma}\label{lemma:log-likelihood-ratio}
 Let $f_{m,n}(w)$ be the density function of $\frac{1}{\sqrt{n}}(W_{\{1,...,m\}}-nI_m)$ and $\tf_m(w)$ be the density function of $\tW_{\{1,...,m\}}$.
  If $m^3=o(n)$, then
\beaa
& & \log f_{m,n}(w)-\log \tf_m(w)\\
&=&o(1) +{O\left(m^{2}n^{-1/2}\|w\|+ m^2n^{-1}\|w\|^2+ m n^{-1/2}\|w\|^{3}\right)}
\eeaa
for all $m\times m$ symmetric matrix $w$ with  $\|w\|\leq \frac{2}{3}\sqrt{n}$.
%  \beaa
%& &    |\log f_{m,n}(w)-\log\tf_{m}(w)|\\
%&=& O\left(n^{-1/2}m^2\|w\|+n^{-1}m^2\|w\|_{2}+n^{-1/2}m\|w\|^3\right).
%   \eeaa
%for all $m\times m$ symmetric matrix $w$ with $\|w\|<\sqrt{n}$.
\end{lemma}
Below, we combine the above lemma and some change of measure arguments to obtain a lower bound of $\PP(A_{S_0}\cap \mathcal{L}_{m,n,p})$. Define a non-random set $B_{m,p}=\{w_{ij}: w_{ij}\geq\tau_{m,p,t}, 1\leq i\leq j\leq m\}$. By the first limit from \eqref{through_water}, $s_{m,n,p}\leq \frac{2}{3}\sqrt{n}$. Therefore, from Lemma \ref{lemma:log-likelihood-ratio} we have
\begin{equation}
  \begin{split}
    &\PP\left(A_{S_0}\cap\mathcal{L}_{m,n,p}
    \right)\\
    =&
    \int_{w\in B_{m,p},\|w\|\leq s_{m,n,p} } e^{\log f_{m,n}(w)}
    d w   \\
    = & \int_{w\in B_{m,p},\|w\|\leq s_{m,n,p}  }\tf_{m}(w)\cdot \exp\{
    \log f_{m,n}(w)-\log\tf_{m}(w)
    \} dw\\
    = &\exp\left\{o(1)+{O\left(\frac{m^{2}s_{m,n,p}}{\sqrt{n}}+\frac{m^2s_{m,n,p}^2}
    {n}+\frac{m s_{m,n,p}^3}{\sqrt{n}}\right)}\right\}\\
    & ~~~~~~~~~~~~~~~~~~~~~~~~~~~~~~~
    \cdot\PP(
    \tA_{\{1,...,m\}}\cap\tilde{\mathcal L}_{m,n,p}
    )\\
    \geq & \frac{1}{2}\cdot \exp\left\{{O\left(\frac{m^{2}s_{m,n,p}}{\sqrt{n}}+\frac{m^2s_{m,n,p}^2}
    {n}+\frac{m s_{m,n,p}^3}{\sqrt{n}}\right)}\right\}\\
    & ~~~~~~~~~~~~~~~~~~~~~~~~~~~~~~~
    \cdot \left[\PP(\tA_{\{1,...,m\}}) - \PP(\tilde{\mathcal L}_{m,n,p}^c
    )\right],
  \end{split}
\end{equation}
where $\tilde{A}_{\{1,...,m\}}$ is as in \eqref{eq:tas} with $S=\{1,\cdots, m\}$ and $\tilde{\mathcal L}_{m,n,p}
=\{
\|\tilde{W}_{\{1,...,m\}}\|\leq s_{m,n,p}\}.
$
Under Assumption 1 in \eqref{nice_birth_1}, evidently $\frac{m}{s_{m,n,p}^2}\to 0$ and $\frac{m}{\sqrt{n}\,s_{m,n,p}}\to 0$. This implies that
\beaa
& & \frac{m^{2}s_{m,n,p}}{\sqrt{n}}+\frac{m^2s_{m,n,p}^2}
    {n}+\frac{m s_{m,n,p}^3}{\sqrt{n}}\\
    &=&\frac{ms_{m,n,p}^3}{\sqrt{n}}
    \left(\frac{m}{s_{m,n,p}^2}+\frac{m}{\sqrt{n}\,s_{m,n,p}}+1\right)\\
    & = & O\left(\frac{ms_{m,n,p}^3}{\sqrt{n}}\right).
\eeaa
%\begin{equation}	O(m^2s_{m,n,p}/\sqrt{n}+m^2s_{m,n,p}^2/n+ms_{m,n,p}^3/\sqrt{n})=O(ms_{m,n,p}^3/\sqrt{n}).
%\end{equation}
Thus, we have
\begin{equation}\label{eq:pas-cm}
	\begin{split}
    \PP\left(A_{S_0}\cap\mathcal{L}_{m,n,p}
    \right)
     \geq  \frac{1}{2}\cdot {e^{O(m s_{m,n,p}^3/\sqrt{n})}}\left\{\PP(
    \tA_{\{1,...,m\}}) - \PP(\tilde{\mathcal L}_{m,n,p}^c
    )\right\}.
  \end{split}
\end{equation}
Obviously, $\E(Q_{m,n,p})=\binom{p}{m} \PP(
    A_{\{1,...,m\}})$. Recalling  $\tA_{\{1,...,m\}}$ and $\tilde{Q}_{m,p}$ as in \eqref{eq:tas} and \eqref{eq:def-q}, respectively, we see
that $\E(\tilde{Q}_{m,p})=\binom{p}{m} \PP(
    \tA_{\{1,...,m\}})$.  Thus, we further have from  \eqref{eq:eq-main} and \eqref{eq:pas-cm} that
    \begin{equation}\label{eq:eq-first-1}
    \begin{split}
    	\E(Q_{m,n,p})
    \geq \frac{1}{2}\cdot {e^{O(m s_{m,n,p}^3/\sqrt{n})}}\left\{\E(\tilde{Q}_{m,p}) -\binom{p}{m} \PP(\tilde{\mathcal L}_{m,n,p}^c
    )\right\}.
    \end{split}
    \end{equation}
To further obtain a lower bound of the above expression, we analyze each term on the right-hand side.
% First, under Assumption 1 and according to our choice of $s_{m,n,p}$, we can see that
% \begin{equation}\label{eq:eq-first-2}
%   O(m^2s_{m,n,p}/\sqrt{n}+m^2s_{m,n,p}^2/n+ms_{m,n,p}^3/\sqrt{n})=O(m s^3_{m,n,p}).
% \end{equation}
Recall the definition of $\varepsilon_{m,p,t}$ below \eqref{eq:tas}, we know $\varepsilon_{m,p,t}\in (0, 1)$. By \eqref{eq:etq-approx},
%according to \eqref{eq:etq-approx} and $0<\varepsilon_{m,p,t}<1$
\bea\label{eq:tq-rough-lower}
	\E(\tilde{Q}_{m,p})
	&\geq & \exp\left\{\varepsilon_{m,p,t} m\log p + O(m^2\log\log p )\right\}\nonumber\\
	&\geq & \exp\left\{\frac{\delta}{2}\sqrt{m\log p} + O(m^2\log\log p )\right\}\nonumber\\
& \geq & \exp\left\{\frac{\delta}{4}\sqrt{m\log p}\right\}
\eea
where the condition $m = o((\log p)^{1/3} /\log\log p)$ from Assumption 1 in \eqref{nice_birth_1} is essentially used in the last step.
%\beaa
%m = o\left(\min\left\{ \frac{(\log p)^{1/3} }{\log\log p },\; \frac{n^{1/4}}{(\log n)^{3/2}(\log p)^{1/2}}\right\}\right).
%\eeaa
%Next, we have the following upper bound on $\binom{p}{m} \PP(\tilde{\mathcal L}_{m,n,p}^c
%    )$.
Now,
\begin{equation}\label{eq:two-tails-wig}
  \begin{split}
   & \binom{p}{m} \PP\left(\tilde{\mathcal L}_{m,n,p}^c
    \right)\\
    &\leq p^m \PP\left(\|\tilde{W}_{\{1,...,m\}}\|\geq s_{m,n,p}\right)\\ 
    %\mbox{\red Xiaoou, we can use chi-squre estimate now, it is even easier}\\
    & \leq p^m \PP\left(\lambda_1(\tilde{W}_{\{1,...,m\}})\geq s_{m,n,p}\right)+p^m\PP\left(\lambda_m(\tilde{W}_{\{1,...,m\}}\leq -s_{m,n,p})\right)\\
    &=2 p^m \PP\left(\lambda_1(\tilde{W}_{\{1,...,m\}})\geq s_{m,n,p}\right),
  \end{split}
\end{equation}
where the fact that $\tilde{W}_{\{1,..,m\}}$ and $-\tilde{W}_{\{1,...,m\}}$ have the same distribution is used in the last step. The following lemma help us estimate the last probability.
\begin{lemma}\label{eq:wig-margin-tail_night}[Lemma 4.1 from \cite{Jiang:2015ir}]
%[Equation (4.8) from \cite{Jiang:2015ir}]\label{lemma:wigner-margin-tail}
Let $\tW_{\{1,...,m\}}$ be defined by $\tW_{S}$ above \eqref{eq:tt-stat} with $S=\{1,...,m\}$. Then there is a constant $\kappa>0$ such that
\begin{equation}
%	P\Big(\lambda_1(W_{\{1,...,m\}})\geq x \text{ or } \lambda_{m}(W_{\{1,...,m\}})\leq -x\Big)
%\leq  e^{-(x^2/4)+\kappa m\log x }	
%\leq  \exp\Big\{-\frac{x^2}{4}+\kappa m\log x  \Big\}
	 P\Big(\lambda_1(\tW_{\{1,...,m\}})\geq x \text{ or } \lambda_{m}(\tW_{\{1,...,m\}})\leq -x\Big)
	 \leq \kappa\cdot e^{-\frac{x^2}{4} + \kappa \sqrt{m}x}
\end{equation}
for all $x>0$ and all $m\geq 2$.
\end{lemma}
By letting $x=s_{m,n,p}$ in Lemma~\ref{eq:wig-margin-tail_night}, we have
\begin{equation}
  \PP\left(\lambda_1(\tilde{W}_{\{1,...,m\}})\geq s_{m,n,p}\right)
  \leq
  \exp\left\{
  -\frac{s_{m,n,p}^2}{4}+\kappa \sqrt{m} s_{m,n,p}\right\}.
\end{equation}
Combining the above inequality with \eqref{eq:two-tails-wig}, we arrive at
\begin{equation}
  \begin{split}
   \binom{p}{m} \PP(\tilde{\mathcal L}_{m,n,p}^c
    )
     \leq  2 \cdot\exp\left\{m\log p
  -\frac{s_{m,n,p}^2}{4}+\kappa \sqrt{m} s_{m,n,p}
  \right\}.
  \end{split}
\end{equation}
Since $s_{m,n,p}\geq 10\sqrt{m\log p}$, we know $m\log p
  -\frac{1}{4}s_{m,n,p}^2\leq -\frac{6}{25}s_{m,n,p}^2.$ Moreover, $\sqrt{m} s_{m,n,p}=o(s_{m,n,p}^2)$.
Consequently,
\begin{equation}\label{paint_today}
  \begin{split}
  \binom{p}{m} \PP\left(\tilde{\mathcal L}_{m,n,p}^c
    \right)
     \leq  \exp\left\{
  -\Big(\frac{6}{25}+o(1)\Big)s_{m,n,p}^2
  \right\}.
  \end{split}
\end{equation}
Comparing the above inequality with \eqref{eq:tq-rough-lower}, we arrive at
  \begin{equation}
  	\binom{p}{m} \PP\left(\tilde{\mathcal L}_{m,n,p}^c
    \right)=o(1)=o(\E(\tilde{Q}_{m,p})).
  \end{equation}
 This result, combined with \eqref{eq:eq-first-1}, gives
 \begin{equation}\label{eq:eq-lower-teq}
 	\E(Q_{m,n,p})
    \geq \frac{1}{3}\cdot e^{O(m s_{m,n,p}^3/\sqrt{n})}\E(\tilde{Q}_{m,p}),
 %
%    \exp\{O(ms_{m,n,p}^3/\sqrt{n}+1)\}\E(\tilde{Q}_{m,p}).
 \end{equation}
which joint with  \eqref{eq:etq-approx} concludes
 \bea\label{eq:eq-lower-1}
& & 	\E(Q_{m,n,p})\nonumber\\
&\geq & \frac{1}{3}\exp\big\{ [1- (1-\varepsilon_{m,p,t})^2] m\log p\nonumber\\
 &&~~~~~~~~~~~~~~~~~~+ O(m^2\log\log p + mn^{-1/2}s_{m,n,p}^3)\big\}.
 \eea
This completes our analysis for $\E(Q_{m,n,p})$.

\medskip

\noindent{\bf Step 2: the estimate of $Var(Q_{m,n,p})$}. Replacing ``$\tA_S$" in \eqref{eq:tas} with ``$A_S$" in \eqref{eq:q}, and using the same argument as obtaining \eqref{eq:var-q-split},  we have from Lemma~\ref{lemma:combine} that
\bea\label{eq:var-q-main}
& &   Var(Q_{m,n,p})\nonumber\\
&\leq &  \E(Q_{m,n,p})+ m \max_{l=1,...,m-1}p^{2m-l} \PP\left( A_{ \{1,...,m\}}\cap A_{ \{1,...,l,m+1,...,2m-l\}}\right).~~~
\eea
Now we bound the last term above. Review $\mathcal{L}_{2m,n,p}$ below \eqref{eq:eq-main}. Trivially,
%find an upper bound of the second term on the right-hand side of the above inequality. We start with the analysis of $\PP\left( A_{ \{1,...,m\}}\cap A_{ \{1,...,l,m+1,...,2m-l\}}\right)$.
\begin{equation}\label{eq:var-q-first-main}
  \begin{split}
    &\PP\left( A_{ \{1,...,m\}}\cap A_{ \{1,...,l,m+1,...,2m-l\}}\right)\\
    \leq & \PP\left( A_{ \{1,...,m\}}\cap A_{ \{1,...,l,m+1,...,2m-l\}}\cap \mathcal{L}_{2m,n,p}\right) + \PP(\mathcal{L}_{2m,n,p}^{c}).
  \end{split}
\end{equation}
By \eqref{paint_today},
%Using arguments similar to those for \eqref{eq:spec-bound-rough}, it is not hard to show that
\begin{equation}\label{eq:var-q-first-2}
	mp^{2m}\PP(\mathcal L_{2m,n,p}^c)= o(1).
\end{equation}
%%%%%%BELOW ARE MORE EXPLANATIONS%%%%%%%%
% For an upper bound of the second term on the right-hand side \eqref{eq:var-q-first-main}, we let
% $t'=s_{2m,n,p}-\sqrt{8m\log p}$ and note that $ \frac{4}{5}s_{2m,n,p}\leq t'\leq s_{2m,n,p}$ and $s_{2m,n,p}=o(\sqrt{n})$.
% Replacing $t$ by $t'$, and $m$ by $2m$ in \eqref{eq:case2-main}, \eqref{eq:case2-first}, and \eqref{eq:case2-second} and combining them, we arrive at
% \begin{equation}\label{eq:var-q-first-2}
% \begin{split}
% &mp^{2m}\PP(\mathcal L_{2m,n,p}^c)\\
%   =& mp^{2m}\PP\left(\frac{\|W_{\{1,...,2m\}}-nI_{2m}\|}{\sqrt{n}}\geq t'+\sqrt{8m\log p}\right)\\
%   \leq & \exp\left\{-\frac{1}{4}\sqrt{m\log p} t'\right\}+\exp\left\{
%   -\frac{1}{2}(1-\log 2)(1+o(1))mn
%   \right\}\\
%   \leq& \exp\left\{-\frac{1}{5}\sqrt{m\log p} s_{2m,n,p}\right\}+\exp\left\{
%   -\frac{1}{4}(1-\log 2)mn
%   \right\}\\
%   = & o(1).
% \end{split}
% \end{equation}
%
%%%%%%%%END OF EXPLANATION
%Similar to \eqref{eq:pas-cm}, we obtain the following upper bound for the first term on the right-hand side of \eqref{eq:var-q-first-main},
Let $f_{2m,n}(w)$ be the density function of
%$\check{W}=(\check{w}_{ij})_{1\leq i, j\leq 2m}=$
$\frac{1}{\sqrt{n}}(W_{\{1,...,2m\}}-nI_{2m})$ and $\tf_{2m}(w)$ be the density function of $\tW_{\{1,...,2m\}}$. Review \eqref{eq:as}. Define (non-random) set
%Define a set $B_{m,p}=\{w_{ij}: w_{ij}\geq\tau_{m,p,t}, 1\leq i\leq j\leq m\}$.
\beaa
& &	 B_{S}=\big\{(w_{ij})_{i,j\in S};\, w_{ij}\geq \tau_{m,p,t}\text{ for all } i,j\in S \text{ with } i\leq j
\big\}.
\eeaa
Then,
\beaa
& & \PP\left( A_{ \{1,...,m\}}\cap A_{ \{1,...,l,m+1,...,2m-l\}}\cap \mathcal{L}_{2m,n,p}\right)\\
&=& \int_{B}\tf_{2m}(w)\cdot \exp\{
    \log f_{2m,n}(w)-\log\tf_{2m}(w)
    \} dw
\eeaa
where $B:=B_{\{1,\cdots,m\}}\cap B_{\{1,...,l,m+1,...,2m-l\}}\cap\{\|w\|\leq s_{2m,n,p}\}.$
By Lemma \ref{lemma:log-likelihood-ratio} and by a change-measure argument similar to the one getting \eqref{eq:pas-cm}, we see
\begin{equation}\label{eq:var-q-first-1}
\begin{split}
  &\PP\left( A_{ \{1,...,m\}}\cap A_{ \{1,...,l,m+1,...,2m-l\}}\cap \mathcal{L}_{2m,n,p}\right)\\
  %\leq & e^{O(m^2s_{2m,n,p}/\sqrt{n}+m^2s_{2m,n,p}^2/n+ms_{2m,n,p}^3/\sqrt{n})}\PP(\tA_{\{1,...,m\}}\cap \tA_{\{1,...,l,m+1,...,2m-l\}})\\
   \leq & 2\cdot e^{O(ms_{m,n,p}^3/\sqrt{n})}\PP(\tA_{\{1,...,m\}}\cap \tA_{\{1,...,l,m+1,...,2m-l\}}).
\end{split}
\end{equation}
The benefit of the above step is transferring the probability on the Wishart matrix to that on the Wigner matrix up to a certain error.  Combining \eqref{eq:var-q-first-main}-\eqref{eq:var-q-first-1}, we have
\begin{equation}
\begin{split}
  &m p^{2m-l}
  \PP\left( A_{ \{1,...,m\}}\cap A_{ \{1,...,l,m+1,...,2m-l\}}\right)\\
  \leq & 2\cdot e^{O(ms_{m,n,p}^3/\sqrt{n})} \cdot mp^{2m-l}\PP(\tA_{\{1,...,m\}}\cap \tA_{\{1,...,l,m+1,...,2m-l\}})
  + o(1).
  %\\
  % \leq & e^{O(ms_{m,n,p}^3/\sqrt{n})} p^{2m-l}\PP(\tA_{\{1,...,m\}}\cap \tA_{\{1,...,l,m+1,...,2m-l\}})\\
  % & + \exp\{-\frac{1}{5}\sqrt{m\log p} s_{2m,n,p}\}+\exp\{
  % -\frac{1}{4}(1-\log 2)mn
  % \}.
\end{split}
\end{equation}
% \begin{equation}
% \begin{split}
%   &m p^{2m-l}
%   \PP\left( A_{ \{1,...,m\}}\cap A_{ \{1,...,l,m+1,...,2m-l\}}\right)\\
%   \leq & p^{2m-l}e^{O(ms_{m,n,p}^3/\sqrt{n})}\PP(\tA_{\{1,...,m\}}\cap \tA_{\{1,...,l,m+1,...,2m-l\}})
%   + p^2m \PP(\mathcal L_{2m,n,p}^c)\\
%   \leq & e^{O(ms_{m,n,p}^3/\sqrt{n})} p^{2m-l}\PP(\tA_{\{1,...,m\}}\cap \tA_{\{1,...,l,m+1,...,2m-l\}})\\
%   & + \exp\{-\frac{1}{5}\sqrt{m\log p} s_{2m,n,p}\}+\exp\{
%   -\frac{1}{4}(1-\log 2)mn
%   \}.
% \end{split}
% \end{equation}
Combining this with \eqref{eq:var-q-main}, we have
\begin{equation}
%  \begin{split}
    Var(Q_{m,n,p})
    \leq  \E(Q_{m,n,p})+ 2\cdot e^{O(s_{m,n,p}^3/\sqrt{n})}\cdot G_m +o(1).
    %+ \exp\{-\frac{1}{5}\sqrt{m\log p} s_{2m,n,p}+\log m\}+\exp\{
  %-\frac{1}{4}(1-\log 2)mn+\log m
  %\}.
%  \end{split}
\end{equation}
where
\beaa
G_m:=\max_{l=1,...,m-1}\left\{mp^{2m-l}\PP(\tA_{\{1,...,m\}}\cap \tA_{\{1,...,l,m+1,...,2m-l\}})\right\}.
\eeaa
Thus,
\begin{equation}\label{eq:vq-eq-ratio-1}
	\begin{split}
		  \frac{Var(Q_{m,n,p})}{\E(Q_{m,n,p})^2}
\leq &  e^{O(ms_{m,n,p}^3/\sqrt{n})} \big[\E(Q_{m,n,p})\big]^{-2} \cdot G_m\\
&~~~~+ \big[{\E(Q_{m,n,p})}\big]^{-1}+o\Big(\big[\E(Q_{m,n,p})\big]^{-2}\Big).
	\end{split}
\end{equation}
According to \eqref{eq:ratio-intermed} and \eqref{eq:eq-lower-teq}, the first term on the right-hand side of the above inequality is no more than
\begin{equation}
\begin{split}
%	&e^{O(ms_{m,n,p}^3/\sqrt{n})} \big(\E(Q_{m,n,p})^{-2}\big) m\max_{l=1,...,m-1}p^{2m-l}\PP(\tA_{\{1,...,m\}}\cap \tA_{\{1,...,l,m+1,...,2m-l\}})\\
%	\leq &
9\cdot\exp\left\{-\Big[1-\frac{1}{m}(1-\varepsilon_{m,p,t})^2\Big]\log p+O\Big(m^2\log\log p+\frac{ms_{m,n,p}^3}{\sqrt{n}}\Big)\right\}.
\end{split}
	\end{equation}
Notice that $1-\varepsilon_{m,p,t}\leq 1$ and $m\geq 2$ and $O(m^2\log\log p+mn^{-1/2}s_{m,n,p}^3)=o(\log p)$. Thus, the above display further implies
\begin{equation}\label{eq:ratio-vq-eq-2}
\begin{split}
	e^{O(s_{m,n,p}^3/\sqrt{n})} \big(\E(Q_{m,n,p})\big)^{-2} \cdot G_m
	\leq  \exp\left\{-\Big(\frac{1}{2}+o(1)\Big)\log p\right\}.
\end{split}
	\end{equation}
We next study the last two terms from \eqref{eq:vq-eq-ratio-1}.

By the condition $m=o((\log p)^{1/3}/\log\log p)$ from Assumption 1 in \eqref{nice_birth_1} and the second limit in \eqref{through_water},
\begin{equation}\lbl{Hosana}
O\left(\frac{m s_{m,n,p}^3}{\sqrt{n}}+m^2\log\log p\right)=o\big(\sqrt{m\log p}\,\big).
\end{equation}
Recall \eqref{talk_hire}. It is readily seen that
 $[1- (1-\varepsilon_{m,p,t})^2] m\log p\geq \varepsilon_{m,p,t} m\log p\geq \frac{t}{2}\sqrt{m\log p}$. Consequently, it is known from \eqref{eq:eq-lower-1} that
\beaa
\E(Q_{m,n,p})
\geq  \frac{1}{3}\cdot \exp\Big\{ \frac{t}{2}\sqrt{m\log p}\Big\}
%[1- (1-\varepsilon_{m,p,t})^2] m\log p\nonumber\\
% &&~~~~~~~~~~~~~~~~~~+ O(m^2\log\log p + n^{-1/2}s_{m,n,p}^3)\big\}.
\eeaa
uniformly over $\delta \leq t\leq 2\sqrt{m\log p}- m\sinf$.  Therefore, we conclude  from \eqref{eq:eq-lower-1} and \eqref{Hosana} that
\begin{equation}\label{eq:ratio-vq-eq-3}
\begin{split}
	&\big({\E(Q_{m,n,p})}\big)^{-1}+o\Big(\big({\E(Q_{m,n,p})\big)^{-2}}\Big)\\
	=& (1+o(1))\big({\E(Q_{m,n,p})}\big)^{-1}\\
	\leq &3\cdot \exp\left\{-\Big(\frac{1}{2}+o(1)\Big)t\sqrt{m\log p} \right\}.	
\end{split}
\end{equation}
Combining \eqref{eq:vq-eq-ratio-1}-\eqref{eq:ratio-vq-eq-3}, we see
\begin{equation}
	\begin{split}
		  &\frac{Var(Q_{m,n,p})}{\E(Q_{m,n,p})^2}\\
\leq & \exp\Big\{-\Big(\frac{1}{2}+o(1)\Big)\log p\Big\}
+ 3\cdot \exp\Big\{-\Big(\frac{1}{2}+o(1)\Big)t\sqrt{m\log p} \Big\}.
	\end{split}
\end{equation}
By \eqref{eq:t-moment-bound} and the above inequality,
\begin{equation}
\begin{split}
	 &\PP\left(T_{m,n,p}\leq 2\sqrt{m\log p}-t\right)\\
	 \leq & \exp\Big\{-\Big(\frac{1}{2}+o(1)\Big)\log p\Big\}
+ 3\cdot \exp\Big\{-\Big(\frac{1}{2}+o(1)\Big)t\sqrt{m\log p} \Big\}.
\end{split}
\end{equation}
Finally, from the inequality $t^2\leq 2 e^t$ we have that
\begin{equation}\label{eq:scenario1-final}
\begin{split}
	  \limsup_{n\to\infty}\sup_{\delta\leq t\leq2\sqrt{m\log p}- m\sinf} e^{\alpha t}t^2\PP\left(T_{m,n,p}\leq 2\sqrt{m\log p}-t\right)=0
\end{split}
\end{equation}
for any $\alpha>0$ and $\delta>0$.

\medskip

\noindent{\bf Scenario 2: $ t>2\sqrt{m\log p}- m\sinf$.} Review \eqref{eq:t-stat}. By the fact that $\lambda_1(M)\geq \max_{1\leq i \leq m}M_{ii}$ for any non-negative definite matrix $M=(M_{ij})_{m\times m}$, we have
\begin{equation}
  T_{m,n,p}
  \geq \max_{S\subset\{1,...,p\},|S|=m}\lambda_1(W_{S})\geq \max_{1\leq i\leq p} W_{ii},
\end{equation}
where $W_{ii}=\sum_{j=1}^nx_{ji}^2$ and $\{W_{ii};\, 1\leq i \leq m\}$ are i.i.d. random variables.  Thus, by independence,
\bea\label{eq:bound-t-iid}
 &     & \PP\left(\frac{T_{m,n,p}-n}{\sqrt{n}}\leq 2\sqrt{m\log p}-t\right)\nonumber\\
 &\leq & \PP\left(\frac{W_{11}-n}{\sqrt{n}}\leq 2\sqrt{m\log p}-t\right)^p.
\eea
Note that $W_{11}=\sum_{j=1}^n x_{j1}^2$ is a sum of i.i.d. random variables with $Var(x_{11})=2$ and $\E(x_{11}^6)<\infty$.
We discuss two situations: $2\sqrt{m\log p}-m\sinf\leq t\leq 4\sqrt{m\log p}$ and $t\geq 4\sqrt{m\log p}$.
%For $t\leq 4\sqrt{m\log p}$, $-\sqrt{2m\log p} \leq\sqrt{2m\log p}-t/\sqrt{2}\leq \sqrt{2m\log p}$.

Assuming $2\sqrt{m\log p}-m\sinf\leq t\leq 4\sqrt{m\log p}$ for now. Recalling $\Phi(x)=(2\pi)^{-1/2}\int_{-\infty}^xe^{-t^2/2}\,dt$,
we get from  the Berry-Essen Theorem that
\begin{equation}
\begin{split}
	  &\PP\left(\frac{W_{11}-n}{\sqrt{n}}\leq 2\sqrt{m\log p}-t\right)\\
	  &\leq \Phi\left(\sqrt{2m\log p}-\frac{t}{\sqrt{2}}\right)+\frac{\kappa}{\sqrt{n}}\\
  \leq &2\cdot\max\left\{\Phi\left(\sqrt{2m\log p}-\frac{t}{\sqrt{2}}\right),\frac{\kappa}{\sqrt{n}}\right\}
\end{split}
\end{equation}
for some constant $\kappa>0$.
Combine the above inequalities with \eqref{eq:very-small-t-first-case} to see
\begin{equation}
  \PP\left(\frac{T_{m,n,p}-n}{\sqrt{n}}\leq 2\sqrt{m\log p}-t\right)\leq 2\cdot\max\left\{
  e^{-mp^{0.9}},e^{-\frac{p\log n}{2}(1+o(1))}
  \right\}.
 \end{equation}
By \eqref{nice_birth_1}, $\sqrt{m\log p}\leq \log p$. It is easy to check
\bea\label{eq:scenario2-first-final}
&& \lim_{n\to\infty} \sup_{2\sqrt{m\log p}-m\sinf\leq t\leq 4\sqrt{m\log p}} e^{\alpha t}t^2 \PP\left(\frac{T_{m,n,p}-n}{\sqrt{n}}\leq 2\sqrt{m\log p}-t\right) \nonumber\\
& = & 0.
\eea

We proceed to the second situation: $t\geq 4\sqrt{m\log p}$. In this case, $2\sqrt{m\log p}-t\leq -2\sqrt{m\log p}$.
By Lemma 1 from \cite{laurent2000adaptive},
  \begin{equation}
    \PP(W_{11}-n\leq -2\sqrt{nx})\leq e^{-x}
  \end{equation}
for any $x>0$.  Thus,
  \begin{eqnarray*}
    \PP\left(\frac{W_{11}-n}{\sqrt{n}}\leq 2\sqrt{m\log p}-t\right)
     &\leq & \exp\left\{
    - \left(\frac{t}{2}-\sqrt{m\log p}\right)^2
    \right\}\\
    &\leq & \exp\left\{
    -\frac{t^2}{16}
    \right\}.
  \end{eqnarray*}
This inequality and \eqref{eq:bound-t-iid} yield
\begin{equation}
  \PP\left(
  \frac{T_{m,n,p}-n}{\sqrt{n}}\leq 2\sqrt{m\log p} - t
  \right)
  \leq \exp\left\{
    -\frac{pt^2}{16}
    \right\}.
\end{equation}
Consequently,
\begin{equation}
\begin{split}
	  &\sup_{t\geq 4\sqrt{m\log p}}e^{
  \alpha t
  }t^2 \PP\left(
  \frac{T_{m,n,p}-n}{\sqrt{n}}\leq 2\sqrt{m\log p} - t
  \right)\\
  \leq & \sup_{t\geq 4\sqrt{m\log p}}\exp\left\{
    -\frac{pt^2}{16}+\alpha t +2\log t
    \right\} \\
    \leq & \exp\left\{
- m p (\log p) (1+o(1))
    \right\}.
\end{split}
\end{equation}
Hence,
\begin{equation}\label{eq:scenario2-second-final}
\begin{split}
	  \lim_{n\to\infty}\sup_{t\geq 4\sqrt{m\log p}}e^{
  \alpha t
  }t^2 \PP\left(
  \frac{T_{m,n,p}-n}{\sqrt{n}}\leq 2\sqrt{m\log p} - t
  \right)=0.
\end{split}
\end{equation}
By collecting  \eqref{eq:scenario1-final}, \eqref{eq:scenario2-first-final} and \eqref{eq:scenario2-second-final} together, we arrive at
\begin{equation}
  \lim_{n\to\infty}\sup_{t\geq \delta} e^{ \alpha t
  }t^2
  \PP\left( \frac{T_{m,n,p}-n}{\sqrt{n}}\leq 2\sqrt{m\log p} - t
  \right)=0.
 \end{equation}
 The proof is completed.
 %Combining this result with \eqref{eq:upper-final}, we arrive at
% \begin{equation}
%  \lim_{n,p\to\infty}\sup_{t\geq \delta} e^{ \alpha t
%  }t^2
%  \PP\left( \left|\frac{T_{m,n,p}-n}{\sqrt{n}}- 2\sqrt{m\log p}\right|\geq  t
%  \right)=0.
% \end{equation}
% Note that the above display have a similar form as \eqref{eq:decay-rate-wig-final}. Thus, by using similar arguments as those after \eqref{eq:decay-rate-wig-final} in the proof of Theorem~\ref{thm:wigner}, we complete the proof.
\end{proof}
%{\color{red}STOP HERE}

\subsubsection{Proofs of  Theorem \ref{thm:general-wig} and Remark~\ref{remark:implications}}\label{Sec_coro:smaller}

The following lemma serves the proof of Theorem \ref{thm:general-wig}. Its own proof is placed in Appendix.
\begin{lemma}\label{lemma:MDP-wigner}
Let $\tilde{W}=\tilde{W}_{m\times m}$ be as defined in \eqref{eq:dist-general-wig} with $0\leq \eta \leq 2$. Then
\begin{equation}
 	\PP(\lambda_1(\tilde{W})\geq x)
 	\leq \frac{m^{1.5}\log m}{\delta^{m}}\cdot\exp\Big\{-\frac{(x-2r\delta)^2}{2m^{-1} (\eta-2)+4}\Big\} + 2\cdot e^{-r^2/8}
 \end{equation}
for all $r\geq 4m$, $\delta\in (0,1)$ and $x>2r\delta+1$.
\end{lemma}

\begin{proof}[Proof of Theorem~\ref{thm:general-wig}]
%Similar to the proof of Theorem~\ref{thm:wigner}. We develop upper and lower bound for the statistic $\tilde{T}_{m,p}$ here. For the upper bound, it is sufficient to show that for each $\varepsilon>0$
For any $0<\varepsilon<1$, we first show that
\begin{equation}\label{eq:upper-general-each} \PP\left(\lambda_1(\tW_{\{1,...,m\}})\geq(1+\varepsilon)\sqrt{\{4m+2(\eta-2)\}\log p}\right)=o(p^{-m})
\end{equation}
by using Lemma \ref{lemma:MDP-wigner}. To do so, set $x=(1+\varepsilon)\sqrt{[4m+2(\eta-2)]\log p}$,\\
 $r=\sqrt{128m\log p}$ and $\delta=(8r)^{-1}\varepsilon\sqrt{[4m+2(\eta-2)]\log p}$. Rewrite $\delta$ such that
 \beaa
 \delta=\left(\frac{1}{64}\sqrt{\frac{2m+\eta-2}{m}}\right)\,\varepsilon.
 \eeaa
It is easy to check that the coefficient of $\varepsilon$ is always sitting in $[1/64, 2/64]$ for any $m\geq 2$ and $\eta\in [0, 2]$. This,  the fact that $\sup_{k\geq 2}(k^{1.5}\log k)\cdot \delta^k<\infty$ and
the definition of $r$ lead  to
 \begin{equation}\lbl{cup_victor}
 	 \frac{m^{1.5}\log m}{ \delta^{m}}
 	= O(\varepsilon^{-2m}) \text{ and } e^{-r^2/8} = o(p^{-6m}).
 \end{equation}
We can see that
\beaa
x-r\delta=\Big(1+\frac{7}{8}\varepsilon\Big)\cdot \sqrt{[4m+2(\eta-2)]\log p}.
\eeaa
It follows that
 \begin{equation}
 \begin{split}
 		&\exp\left\{-\frac{(x-2r\delta)^2}{2m^{-1} (\eta-2)+4}\right\}\\
 	\leq &\exp\left\{
-\frac{\big(1+\frac{\varepsilon}{2}\big)^2 [4m+2(\eta-2)]\log p }{2m^{-1} (\eta-2)+4}\right\}\\
%= & \exp\left\{-\big(1+\frac{\varepsilon}{2}\big)^2 m\log p\right\}\\
 	=&
 	p^{-[1+(\varepsilon/2)]^2m}.
 \end{split}
 \end{equation}
 This and \eqref{cup_victor} implies \eqref{eq:upper-general-each}. Consequently,
% Combining the above three displays, we complete the proof for \eqref{eq:upper-general-each}. Thus, we proved that
 \begin{equation}
 \begin{split}
 	& \PP\left(\tilde{T}_{m,p}
 	\geq(1+\varepsilon)\sqrt{[4m+2(\eta-2)]\log p}\right)\\
 	\leq & p^m
 	\PP\left(\lambda_1(\tW_{\{1,...,m\}})\geq(1+\varepsilon)\sqrt{[4m+2(\eta-2)]\log p}\right)\to 0.
 \end{split}
 \end{equation}
To complete the proof, it is enough to check that
\begin{equation}\label{new_peral}
 	 \PP\left(\tilde{T}_{m,p}
 	< (1-\varepsilon)\sqrt{[4m+2(\eta-2)]\log p}\right)\to 0
 \end{equation}
for each $\varepsilon\in (0, 1).$
For notational simplicity, let $K_m=4m+2(\eta-2)$ and  $\tau_{m,p}=\log p/K_m.$  Similar to the proof of Theorem~\ref{thm:wigner}, define
\beaa
% 	\tA_{S}=\{ \tW_{kk}\geq (1-\varepsilon)\sqrt{\frac{4\log p}{m}}, \tW_{ij}\geq (1-\varepsilon)\sqrt{\frac{4\log p}{m}} \text{ for all } i,j,k\in S \text{ and } i<j
% \}.
\tA_{S}&=& \Big\{  \tW_{ii}\geq 2(1-\varepsilon)\eta \sqrt{\tau_{m,p}},\, \tW_{ij}\geq 4(1-\varepsilon)\sqrt{\tau_{m,p}}\\
&&~~~~~~~~~~~~~~~~~~~~~~~~~\text{ for all } i,j\in S \text{ and } i\leq j\Big\}
\eeaa
for each $S\subset\{1,...,p\}$ with $|S|=m$.
We next compute $\PP(\tA_{S_0})$ and $\PP(\tA_{S_0}\cap\tA_{S_1})$, respectively, where $S_0=\{1,...,m\}$ and $S_1=\{1,...,l,m+1,...,2m-l\}$.
%We start with $\PP(\tA_{S_0})$.
By independence,
%Similar to the proof of Theorem~\ref{thm:wigner}, it is not hard to see
\beaa
	\PP\left(\tA_{S_0}\right)
	&=&\prod_{i=1}^m \PP\left(\tW_{ii}\geq 2(1-\varepsilon)\eta\sqrt{\tau_{m,p}} \right)\cdot \\
&& ~~~~~~~~~~~\prod_{1\leq i<j\leq m} \PP\left(
 	\tW_{ij}\geq 4(1-\varepsilon)\sqrt{\tau_{m,p}}
 	\right).
\eeaa
Since $ \tW_{ii}\sim N(0,\eta)$ and $\tW_{ij}\sim N(0,1) $ for all $i\ne j$, we further have
\begin{equation}
	\PP\left(\tA_{S_0}\right)
	= \bar{\Phi}\left( 2(1-\varepsilon)\sqrt{\eta\tau_{m,p}}\,\right)^m  \bar{\Phi}\left(
 	 4(1-\varepsilon)\sqrt{\tau_{m,p}}\,
 	\right)^{\frac{m(m-1)}{2}},
\end{equation}
where $\bar{\Phi}(x)=(2\pi)^{-1/2}\int_x^{\infty}e^{-t^2/2}\,dt$ for $x\in \mathbb{R}.$
Similar to \eqref{eq:intersection},
%for $\tilde{A}_{S_0}\cap \tilde{A}_{S_1}$, we have
\bea\label{eq:eta-cross-prob}
&&	\PP\left(\tilde{A}_{S_0}\cap \tilde{A}_{S_1}\right)\nonumber\\
&=& \bar{\Phi}\left(2(1-\varepsilon)\sqrt{\eta\tau_{m,p}}\right)^{2m-l}\bar{\Phi}\left(4(1-\varepsilon)\sqrt{\tau_{m,p}}\right)^{\frac{m(m-1)}{2}\cdot 2-\frac{l(l-1)}{2}}.
\eea
From \eqref{eq:tail-sdn}, $\log \bar{\Phi}(x)=-  \frac{x^2}{2}-\log(x)-\log \sqrt{2\pi}+o(1)$ as $x\to\infty$. Then,
%we obtain the following approximations for $\PP\left(\tA_{S_0}\right)$.
\begin{equation}
	\log \PP(\tA_{S_0})
	= -2m (1-\varepsilon)^2 \eta \tau_{m,p}-4m(m-1)(1-\varepsilon)^2 \tau_{m,p}+ R_{m,p},
\end{equation}
where
\begin{equation}
\begin{split}
	R_{m,p}:= &- m \log\Big[
2(1-\varepsilon)\sqrt{\eta\tau_{m,p}}
	\Big]-m\log\sqrt{2\pi}\\
	&-\frac{m(m-1)}{2}\cdot \log\Big[
4(1-\varepsilon)\sqrt{\tau_{m,p}}
\Big]-\frac{m(m-1)}{2}\cdot \log\sqrt{2\pi} +o(m^2).
\end{split}
	\end{equation}
%For the dominating term in the above equation, we obtain
Notice
\begin{equation}
\begin{split}
		&-2m (1-\varepsilon)^2 \eta \tau_{m,p}-4m(m-1)\cdot(1-\varepsilon)^2 \tau_{m,p}\\
	= & -m(1-\varepsilon)^2 \tau_{m,p}(2\eta + 4m-4)\\
	= & -m(1-\varepsilon)^2 \tau_{m,p}K_m\\
	= & -(1-\varepsilon)^2m\log p.
\end{split}
\end{equation}
Similar to \eqref{eq:gaussian-approx}, we obtain that $R_{m,p}=O(m^2\log\log p)$. Thus,
\begin{equation}
	\log \PP\left(\tA_{S_0}\right) = -(1-\varepsilon)^2m\log p + O(m^2\log\log p).
\end{equation}
By the same argument as obtaining \eqref{eq:etq-approx}, we see
\begin{equation}\label{eq:eta-single-approx}
	\log \left[\binom{p}{m}\PP\left(\tA_{S_0}\right) \right]=
	[1-  (1-\varepsilon)^2]m\log p+ O(m^2\log\log p).
\end{equation}
In particular the above  goes to infinity as $n\to\infty$.
%For $\PP\left(\tilde{A}_{S_0}\cap \tilde{A}_{S_1}\right)$, we apply the tail approximation for
By \eqref{eq:tail-sdn} and \eqref{eq:eta-cross-prob},
% with a similar analysis. We obtain
\begin{equation}
\begin{split}
		&\log \PP\left(\tilde{A}_{S_0}\cap \tilde{A}_{S_1}\right)\\
	= &-2(2m-l)(1-\varepsilon)^2 \eta \tau_{m,p} - 8\left[m(m-1)-\frac{1}{2}l(l-1)\right](1-\varepsilon)^2 \tau_{m,p}\\
&~~~~~~~~~~~~~~~~~~~~~~~~~~~~~~~~~ + O(m^2\log\log p).
\end{split}
\end{equation}
The right hand side above without the term ``$O(m^2\log\log p)$" is identical to
%The dominating term in the above display is simplified as follows
\begin{equation}
	\begin{split}
%		&-(2m-l)(1-\varepsilon)^2 2\eta \tau_{m,p} - [m(m-1)-\frac{l(l-1)}{2}](1-\varepsilon)^2 8\tau_{m,p}\\
		 &  -(1-\varepsilon)^2 \tau_{m,p}
		\Big[
 2(2m-l)\eta + 8m(m-1) - 4l(l-1)
		\Big]\\
		= &  -(1-\varepsilon)^2 \tau_{m,p} \Big[
		2m (2\eta + 4m-4)-4l^2 +(4-2\eta)l
		\Big]\\
		= & -(1-\varepsilon)^2 \tau_{m,p}\Big[
		2m K_m -4l^2 +(4-2\eta)l
		\Big]\\
		= & - 2(1-\varepsilon)^2 m \log p + (1-\varepsilon)^2 (\log p)\cdot K_m^{-1} \cdot(4l- 4+2\eta)l.
	\end{split}
\end{equation}
%Use the above approximations, we further obtain
The above two assertions yield
\begin{equation}\label{eq:eta-cross-approx}
\begin{split}
		&\log \left[m p^{2m-l}\PP\left(\tilde{A}_{S_0}\cap \tilde{A}_{S_1}\right)\right]\\
	= &[2m-l- 2(1-\varepsilon)^2 m] \log p + (1-\varepsilon)^2 (\log p)\cdot K_m^{-1}\cdot l(4l- 4+2\eta)\\
&~~~~~~~~~~~~~~~~~~~~~~~~~~~~~~~~~~~+O(m^2\log\log p)\\
	= & (\log p) \left\{ 2\left[1-(1-\varepsilon)^2\right]m +\left[-1+ (1-\varepsilon)^2\cdot K_m^{-1}\cdot(4l- 4+2\eta) \right] l
	\right\}\\
&~~~~~~~~~~~~~~~~~~~~~~~~~~~~~~~~~~+O(m^2\log\log p).
\end{split}
\end{equation}
Let us take a closer look at  the above display. For $1\leq l\leq m-1$,
\begin{equation}
	K_m^{-1}(4l- 4+2\eta)
	= \frac{4(l-1)+2\eta}{4(m-1)+2\eta}
%\frac{4(m-1)+2\eta-4}{4(m-1)+2\eta}
<1.
\end{equation}
Thus, for $1\leq l\leq m-1$ and $0<\varepsilon<1$,
\begin{equation}
	\left[-1+ (1-\varepsilon)^2K_m^{-1}(4l- 4+2\eta) \right] l
	\leq [-1 + (1-\varepsilon)^2]\cdot l \leq -\varepsilon.
\end{equation}
Combining this with \eqref{eq:eta-cross-approx}, we obtain that
\beaa\label{eq:eta-cross-approx1}
	&&\log \left[m p^{2m-l}\PP\left(\tilde{A}_{S_0}\cap \tilde{A}_{S_1}\right)\right]\\
	&\leq & 2\left[1-(1-\varepsilon)^2\right]m\log p-\varepsilon \log p + O(m^2\log\log p)
\eeaa
% Note that the above two equations are the same as \eqref{eq:pas} and \eqref{eq:intersection} with ``$\varepsilon_{p,m,t}$" being replaced by ``$\varepsilon$".
% Thus, $\PP(\tilde{A}_{S_0}\cap \tilde{A}_{S_1})$ and its analogue have the same asymptotic approximation.
uniformly for $1\leq l \leq m-1$. Define $\tilde{Q}_p=\sum_{S\subset\{1,...,p\},|S|=m}{\mathbf{1}}_{\tilde{A}_{S}}$. From \eqref{eq:eta-single-approx},  $$\E\tilde{Q}_p=\binom{p}{m}\PP\left(\tA_{S_0}\right)\to\infty$$ as $n\to\infty$.
Moreover, we see from \eqref{eq:eta-single-approx} and \eqref{eq:eta-cross-approx} that
\beaa
&&	\Big(\E\tilde{Q}_p \Big)^{-2}\max_{1\leq l\leq m-1} m p^{2m-l}\PP\left(\tilde{A}_{S_0}\cap \tilde{A}_{S_1}\right)\\
& \leq & \exp\{-\varepsilon\log p+O(m^2\log\log p)\}\\
& \to & 0.
\eeaa
By Lemma \ref{lemma:combine} and a similar argument to \eqref{eq:var-q-split}, we get
\begin{equation}
  \begin{split}
  Var(\tQ_{p})
  \leq   \E\tQ_{p}+  \max_{1\leq l \leq m-1}mp^{2m-l} \PP\left(\tilde{A}_{S_0}\cap \tilde{A}_{S_1}\right).
%  \PP\left( \tA_{ \{1,...,m\}}\cap \tA_{ \{1,...,l,m+1,...,2m-l\}}\right).
\end{split}
\end{equation}
This and the above two limits imply
$\frac{Var(\tilde{Q}_p)}{(\E\tilde{Q}_p)^2}\to 0$. As a result, $\lim_{p\to\infty}\PP(\tilde{Q}_p=0)=0$ by \eqref{sun_set_flower}.
According to \eqref{eq:lb-spec}, if there exists $S_0\subset\{1,...,p\}$ such that $|S_0|=m$ and $\tilde{A}_{S_0}$ occurs, then
\begin{equation}
\begin{split}
	\tilde{T}_{m,p}\geq & \lambda_1(\tilde{W}_{S_0})\\
	\geq &
\frac{1}{m}\Big(
2m(1-\varepsilon)\eta\sqrt{\tau_{m,p}}+ 4m(m-1)(1-\varepsilon)\sqrt{\tau_{m,p}}
\Big)\\
=&(1-\varepsilon)\sqrt{[4m+2(\eta-2)]\log p}.
\end{split}
\end{equation}
%Similar to \eqref{eq:tt-tq}, we know
Therefore,
\begin{equation}
\PP\left(\tilde{T}_{m,p}
 	< (1-\varepsilon)\sqrt{\{4m+2(\eta-2\}]\log p}\right)\leq
  \PP\left(\tQ_{p}=0\right)\to 0.
\end{equation}
This implies \eqref{new_peral}. The proof is finished.

\end{proof}

\begin{proof}[Proof of Remark~\ref{remark:implications}]
	These results are direct consequences of the following lemma, whose proof is given in Appendix~\ref{fish_dead}.
\end{proof}
	\begin{lemma}\label{lemma:convergence-concepts}
	Let $\{Z_p\}_{p\geq 1}$ be a sequence of non-negative random variables. Consider the following statements.
	\begin{itemize}
		\item [(i)]
			$
			\lim\limits_{p\to\infty}\E\left[
   e^{\alpha Z_p}\mathbf{1}_{\{Z_p\geq \delta\} }
  \right]=0
			$ for all $\alpha>0$ and $\delta>0$.
			\item [(ii)]  $
			\lim\limits_{p\to\infty}\E(
   e^{\alpha Z_p}  )=1
			$ for all $\alpha>0$.
			\item [(iii)]  $
			\lim\limits_{p\to\infty}\E(
   Z_p^{\alpha}  )=0
			$ for all $\alpha>0$.
			\item [(iv)]  $
			\lim\limits_{p\to\infty}\PP(Z_p\geq \delta)=0
			$ for all $\delta>0$.
			\item [(v)]  $
			\lim\limits_{p\to\infty}{\rm Var}(Z_p)=0
			$
for all $\alpha>0$.
	\end{itemize}
	Then, (i)$\Longleftrightarrow$ (ii) $\implies$ (iii)  $\implies$ (iv) and (v).
	Here, ``A $\Longleftrightarrow$ B'' means two statements A and B are equivalent, and A $\implies$ B means statement A implies statement B.
\end{lemma}
%\end{proof}
\section*{Acknowledgment}
The research of Tony Cai was supported in part by NSF Grant DMS-1712735 and NIH grants R01-GM129781 and R01-GM123056.
Tiefeng Jiang is partially supported by NSF Grant DMS-1406279. Xiaoou Li is partially supported by NSF Grant DMS-1712657.

\bibliographystyle{apalike}
\bibliography{RadomMatrix.bib}

% \section*{Acknowledgements}
% And this is an acknowledgements section with a heading that was produced by the
% $\backslash$section* command. Thank you all for helping me writing this
% \LaTeX\ sample file. See \ref{suppA} for the supplementary material example.

% \begin{supplement}
% \sname{Supplement A}\label{suppA}
% %\stitle{Title of the Supplement A}
% % \slink[url]{http://www.e-publications.org/ims/support/dowload/imsart-ims.zip}
% \sdescription{The supplement material provides}
% \end{supplement}
\newpage
%\appendix

\appendix
\begin{center}
	{\Large \bf Appendix}
\end{center}
\label{sec:lemmaproof}
\numberwithin{lemma}{section}

There are two sections in this part. In Appendix \ref{bu_cucumber} we derive some results on Gamma functions, which will be used later on. The material in this part is independent of previous sections. In Appendix \ref{fish_dead} we will prove the lemmas appeared in earlier sections.

\section{Auxiliary results on Gamma functions}\label{bu_cucumber}
Recall the Gamma function $\Gamma(x)=\int_0^{\infty}t^{x-1}e^{-t}\,dt$ for $x>0.$
\begin{lemma}\label{lemma:wigner-eigen-const-approx} Let
\beaa
& & c_{m,n}=m!2^{-nm/2}\prod_{j=1}^m\frac{\Gamma(3/2)}{\Gamma(1+j/2)\Gamma((n-m+j)/2 )};\\
& & C(m, n)=n^{m/2}e^{-nm/2}n^{m(n-m+1)/2-m}n^{m(m-1)/4}c_{m,n};\\
& & c_m= m!2^{-m}2^{-m(m-1)/4}\pi^{-m/2}\prod_{j=1}^m\frac{\Gamma(3/2)}{\Gamma(1+j/2)}.
\eeaa
If $m^3=o(n)$, then $\log C({n,m})-\log c_m=o(1)$ as $n\to\infty$.
\end{lemma}
\begin{proof}[Proof of Lemma~\ref{lemma:wigner-eigen-const-approx}] %Recall that From \eqref{eq:cwig} we know
%\begin{equation}
%c_m= m!2^{-m}2^{-m(m-1)/4}\pi^{-m/2}\prod_{j=1}^m\frac{\Gamma(3/2)}{\Gamma(1+j/2)}.	\end{equation}
%Therefore, by canceling $m!$ and the product in (\ref{eq:cwig}) we see
Easily,
\beaa
 \frac{C({m,n})}{ c_m} &= & \frac{n^{m/2}e^{-nm/2}n^{m(n-m+1)/2-m}n^{m(m-1)/4}\cdot 2^{-nm/2}}{2^{-m}2^{-m(m-1)/4}\pi^{-m/2}\cdot\prod_{j=1}^m\Gamma((n-m+j)/2 )}\\
 &= & n^{m(2n-m-1)/4}e^{-mn/2}2^{(m-2n+3)m/4}\pi^{m/2}e^{-J_n}
\eeaa
where
\begin{eqnarray}\label{Xia}
J_n:&=&\log \prod_{j=0}^{m-1}\Gamma((n-j)/2)\nonumber\\
& = &  m\log \Gamma\Big(\frac{n}{2}\Big) + \log \prod_{j=0}^{m-1}\frac{\Gamma((n-j)/2)}{\Gamma(n/2)}.
\end{eqnarray}
%\begin{equation}\lbl{Xia}
%J_n:&=&\log \prod_{j=0}^{m-1}\Gamma((n-j)/2)\nonumber\\
%& = &  m\log \Gamma(n/2) + \log \prod_{j=0}^{m-1}\frac{\Gamma((n-j)/2)}{\Gamma(n/2)}.
%\end{equation}
Then
\begin{eqnarray}\label{T_o}
& & \log \frac{C({m,n})}{ c_m}\nonumber\\
&=&\frac{m}{4}\big(2n-m-1\big)\log n-\frac{1}{2}mn+\frac{m}{4}(m-2n+3)\log 2+\frac{m}{2}\log \pi - J_n\nonumber\\
& = &  \frac{m}{4}\big(2n-m-1\big)\log \frac{n}{2}-\frac{1}{2}mn +\frac{m}{2}\log (2\pi)-J_n.
\end{eqnarray}
Write
\beaa
\log \frac{\Gamma(x+b)}{\Gamma(x)}=(x+b)\log (x+b)-x\log x-b + \delta(x,b).
\eeaa
 By Lemma 5.1 from Jiang and Qi (2015), there exists a constant $C>0$ free of $x$ and $b$ such that
$$|\delta(x,b)|\leq C\cdot\frac{b^2+|b|x+1}{x^2}$$
for all $x\geq 10$ and $|b|\leq x/2$. It is easy to see that
\beaa
\sum_{j=0}^{m-1}\Big|\delta\Big(\frac{n}{2}, -\frac{j}{2}\Big)\Big| \leq C'\cdot  \frac{m^3+nm^2+m}{n^2}
\eeaa
where $C'$ is a constant free of $m$ and $n$.  This implies that
\beaa
\log \prod_{j=0}^{m-1}\frac{\Gamma((n-j)/2)}{\Gamma(n/2)}
=O\Big(\frac{m^2}{n}\Big)+\sum_{j=0}^{m-1}\Big(\frac{n-j}{2}\log \frac{n-j}{2}-\frac{n}{2}\log \frac{n}{2}+\frac{j}{2}\Big)
\eeaa
as $n\to\infty$. Write
\beaa
\frac{n-j}{2}\log \frac{n-j}{2}-\frac{n}{2}\log \frac{n}{2}=\frac{n}{2}\log \Big(1-\frac{j}{n}\Big) -\frac{j}{2}\log \frac{n}{2}-\frac{j}{2}\log \Big(1-\frac{j}{n}\Big).
\eeaa
 Easily, $\log (1-\frac{j}{n})=-\frac{j}{n}+O(\frac{m^2}{n^2})$ as $n\to\infty$ uniformly for all $1\leq j \leq m.$ Hence,
\beaa
& & \sum_{j=0}^{m-1}\Big(\frac{n-j}{2}\log \frac{n-j}{2}-\frac{n}{2}\log \frac{n}{2}+\frac{j}{2}\Big)\\
& = & O\Big(\frac{m^2}{n}\Big)+\sum_{j=0}^{m-1}\Big(-\frac{j}{2}-\frac{j}{2}\log \frac{n}{2}+\frac{j^2}{2n}+\frac{j}{2}\Big)\\
&= & O\Big(\frac{m^2}{n}\Big)-\frac{1}{4}m(m-1)\log \frac{n}{2}+\frac{(m-1)m(2m-1)}{12n}\\
& = & -\frac{1}{4}m(m-1)\log \frac{n}{2}+O\Big(\frac{m^3}{n}\Big)
\eeaa
as $n\to\infty$. In summary,
\beaa
\log \prod_{j=0}^{m-1}\frac{\Gamma((n-j)/2)}{\Gamma(n/2)}
=-\frac{1}{4}m(m-1)\log \frac{n}{2}+O\Big(\frac{m^3}{n}\Big)
\eeaa
as $n\to\infty.$ On the other hand, by the Stirling formula,
\beaa
m\log \Gamma\Big(\frac{n}{2}\Big)=\frac{(n-1)m}{2}\log \frac{n}{2}-\frac{1}{2}mn+\frac{m}{2}\log (2\pi) + O\Big(\frac{m}{n}\Big)
\eeaa
as $n\to\infty.$ From (\ref{Xia}) and the above two assertions we see
\bea\label{cool_flower}
J_n=\frac{m}{4}\big(2n-m-1\big)\log \frac{n}{2}-\frac{1}{2}mn +\frac{m}{2}\log (2\pi) + o(1)
\eea
as $n\to\infty$, which together with (\ref{T_o}) proves the lemma.
\end{proof}

\begin{lemma}\label{lemma:AB} Let
\beaa
& & \Gamma_m\Big(\frac{n}{2}\Big)= \pi^{m(m-1)/4}\prod_{j=1}^m \Gamma\Big(\frac{n-j+1}{2}\Big);\\
& & A(m,n)=n^{m(m+1)/4 +m(n-m-1)/2}e^{-nm/2}2^{-\frac{nm}{2}}/\Gamma_m(n/2);\\
& & B(m)= (2\pi)^{-m(m+1)/4}2^{-m/2}.
\eeaa
If $m^3=o(n)$, then $\log A(m,n)-\log B(m)\to 0$ as $n\to\infty$.
\end{lemma}
\begin{proof}[Proof of Lemma~\ref{lemma:AB}] Observe
\bea\lbl{Great_moon}
& &  \log \frac{A(m,n)}{B(m)}\nonumber\\
&=& \Big[\frac{1}{4}m(m+1)+\frac{1}{2}m(n-m-1)\Big]\log n -\frac{1}{2}mn\nonumber\\
& & -\frac{1}{2}m(n-1)\log 2+\frac{1}{4}m(m+1)\log (2\pi)-\log \Gamma_m\Big(\frac{n}{2}\Big).
\eea
By definition,
\beaa
\Gamma_m\Big(\frac{n}{2}\Big)
%&=& \pi^{m(m-1)/4}\prod_{j=1}^m \Gamma\Big(\frac{n}{2}+\frac{1-j}{2}\Big)\\
%& = &
=\pi^{m(m-1)/4}\prod_{j=0}^{m-1} \Gamma\Big(\frac{n-j}{2}\Big).
\eeaa
From (\ref{Xia}) and (\ref{cool_flower}), we see that
\beaa
& & \log \Gamma_m\Big(\frac{n}{2}\Big)\\
&=&\frac{1}{4}m(m-1)\log \pi +\frac{m}{4}\big(2n-m-1\big)\log \frac{n}{2}-\frac{1}{2}mn +\frac{m}{2}\log (2\pi) + o(1)\\
& = & \frac{m}{4}\big(2n-m-1\big)\log n -\frac{1}{2}mn-\frac{1}{2}m(n-1)\log 2+\frac{1}{4}m(m+1)\log (2\pi)+o(1).
\eeaa
By comparing this identity with (\ref{Great_moon}), we conclude $\log \frac{A(m,n)}{B(m)}\to 0$.
\end{proof}

\section{Proofs of lemmas}\label{fish_dead}

The following result is based on a slight modification of the second inequality of (4.8) from \cite{Jiang:2015ir} and a care taken by noticing that the version of the Wigner matrix here is $\sqrt{2}$ times of the version there. It will enable us to bound the last probability.
\begin{proof}[Proof of Lemma~\ref{eq:wig-margin-tail}] Review the proof of Lemma 4.1 from \cite{Jiang:2015ir}. Notice  that the version of the Wigner matrix here is $\sqrt{2}$ times of the version there. From the second inequality in (4.8) in the paper, there is a positive constant $C$ not depending on $m$ such that
\beaa
&&	P\Big(\lambda_1(\tW_{\{1,...,m\}})\geq x \text{ or } \lambda_{m}(\tW_{\{1,...,m\}})\leq -x\Big)\\
&	\leq & C\cdot \exp\Big\{-\frac{x^2}{4} + C m\log x +Cm \Big\}
\eeaa
for all $x>4\sqrt{m}$ and all $m\geq 2$. Since the right hand side above is increasing in $C$, without loss of generality, we assume $C>1$. It is easy to see $\log C\leq Cm \leq C m \log x$ under the assumption that $x>4\sqrt{m}$. By taking $\kappa=3C$ we get the desired conclusion.
\end{proof}

\begin{proof}[Proof of Lemma~\ref{lemma:p-choose-m}]
Note that
	\begin{equation}
		\binom{p}{m}= \frac{p\cdot(p-1)\cdot...\cdot (p-m+1)}{m!}
	\end{equation}
	and $\frac{p-l}{m-l}\geq \frac{p}{m}$ for $l\geq 0$. Thus,
	\begin{equation}
		\binom{p}{m}\geq \Big(\frac{p}{m}\Big)^m.
	\end{equation}	
	On the other hand, by the Sterling formula,
	\begin{equation}
		m!\geq \sqrt{2\pi}m^{m+{1/2}}e^{-m}> m^m e^{-m}.
	\end{equation}
	Therefore,
	\begin{equation}
		\binom{p}{m}\leq\frac{p^m}{m!}<\frac{p^m}{m^me^{-m}}.
	\end{equation}
	Combining the two inequalities, we complete the proof.
\end{proof}

\begin{proof}[Proof of Lemma~\ref{lemma:max-l}] Write
\begin{eqnarray*}
g(x):&=&(2m-x)-\frac{2m^2-x^2}{m}(1-\varepsilon)^2\\
& = & 2m\big[1-(1-\varepsilon)^2\big]+\frac{(1-\varepsilon)^2}{m}x^2-x
\end{eqnarray*}
for $1\leq x\leq m-1.$ Obviously, $g(x)$ is a convex function. This leads to that  $\max_{1\leq l \leq m-1}g(l)=g(1)\vee g(m-1).$ It is trivial to check that $g(1)\geq g(m-1)$. The first identity is thus obtained. The second identity follows from the first one.
%
%	Note that
%	\begin{equation}
%	\begin{split}
%				&\max_{1\leq l\leq m-1}\{(2m-l)-\frac{2m^2-l^2}{m}(1-\varepsilon)^2\}\\
%			=&2m-2m(1-\varepsilon)^2 + \frac{(1-\varepsilon)^2}{m}(l-\frac{m}{2(1-\varepsilon)^2})^2-\frac{m}{4(1-\varepsilon)^2}.
%	\end{split}
%	\end{equation}
%	The above expression is maximized when $|l-\frac{m}{2(1-\varepsilon)^2}|$ is maximized. Note that the range of $l$ is $\{1,...,m-1\}$ and
%	$\frac{m}{2}<\frac{m}{2(1-\varepsilon)}<m$. Thus, $|l-\frac{m}{2(1-\varepsilon)^2}|$ is maximized if we choose $l=1$. Thus, we have
%	\begin{equation}
%		(2m-l)-\frac{2m^2-l^2}{m}(1-\varepsilon)^2
%		= 		(2m-1)-\frac{2m^2-1}{m}(1-\varepsilon)^2.
%	\end{equation}
\end{proof}

\begin{proof}[Proof of Lemma~\ref{lemma:combine}]
By rearranging the terms, we have
\begin{equation}
\begin{split}
		\frac{(p-m)!^2}{p!(p-2m)!}
		=&\frac{(p-m)\cdots (p-2m+1)}{p\cdots (p-m+1)}\\
		=&\frac{p-m}{p}\cdot\frac{p-m-1}{p-1}\cdots \frac{p-2m+1}{p-m+1}\\
		<& 1.
\end{split}
\end{equation}

\end{proof}

\begin{proof}[Proof of Lemma~\ref{lemma:integral}] Note that
	\begin{equation}
		\E\big(e^{\alpha Z}\mathbf{1}_{\{Z\geq \delta\}}\big)
		= \E\Big[\int_{-\infty}^{Z}\alpha e^{\alpha t} dt\cdot \mathbf{1}\{Z\geq \delta\}\Big]= \E\Big[\int_{-\infty}^{\infty}\alpha e^{\alpha t} dt\cdot \mathbf{1}{\{Z\geq \delta\vee t\}}\Big]dt.
	\end{equation}
	Use the Fubini Theorem to see
	\begin{equation}
	\begin{split}
		\E(e^{\alpha Z}\mathbf{1}_{\{Z\geq \delta\}})
		= &\int_{-\infty}^{\infty} \alpha e^{\alpha t} \PP(Z\geq \delta\vee t)dt\\
		 = & \int_{-\infty}^{\delta}\alpha e^{\alpha t} dt \PP(Z\geq \delta) + \alpha \int_{\delta}^{\infty}e^{\alpha t}\PP(Z\geq t)dt\\
		 =& e^{\alpha \delta} \PP(Z\geq \delta) + \alpha \int_{\delta}^{\infty}e^{\alpha t}\PP(Z\geq t)dt.
	\end{split}
	\end{equation}
\end{proof}

\begin{proof}[Proof of Lemma~\ref{lemma:moderate-bound}]
	The technique to be used here is similar to that from \cite{Fey:2008go}, where the large deviations  for the extreme eigenvalues of Wishart matrices are developed. Thus we will omit the repetitive details and only state the main steps.
	To ease notation, we write $U_n=\frac{1}{n}W_{\{1,...,m\}}$, and we use $\lambda_{m}(U_n)$ to denote its smallest eigenvalue. The event $\big\{\frac{\lambda_1(W_{\{1,...,m\}})-n}{\sqrt{n}}\geq \sqrt{n}y\big\}$ is equal to $\{\lambda_{1}(U_n)\geq 1+y\}$ and
$\big\{\frac{\lambda_m(W_{\{1,...,m\}})-n}{\sqrt{n}}\leq- \sqrt{n}y\big\}$ is equal to
	$ \{\lambda_{m}(U_n)\leq 1-y\}$.
	% Therefore,
	% \begin{equation}
	% 	\PP(\frac{\|W_{[1,..,m]}-nI_m\|}{\sqrt{n}}\geq \sqrt{n}y)
	% 	\leq \PP(\lambda_{1}(U_n)\geq 1+y) + \PP(\lambda_{m}(U_n)\leq 1-y).
	% \end{equation}
	We start to bound $\PP(\lambda_{m}(U_n)\leq 1-y)$. Since $\lambda_m(W_{\{1,...,m\}})\geq 0$, we assume $y \in (0,1)$ without loss of generality.

Note that $\lambda_{m}(U_n)=\min_{v:\|v\|=1}v^{\intercal}U_nv$ and
 the sphere $S^{m-1}=\{v\in \mathbb{R}^m:\, \|v\|=1\}$ can be covered by $\cup B(v^{(i)},d)$ for some $v^{(1)},...,v^{(N_d)}\in S^{m-1}$. Here, we use $B(v,d)$ to denote an open ball centered around $v$ with radius $d$.
It is straightforward to verify that for any $v \in S^{m-1}$, there always exists $j\in \{1,...,N_d\}$ such that
\begin{equation}\label{Wangjing}
	|v^{\intercal} U_nv-v^{(j)\intercal}U_nv^{(j)}	|\leq 2 \lambda_{1}(U_n)d.
\end{equation}
Therefore,  by considering $\{\lambda_{1}(U_n)\geq rm\}$ occurs or not, we have
\begin{eqnarray}\label{eq:lambda-min-bound}
& & \PP(\lambda_{m}(U_n)\leq 1-y) \nonumber\\
&=&\PP(\min_{v:\|v\|=1}v^{\intercal}U_n v\leq 1- y) \nonumber\\
%&\leq& N_d\sup_{v^{(j)}}\PP(v^{(j)\intercal}U_nv^{(j)}\leq 1-r_n+2d\kappa m)+\PP(\lambda_{1}(U_n)\geq \kappa m),\\
&\leq &N_d\cdot \sup_{v:\|v\|=1}\PP(v^{\intercal}U_nv\leq 1-y+2d mr)+\PP(\lambda_{1}(U_n)\geq mr)
\end{eqnarray}
for all $r>0$. We next analyze $N_d$, $\PP(v^{\intercal}U_nv\leq 1-y+2dmr)$ and $\PP(\lambda_{1}(U_n)\geq mr)$ separately.

We start with $N_d$, which is the minimum number of balls with the radius $d$ required to cover $S^{m-1}$.
%This covering number of a sphere has been well studied in literature. Here we use
By a result from \cite{Rogers:1963tt} we see
\begin{equation}\label{eq:nd-uniform}
	N_d=O\big(m^{1.5}(\log m) d^{-m}\big)
\end{equation} for all $0<d<1/2$ and $m\geq 1$.  As a result,
\begin{equation}\label{eq:covering-number}
	\log N_d
%= O\Big(1.5\log m + \log\log m +m\log\Big(\frac{1}{d}\Big)\Big)
=O\Big(m\log\frac{1}{d}\Big),\ \ 0<d<\frac{1}{2}.
\end{equation}
%for all $0<d<1/2$.

We proceed to an upper bound for $\PP(v^{\intercal}U_nv\leq 1-y+2dmr)$.
Recall that $U_n=\frac{1}{n}X_{\cdot,[1,..,m]}^{\intercal}X_{\cdot,[1,..,m]}$, where we use the notation
\begin{equation}
	X_{\cdot,[1,..,m]}= (x_{ij})_{1\leq i\leq n, 1\leq j\leq m}.
\end{equation}
Thus,
\bea\lbl{morning_sun}
	v^{\intercal}U_nv=\frac{1}{n}\|X_{\cdot,[1,..,m]}v \|^2
	=\frac{1}{n}\sum_{i=1}^n S_{v,i}^2,
\eea
where we define $S_{v,i}=\sum_{l=1}^m X_{il}v_{l}$. Review $\|v\|=1$. Since  $x_{ij}$'s are standard normals, so are $\{S_{v,i};\, 1\leq i\leq n\}$. By the large deviation bound for the sum of i.i.d. random variables [see, e.g., page 27 from Dembo and Zeitouni (1998)],
\bea\lbl{tea_coffee}
\PP\Big(\frac{1}{n}\sum_{i=1}^n S_{v,i}^2\in A\Big) \leq 2\cdot \exp\big\{-n\inf_{x\in A}I(x)\big\}
\eea
where $A\subset \mathbb{R}$ is any Borel set and $I(x)=\sup_{t\in \mathbb{R}}\{t x-\log \E e^{t N(0,1)^2} \}$. Since $\log \E(e^{tN(0,1)^2})=-\frac{1}{2}\log(1-2t)$ for $t<1/2,$
it is easy to check that
\beaa
I(x)=\begin{cases}
\frac{1}{2}(x-1-\log x), & \text{if $x>0$;}\\
\infty, & \text{if $x\leq 0$.}
\end{cases}
\eeaa
Observe that $I(x)$ is decreasing for $x\leq 1$. This together with (\ref{morning_sun}) and (\ref{tea_coffee}) implies that
%It is easy to check that $\phi(t):=\log(E(e^{tN(0,1)^2}))=-\frac{1}{2}\log(1-2t)$ for $t<1/2$ and $\sup_{t\leq 0}\{t\beta+\phi(t)\}=\frac{1}{2}(\beta-\log(1+\beta))$ for $-1<\beta<0$, see, e.g., \cite[p. 1056]{Fey:2008go}.
%%$(1-2t)^{-1/2}$ for $t<1/2$,
%% By the standard Chernoff bound (see, e.g., \cite{DZ}),
%% It is not hard to verify that $S_{v,i}$ are i.i.d. for $i=1,...,n$ and $E(e^{tS_{x,1}^2})=\frac{1}{\sqrt{1-2t}}$ for $t<1/2$.
%% With this moment generating function, we could obtain the following Chernoff type of upper bound
%By the standard Chernoff bound
%(see, e.g., \cite{Chernoff:1952fe}) for any  $-1<\beta<0$,
%\begin{equation}\label{eq:chernoff-bound}
%	\PP\left(v^{\intercal}U_nv\leq 1-\beta\right)\leq \inf_{t\leq 0}e^{-n(t\beta+\phi(t)) }=e^{-\frac{n}{2}(\beta-\log(1+\beta)) }.
%\end{equation}
%% where $\phi(t)=\log E(e^{tS_{x,1}^2})$.
%% Plugging $\phi(t)=\log\frac{1}{\sqrt{1-2t}}$ into the above optimization, it is not hard to get that for all $\beta<0$,
%% \begin{equation}
%% 	\inf_{t\leq 0}e^{-n(t\beta+\log E(e^{tS_{x,1}^2})  }
%% 	= e^{-\frac{n}{2}(\beta-\log(1+\beta)) }.
%% \end{equation}
%% \begin{equation}\label{eq:chernoff-bound}
%% 	\PP\left(v^{\intercal}U_nv\leq 1-\beta\right)\leq e^{-\frac{n}{2}(\beta-\log(1+\beta)) },
%% \end{equation}
%% for any $\beta \in (0,1)$, see, e.g., p??? from {\red Dembo and Zeitouni}.
%Taking $\beta= y - 2dmr$ in the above inequality, we have
\begin{equation}\label{eq:ldp-part}
	%\PP\left(v^{\intercal}U_nv\leq 1-y+2dmr\right)\leq e^{-n I_-(y - 2dmr) }
\PP\left(v^{\intercal}U_nv\leq 1-y+2dmr\right)\leq e^{-n I(1-y +2dmr) }
\end{equation}
for all $y>2dmr.$

Now we estimate $\PP(\lambda_{1}(U_n)\geq r)$ appeared in (\ref{eq:lambda-min-bound}).
Noting that $U_n$ is semi-positive definite, we have
$\lambda_{1}(U_n)\leq \text{trace}(U_n)\leq
\frac{1}{n}\sum_{i=1}^n\sum_{l=1}^m x_{il}^2$, and hence
\begin{equation}\label{eq:max-tail}
	\PP\left(
\lambda_{1}(U_n)\geq mr
	\right)
	\leq \PP\left(
\frac{1}{mn}\sum_{i=1}^n\sum_{l=1}^m x_{il}^2\geq r \right)\leq
e^{-mn I(r)}
\end{equation}
for $r\geq 1$ by (\ref{tea_coffee}).
%where $I_1(r):=\frac{1}{2}(r-1-\log(r))=I(r-1)$.
% The last inequality in (\ref{eq:max-tail}) is again obtained by using the Chernoff bound similar to \eqref{eq:chernoff-bound}.
Combining \eqref{eq:covering-number}, \eqref{eq:ldp-part}, and \eqref{eq:max-tail}, we obtain from (\ref{eq:lambda-min-bound}) that
%\begin{equation}\lbl{cell_511}
%	\PP(\lambda_{m}(U_n)\leq 1-y)
%	\leq \exp\Big\{
%	-n I_-(y - 2dmr)  + O\Big(m\log\big(\frac{1}{d}\big)\Big)
%	\Big\}+\exp\{
%	-nm I_1(r)
%	\}.
%\end{equation}
\bea\lbl{cell_511}
	& & \PP(\lambda_{m}(U_n)\leq 1-y)\nonumber\\
	&\leq & \exp\Big\{
	-n I(1-y + 2dmr)  + O\Big(m\log\frac{1}{d}\Big)
	\Big\}+\exp\{
	-mn I(r)
	\}.
\eea
for $y>2dmr$ and $r\geq 1$. This confirms (\ref{mid_range_2}).

To get  (\ref{mid_range}), just notice $\lambda_{1}(U_n)=\max_{v:\|v\|=1}v^{\intercal}U_nv$. From (\ref{Wangjing}) and (\ref{eq:lambda-min-bound}) we see that
\begin{eqnarray*}\label{eq:lambda-min-bound_hi}
& & \PP(\lambda_{1}(U_n)\geq 1+y)\\
%&=&\PP(\min_{v:\|v\|=1}v^{\intercal}U_n v\leq 1- y)\\
%&\leq& N_d\sup_{v^{(j)}}\PP(v^{(j)\intercal}U_nv^{(j)}\leq 1-r_n+2d\kappa m)+\PP(\lambda_{1}(U_n)\geq \kappa m),\\
&\leq &N_d\cdot \sup_{v:\|v\|=1}\PP(v^{\intercal}U_nv\geq 1+y-2dmr )+\PP(\lambda_{1}(U_n)\geq mr).
\end{eqnarray*}
Then (\ref{mid_range}) follows from similar arguments to (\ref{eq:ldp-part})-(\ref{cell_511}).
% as that obtaining (\ref{cell_511}).
%Using similar arguments as those for the above inequality, we obtain
%\begin{equation}
%	\PP(\lambda_{1}(U_n)\geq 1+y)
%	\leq \exp\{
%	-n I_+(y - 2dmr)  + O(m\log(\frac{1}{d}))
%	\}+\exp\{
%	-nm I_1(r)
%	\}.
%\end{equation}
\end{proof}

\begin{proof}[Proof of Lemma~\ref{lemma:case1}] Review Assumption 1 in \eqref{nice_birth_1}.
	We start with the analysis of \eqref{eq:case1-second}. Here, we consider two sub-cases: $t\leq \frac{mn}{80\alpha}$ and $t> \frac{mn}{80\alpha}$. For $t\leq \frac{mn}{80\alpha}$, we have $r=\max(2, 1+\frac{80\alpha t}{mn})=2$ and
\begin{equation}
\begin{split}
  &\exp\left\{
  -\frac{1}{2}(r-1-\log r) mn+\alpha t+2\log t+m\log p
  \right\}\\
  % = & \exp\{
  % -\frac{1-\log 2}{2}mn + \alpha t+2\log t + m\log p
  % \}\\
  \leq&
  \exp\left\{
  -\frac{1-\log 2}{2}mn + \frac{mn}{80}+2\log\left(\frac{mn}{80\alpha}\right) + m\log p
  \right\}.
\end{split}
\end{equation}
Trivially $\frac{1-\log 2}{2}-\frac{1}{80}=0.14\cdots > \frac{1}{10}$. Note that $\log \frac{mn}{80\alpha}=o(mn)$ and $m\log p=o(mn)$ under Assumption 1 in \eqref{nice_birth_1}. It follows that
\begin{equation}
		  -\frac{1-\log 2}{2}mn + \frac{mn}{80} +2\log\Big(\frac{mn}{80\alpha}\Big) + m\log p
  \leq -\Big[\frac{1}{10}+o(1)\Big] mn.
\end{equation}
% \begin{equation}
% \begin{split}
% 	  &-\frac{1-\log 2}{2}mn + \frac{mn}{80} +2\log(\frac{mn}{80}) + m\log p\\
%   \leq & -(\frac{1}{10}+o(1)) mn.
% \end{split}
% \end{equation}
This implies
\begin{equation}\label{eq:case1-second-1}
\lim_{n\to\infty}\sup_{\frac{\delta\sqrt{n}}{100}\leq t\leq \frac{mn}{80\alpha} }\exp\left\{
  -\frac{1}{2}(r-1-\log r) mn+\alpha t+2\log t+m\log p
  \right\}= 0.
\end{equation}
Now we consider another sub-case where $t\geq \frac{mn}{80\alpha}$. For this case, $r=1+\frac{80\alpha t}{mn}$. It is not hard to see
 $$r-1-\log r\geq \frac{r}{12}$$ for $r\geq 2$. Apparently,
 $m\log p\leq \alpha t$ for $t\geq \frac{mn}{80\alpha}$ as $n$ is sufficiently large. It follows that
\begin{equation}
\begin{split}
  &\exp\left\{
  -\frac{1}{2}(r-1-\log r) mn+\alpha t+2\log t+m\log p
  \right\}\\
  \leq & \exp\left\{
-\frac{1}{24}\Big(1+\frac{80\alpha t}{mn}\Big)mn +2\alpha t +2 \log t
  \right\}\\
%   \leq & exp\{
% -\frac{80\alpha t}{48} +\alpha t +2 \log t + m\log p
%   \}\\
  \leq & \exp\left\{
-\left(\frac{80}{24}-2\right)\alpha t  +2 \log t
  \right\}\\
  = & \exp\left\{
-\left(\frac{4\alpha }{3}+o(1)\right)t
  \right\}.
\end{split}
\end{equation}
This implies
\begin{equation}\label{eq:case1-second-2}
  \lim_{n\to\infty}\sup_{t\geq \frac{mn}{80\alpha}}
 \exp\left\{
  -\frac{1}{2}(r-1-\log r) mn+\alpha t+2\log t+m\log p
  \right\}=0.
\end{equation}
% under Assumption 1 (we used $\log p=o(n)$ here).
Combining \eqref{eq:case1-second-1} and \eqref{eq:case1-second-2}, we obtain
\begin{equation}
  \lim_{n\to\infty}\sup_{ t\geq \frac{\delta\sqrt{n}}{100}}\exp\left\{
  -\frac{1}{2}(r-1-\log r) mn+\alpha t+2\log t+m\log p
  \right\}= 0.
\end{equation}
% under Assumption 1.
This completes the proof of \eqref{eq:case1-second}. We next show  (\ref{eq:case1-first}).
%proceed to the analysis of  $\exp\{
%  -\frac{n}{2}(z-\log(1+z))+\kappa m\log(1/d) +\alpha t+2\log t+ m \log p
%  \}$.

Recall $z=\frac{2\sqrt{m\log p}}{\sqrt{n}}+\frac{t}{2\sqrt{n}}\geq \frac{t}{2\sqrt{n}}$. Obviously, $z>\frac{\delta}{200}$ as $t>\frac{\delta\sqrt{n}}{100}$. It is elementary  to check there exists $\varepsilon>0$ such that $x-\log (1+x)\geq \varepsilon x$ for all $x>\frac{\delta}{200}$. Hence,
\bea\lbl{mom_seal}
  \frac{n}{2}[z-\log (1+z)] \geq \frac{1}{2}n\varepsilon z\geq \frac{\varepsilon}{4}\sqrt{n}\, t
\eea
for all $t>\frac{\delta\sqrt{n}}{100}$.
%that in the current case, $t>\frac{\delta\sqrt{n}}{100}$ and $z=\frac{2\sqrt{m\log p}}{\sqrt{n}}+\frac{t}{2\sqrt{n}}>\frac{t}{2\sqrt{n}}\geq\frac{\delta}{400}+\frac{t}{4\sqrt{n}}$.
%Let $w(x)=x-\log (1+x)$ for $x\geq 0$. Then, $w(x)$ is a convex function in $x$, and, as a result,
% $$w(x+u)\geq w(x)+w'(x)u$$ for all $x>0$ and $u>0$.
%Using this inequality, we have
%\begin{equation}
%-\frac{n}{2}(z-1-\log z)\leq
%-w(\frac{\delta}{400})\frac{n}{2}-
%w'(\frac{\delta}{400})\frac{\sqrt{n}t}{8}.
%\end{equation}
%Note that $w(\frac{\delta}{400})$ and $w'(\frac{\delta}{400})$ are positive constants (depending only on $\delta$).
Reviewing $r=\max(2, 1+\frac{80\alpha t}{mn})$ and  $d=\min(\frac{1}{2}, \frac{t}{4m\sqrt{n}r})$, we have
\begin{equation}
\begin{split}
  m\log\frac{1}{d}=&\max\Big\{
  m\log 2, m\log \frac{4m\sqrt{n}r}{t}
  \Big\}\\
  =&\max\{
m\log 2, O(m\log m + m\log n+ m\log r)
  \}\\
  =&O\Big(m\log n+ m\log\Big(1+\frac{80\alpha t}{mn}\Big)\Big)\\
  =&O\Big(m\log n + \frac{ t}{n}\Big)
\end{split}
\end{equation}
since $0< \log (1+x)< x$ for all $x>0$. Furthermore, $\alpha t+2\log t\leq 2\alpha t$ as $t$ is sufficiently large, and $m \log p=o(mn)$ by Assumption 1 in \eqref{nice_birth_1}.
Consequently,
\begin{equation}
\begin{split}
  &\sup_{ t>\frac{\delta\sqrt{n}}{100}}\exp\left\{
  -\frac{n}{2}(z-\log(1+z))+\kappa m\log\frac{1}{d} +\alpha t+2\log t+ m \log p
\right\}\\
 \leq & \sup_{ t>\frac{\delta\sqrt{n}}{100}}\exp\left\{-\frac{\varepsilon}{4}\sqrt{n}\, t+3\alpha t + O(m\log n)
\right\}\\
=&\exp\left\{
-(1+o(1))\frac{\varepsilon}{4}\sqrt{n}\, t
\right\}\\
=& o(1).
\end{split}
\end{equation}
We obtain (\ref{eq:case1-first})  and the proof is completed.
\end{proof}

\begin{proof}[Proof of Lemma~\ref{lemma:case2}] It is trivial to show that
%	
%We start with the first term $\exp\{
%  -\frac{n}{2}(z-\log(1+z))+\kappa m\log(1/d) +\alpha t+2\log t+ m \log p
%  )
%  \}$.
%Use the Taylor expansion of the function $\log(1+z)$,
\begin{equation}\label{eq:small-1}
  z-\log (1+z)\geq \frac{1}{4}z^2
%(1-\frac{2}{3}z).
\end{equation}
for $0\leq z\leq 1$.
Recall the assumption that $\delta \in (0,1)$.  Then $z= \frac{2\sqrt{m\log p}}{\sqrt{n}}+\frac{t}{2\sqrt{n}}\leq \frac{2\sqrt{m\log p}}{\sqrt{n}}+\frac{\delta}{200}\leq 1$ as $n$ is sufficiently large. Now,
%\begin{equation}\label{eq:small-2}
%  1-\frac{2}{3}z\geq 0
% \end{equation}
% for $n,p$ sufficiently large.
% Also, we have
 \begin{equation}\label{eq:small-3}
  z^2=\left(\frac{2\sqrt{m\log p}}{\sqrt{n}}+\frac{t}{2\sqrt{n}}\right)^2\geq \frac{4m\log p}{n}+\frac{2t\sqrt{m\log p}}{n}.
 \end{equation}
By \eqref{eq:small-1}, we see
 \bea\label{eq:small-3.5}
    z-\log (1+z)
    \geq  \frac{m\log p}{n}+\frac{t\sqrt{m\log p}}{2n}.
 \eea
Now, reviewing $d=\frac{t}{8m\sqrt{n}}$ and $t\geq \delta$, we have $m\log(1/d)=O(m\log m + m\log n)=O(m\log n)$. This joint with \eqref{eq:small-3.5} implies that
 \begin{equation}
  \begin{split}
    &\exp\left\{
  -\frac{n}{2}[z-\log(1+z)]+\kappa m\log\frac{1}{d} +\alpha t+2\log t+ m \log p
  \right\}\\
  \leq  &
  \exp\left\{
 -\frac{1}{4}t\sqrt{m\log p} + O(m\log n)+\alpha t+ 2\log t
  \right\} \\
  = & \exp\left\{
  \big[-\frac{1}{4}+o(1)\big]t\sqrt{m\log p} + O(m\log n)
  \right\}.
  \end{split}
 \end{equation}
Since $t\geq (m/\log p)^{1/2}\sinf\log n$, we know $O(m\log n)= o(t\sqrt{m\log p})$ uniformly in $t$. Thus,
 \begin{equation}
 %\label{eq:case2-first}
 \begin{split}
 	&\exp\left\{
  -\frac{n}{2}[z-\log(1+z)]+\kappa m\log\frac{1}{d} +\alpha t+2\log t+ m \log p
  \right\}\\
  \leq&
  \exp\left\{
  -\big[\frac{1}{2}+o(1)\big]t\sqrt{m\log p}
  \right\}\\
  \leq &\exp\left\{
  -\frac{1}{4}t\sqrt{m\log p}
  \right\}
 \end{split}
 \end{equation}
 as $n$ is sufficiently large. We then get \eqref{eq:case2-first}.
% Due to the restriction that $t\leq \frac{\delta\sqrt{n}}{100}$, we have
Evidently,
 \beaa
 \sup_{\delta\vee \omega_n \leq t\leq \frac{\delta\sqrt{n}}{100}}\{\alpha t+ 2\log t\} =O(\sqrt{n}\,)
 \eeaa
as $n\to\infty.$ This implies that
 \begin{equation}
 \begin{split}
  -\frac{1}{2}(1-\log 2) mn+\alpha t+2\log t+m\log p
  = -\frac{1-\log 2}{2}[1+o(1)] mn.
 \end{split}
\end{equation}
The assertion \eqref{eq:case2-second} is verified.
\end{proof}

\begin{proof}[Proof of Lemma~\ref{lemma:wig-eigen-approx}] Review the notation $W_{\{1,...,m\}}$ above (\ref{eq:t-stat}) with $S=\{1,\cdots, m\}.$
		Let $\mu_1>...>\mu_m$ be the eigenvalues of $W_{\{1,...,m\}}$. According to \cite{James:1964kw} or \cite{muirhead2009aspects},
		 $\mu=(\mu_1,...,\mu_m)$ has density function
	\begin{equation}
		f_{m,n}(\mu)=c_{m,n} e^{-\frac{\sum_{i=1}^m \mu_i}{2}}\prod_{i=1}^m \mu_i^{\frac{n-m+1}{2}-1}\prod_{1\leq j<i\leq m} (\mu_j-\mu_i)I(\mu_1>...>\mu_m>0),
	\end{equation}
	where $c_{m,n}=m!2^{-nm/2}\prod_{j=1}^m\frac{\Gamma(3/2)}{\Gamma(1+(j/2))\Gamma((n-m+j)/2 )}$. In addition, $\lambda=(\lambda_1,...,\lambda_m)$ has density
	\begin{equation}\label{eq:cwig}
		\begin{split}
			\fwig{m}(\lambda) &= c_m e^{-\frac{1}{4}\sum_{k=1}^m\lambda_k^2}\prod_{1\leq j<i\leq m} (\lambda_j-\lambda_i)  \text{ with }\\
			c_m&= m!2^{-m}2^{-m(m-1)/4}\pi^{-m/2}\prod_{j=1}^m\frac{\Gamma(3/2)}{\Gamma(1+(j/2))};
		\end{split}
	\end{equation}
see, for example, Chapter 17 from \cite{Mehta:2004wq}.
%{\color{red}XL: I cited the newer version of the book. Please check.}
%\cite{Mehta, 1991}.
% \begin{equation}
% \end{equation}
% where
% \begin{equation}
% \end{equation}	
% 	\begin{equation}\label{eq:cwig}
% c_m= m!2^{-m}2^{-m(m-1)/4}\pi^{-m/2}\prod_{j=1}^m\frac{\Gamma(1+1/2)}{\Gamma(1+j/2)}.	
% \end{equation}
%and $0<\lambda_m<...<\lambda_1$.
		Note that $\nu_i=(\mu_i-n)/\sqrt{n}$, so we can write down the expression of $g_{n,m}$ as follows.
			\begin{equation}
			\begin{split}
						&g_{n,m}(v)\\
				=&n^{m/2}c_{m,n} \exp\left(-\frac{1}{2}\sum_{i=1}^m (\sqrt{n}v_i+n)\right)\prod_{i=1}^m (\sqrt{n}v_i+n)^{\frac{n-m+1}{2}-1}\\
				&~~~~~~~~~\cdot\prod_{1\leq j<i\leq m} (\sqrt{n}v_j-\sqrt{n}v_i)\\
				=&n^{m/2}e^{-nm/2}n^{m(n-m+1)/2-m}n^{m(m-1)/4}c_{m,n}\\
				&~~~~~~~~~\cdot e^{-(\sqrt{n}/2)\sum_{i=1}^m v_i}\prod_{i=1}^m \left(1+\frac{v_i}{\sqrt{n}}\right)^{\frac{n-m-1}{2}}\prod_{1\leq j<i\leq m}(v_j-v_i)
				%=&C(m, n)e^{-(\sqrt{n}/2)\sum_{i=1}^m v_i}\prod_{i=1}^m \left(1+\frac{v_i}{\sqrt{n}}\right)^{\frac{n-m-1}{2}}\prod_{1\leq j<i\leq m}(v_j-v_i),
				\end{split}
			\end{equation}
for $v_1>v_2>...>v_m>-\sqrt{n}$ and $g_{n,m}(v)=0$, otherwise. Denote
\begin{equation}
C(m, n)=n^{m/2}e^{-nm/2}n^{m(n-m+1)/2-m}n^{m(m-1)/4}c_{m,n}.
\end{equation}
Then,
		\begin{equation}\label{eq:to-combine}
		\begin{split}
				&\log g_{n,m}(v)-\log \fwig{m}(v)\\
			=&
			\log C(m, n)-\log c_m +\sum_{i=1}^m\left[-\frac{\sqrt{n}}{2}v_i +\frac{n-m-1}{2}\log\Big(1+\frac{v_i}{\sqrt{n}}\Big)\right]+\frac{1}{4}\sum_{i=1}^m v_i^2
		\end{split}
		\end{equation}
for $v_1>v_2>...>v_m>-\sqrt{n}$. By Lemma \ref{lemma:wigner-eigen-const-approx} in Appendix \ref{bu_cucumber},
\bea\label{eq:g-by-fwig}
&&		\log g_{n,m}(v)-\log \fwig{m}(v)\nonumber\\
&=&o(1)+\sum_{i=1}^m\left[-\frac{\sqrt{n}}{2}v_i +\frac{n-m-1}{2}\log\Big(1+\frac{v_i}{\sqrt{n}}\Big)+\frac{1}{4}v_i^2\right]
\eea
for $v_1>v_2>...>v_m>-\sqrt{n}$. By the Taylor expansion,
\beaa
\Big|\log (1+x)-\big(x-\frac{x^2}{2}\big)\Big| \leq \sum_{i=3}^\infty\frac{|x|^k}{k}\leq \frac{|x|^3}{3(1-|x|)}
\eeaa
for all $|x|<1$. Therefore,
\begin{equation}\label{eq:taylor}
\Big|\log (1+x)-\big(x-\frac{x^2}{2}\big)\Big| \leq  |x|^3
\end{equation}
for $|x|<\frac{2}{3}.$ Writing $\frac{n-m-1}{2}=\frac{n}{2}-\frac{m+1}{2}$, it is easy to check
\bea\lbl{steam_water}
& & \sum_{i=1}^m\left[-\frac{\sqrt{n}}{2}v_i +\frac{n-m-1}{2}\Big(\frac{v_i}{\sqrt{n}}-\frac{v_i^2}{2n}\Big)
+\frac{1}{4} v_i^2\right]\nonumber\\
&=&-\frac{m+1}{2}\sum_{i=1}^m\Big(\frac{v_i}{\sqrt{n}}-\frac{v_i^2}{2n}\Big).
\eea
Combining \eqref{eq:g-by-fwig}-\eqref{steam_water}, and noting that $|v_i|\leq \|v\|_{\infty}$ for all $i$, we get
\begin{equation}
\begin{split}
		&\log g_{n,m}(v)-\log \fwig{m}(v)\\
		=&o(1)
	+\varpi_{m,n}\\	
&~~~~~+\sum_{i=1}^m\left[-\frac{\sqrt{n}}{2}v_i +\frac{n-m-1}{2}\Big(\frac{v_i}{\sqrt{n}}-\frac{v_i^2}{2n}\Big)+\frac{1}{4}v_i^2\right]\\
=&o(1) +\varpi_{m,n}-\frac{m+1}{2}\sum_{i=1}^m\Big(\frac{v_i}{\sqrt{n}}-\frac{v_i^2}{2n}\Big)
\end{split}
\end{equation}
provided $\|v\|_{\infty}\leq \frac{2}{3}\sqrt{n}$, where $\varpi_{m,n}$ is the error term and it is controlled by
\beaa
|\varpi_{m,n}|\leq \frac{n-m-1}{2}\cdot\frac{1}{n^{3/2}}\sum_{i=1}^m|v_i|^3\leq \frac{1}{n^{1/2}}\sum_{i=1}^m|v_i|^3.
\eeaa
By using the trivial bound that $|v_i|\leq \|v\|_{\infty}$ for each $i$, we obtain the desired conclusion from the above two assertions.
%\begin{equation}
%	\begin{split}
%		\log g_{n,m}(v)-\log \fwig{m}(v)
%		=  O(m^2\|v\|_{\infty}/{n^{1/2}}) + O(m^2 \|v\|_{\infty}^2/n)+O(m \|v\|_{\infty}^3/n^{1/2}).
%	\end{split}
%\end{equation}
\end{proof}

\begin{proof}[Proof of Lemma~\ref{lemma:log-likelihood-ratio}]
	According to the density function of the Wishart distribution [see, e.g., \cite{anderson1962introduction} or \cite{muirhead2009aspects}], the density function for $W_{\{1,...,m\}}$ is
	\begin{equation}
		f_{W,m}(V)=
		\frac{|V|^{(n-m-1)/2}e^{-tr(V)/2}}{\Gamma_m(\frac{n}{2}) 2^{mn/2} },
	\end{equation}
for every $m\times m$ positive definite matrix $V$,  where $\Gamma_m(\cdot)$ is the multivariate gamma function defined by
	\begin{equation}
		\Gamma_m\left(\frac{n}{2}\right)={\pi}^{m(m-1)/4}\prod_{j=1}^m \Gamma\left(\frac{n-j+1}{2}\right),
	\end{equation}
	and we write $|V|$ for the determinant of a matrix $V$.
It is easy to see that the density function for $\frac{W_{\{1,...,m\}}-nI_m}{\sqrt{n}}$ is given by
	\begin{align}
		&f_{m,n}(w):\\
				=&(\sqrt{n}\,)^{\frac{m(m+1)}{2}}f_{W,m}(\sqrt{n}w+nI_m)\\
		= &n^{m(m+1)/4}\frac{|\sqrt{n}w+nI_m|^{(n-m-1)/2}
e^{-tr(\sqrt{n}w+nI_m)/2}}{\Gamma_m(\frac{n}{2}) 2^{mn/2}}
	\end{align}
for every $m\times m$ matrix $w$ such that $w+\sqrt{n}I_m$ is positive definite.
Simplifying the above display, we further have
	\begin{align}
		 f_{m,n}(w)
		 %%%%%% HIDING SOME INTERMEDIATE DERIVATION HERE%%%%
		% =& n^{m(m+1)/4}\cdot n^{m(n-m-1)/2}\cdot e^{-\frac{nm}{2}} |1+\frac{w}{\sqrt{n}}|^{(n-m-1)/2}e^{-tr(\sqrt{n}w)}\Gamma_m(\frac{n}{2})^{-1}2^{-\frac{nm}{2}}\\
		% =&n^{m(m+1)/4 +m(n-m-1)/2}e^{-nm/2}2^{-\frac{nm}{2}}\Gamma_m(n/2)^{-1}|1+\frac{w}{\sqrt{n}}|^{(n-m-1)/2}e^{-\sqrt{n}tr(w)}\\
		% =& A(m,n) |1+\frac{w}{\sqrt{n}}|^{(n-m-1)/2}e^{-\sqrt{n}tr(w)}\\
		%%%%%% END OF INTERMEDIATE STEP
		= A(m,n)\exp\left\{
		\frac{n-m-1}{2}\log\Big|1+\frac{w}{\sqrt{n}}\Big|-\frac{1}{2}\sqrt{n}\,tr(w)
		\right\},
	\end{align}
	where $A(m,n)=n^{m(m+1)/4 +m(n-m-1)/2}e^{-nm/2}2^{-\frac{nm}{2}}/\Gamma_m(n/2)$.
	On the other hand,
	\begin{equation}
		\tf_m(w)
		=
		B(m)e^{-\frac{tr(w^2)}{4}},
	\end{equation}
	where $B(m)= (2\pi)^{-m(m+1)/2}2^{-m/2}$; see, for instance, \cite{Mehta:2004wq}. Now we consider
	\begin{equation}
	\begin{split}
			&	\log f_{m,n}(w)-\log \tf_m(w)\\
		= &\log A(m,n)-\log B(m)\\
& ~~~~~~~~~~~~~~~ +
		\frac{n-m-1}{2}\log \left|1+\frac{w}{\sqrt{n}}\right|-\frac{1}{2}\sqrt{n}\,tr(w)+\frac{1}{4}tr(w^2)
		\\
= & o(1)+\sum_{i=1}^m\left[\frac{n-m-1}{2}\log\Big(1+\frac{\lambda_i}{\sqrt{n}}\Big)
-\frac{1}{2}\sqrt{n}\lambda_i+\frac{1}{4}\lambda_i^2\right]
	\end{split}
	\end{equation}
for every $\lambda_i>-\sqrt{n}$ and $i=1,\cdots, m$ by Lemma \ref{lemma:AB} in Appendix \ref{bu_cucumber}, where $\lambda_1, \cdots \lambda_m$ are the eigenvalues of $w$. From \eqref{eq:taylor} and \eqref{steam_water},
\beaa
\log f_{m,n}(w)-\log \tf_m(w)=o(1) +\varepsilon_{m,n} -\frac{m+1}{2}\sum_{i=1}^m\Big(\frac{\lambda_i}{\sqrt{n}}-\frac{\lambda_i^2}{2n}\Big)
\eeaa
if $\max_{1\leq i \leq m}|\lambda_i|\leq \frac{2}{3}\sqrt{n}$, where $\varepsilon_{m,n}$ is the error term satisfying
{\beaa
|\varepsilon_{m,n}|\leq  \frac{n-m-1}{2}\sum_{i=1}^m\frac{|\lambda_i|^3}{n^{3/2}}
 \leq  \frac{m}{n^{1/2}}\max_{1\leq i\leq m}|\lambda_i|^3 =  mn^{-1/2}\|w\|^{3}.
\eeaa
}
In addition, $|\sum_{i=1}^m\lambda_i|\leq m\max_{1\leq i\leq m}|\lambda_i|=m\|w\|,$ and  $\sum_{i=1}^m\lambda_i^2\leq m\max_{1\leq i\leq m}|\lambda_i|^2=m\|w\|^2_2,$
The above three assertions lead to
\beaa
& & \log f_{m,n}(w)-\log \tf_m(w)\\
&=&o(1) +O\left(m^{2}n^{-1/2}\|w\|+ m^2n^{-1}\|w\|^2+ n^{-1/2}m\|w\|^{3}\right)
\eeaa
provided $\|w\|=\max_{1\leq i \leq m}|\lambda_i|\leq \frac{2}{3}\sqrt{n}$. The proof is finished.
\end{proof}

\begin{proof}[Proof of Lemma~\ref{lemma:MDP-wigner}]
	Let $B(v_1,\delta),..,B(v_N,\delta)$ be $N$ balls centered around $v_1,..,v_N$, respectively,  such that $\cup B(v_i,\delta)$ covers the unit sphere $\{v\in \mathbb{R}^m;\, \|v\|=1\}$. Then, for any $r>0$, by \eqref{Wangjing},
\bea\label{eq:tail-general-wig}
		& & \PP\big(\sup_{\|v\|=1}v^{\intercal} \tilde{W}_{\{1,..,m\}}v\geq x\big)\nonumber\\
		&\leq & N\cdot \max_{1\leq i\leq N}\PP\big(v_i^{\intercal}\tilde{W}_{\{1,..,m\}}v_i\geq x-2r\delta\big)+\PP\big(\lambda_{1}(\tilde{W}_{\{1,...,m\}})\geq r\big).
\eea
	According to the distribution of $\tilde{W}$,
	\begin{equation}
		v_i^{\intercal}\tilde{W}v_i\sim N(0, f(v_i)),
	\end{equation}
	where $f(y):=f(y)=2+(\eta-2)\sum_{i=1}^m y_i^4$
%is a function defined by
%	\begin{equation}
%		f(y)=\sum_{i=1}^m y_i^4 (\eta-2)+2,
%	\end{equation}
for any $y=(y_1,\cdots,y_m)^{\intercal}\in\mathbb{R}^m$. In fact, for any $y=(y_1,\cdots,y_m)^{\intercal}\in\mathbb{R}^m$,
\beaa
y^{\intercal}\tilde{W}y=\sum_{i=1}^m\tilde{w}_{ii}y_i^2+2\sum_{i<j}\tilde{w}_{ij}y_iy_j\sim N(0, \sigma_y^2)
\eeaa
such that $\sigma_y^2$ is equal to
\beaa
\E\left(\sum_{i=1}^m\tilde{w}_{ii}y_i^2+2\sum_{i<j}\tilde{w}_{ij}y_iy_j\right)^2
&=& \eta\sum_{i=1}^my_i^4+4\sum_{i<j}y_i^2y_j^2=f(y)
\eeaa
by independence. Recall $\PP(N(0, 1)\geq x)\leq e^{-x^2/2}$ for all $x\geq 1$.
	Thus, for $x-2r\delta>1$, the first term on the right side of \eqref{eq:tail-general-wig} is bounded by
	\begin{equation}
	\begin{split}
%		&N\cdot \max_{1\leq i\leq N}\PP\left(v_i^{\intercal}\tilde{W}_{\{1,..,m\}}v_i\geq x-2r\delta\right)\\
%		\leq
& N \cdot\sup_{\|y\|=1}\exp\left\{-\frac{(x-2r\delta)^2}{2(\sum_{i=1}^m y_i^4 (\eta-2)+2)}\right\}\\
		= & N \cdot\exp\left\{-\frac{(x-2r\delta)^2}{2\left[(\inf_{\|y\|=1}\sum_{i=1}^m y_i^4) (\eta-2)+2\right]}\right\},
	\end{split}
	\end{equation}
since $0\leq \eta\leq 2$. Observe that
	\begin{equation}
		\inf_{\|y\|=1}\sum_{i=1}^m y_i^4 = \frac{1}{m}.
	\end{equation}
	Thus,
	\begin{equation}\label{eq:first-term-wig-tail}
		N\cdot \max_{1\leq i\leq N}\PP\left(v_i^{\intercal}\tilde{W}_{\{1,..,m\}}v_i\geq x-2r\delta\right)
		\leq N \cdot\exp\left\{-\frac{(x-2r\delta)^2}{2\left[m^{-1} (\eta-2)+2\right]}\right\}
	\end{equation}
	if $x-2r\delta>1$. Now turn to estimate the last probability in
 \eqref{eq:tail-general-wig}. Note that
	\begin{equation}
		\lambda_{1}(\tilde{W}_{\{1,..,m\}})\leq \left(tr\big(\tilde{W}_{\{1,..,m\}}^2\big)\right)^{1/2} = \left(\sum_{i=1}^m\sum_{j=1}^m \tilde{W}_{ij}^2\right)^{1/2}.
	\end{equation}
%	Thus,
%	\begin{equation}
%		\PP\left(\lambda_{1}(\tilde{W}_{\{1,..,m\}})\geq r\right)
%		\leq \PP\left(\sum_{i=1}^m\sum_{j=1}^m \tilde{W}_{ij}^2\geq r^2\right).
%	\end{equation}
	Note that $\sum_{i=1}^m\sum_{j=1}^m \tilde{W}^2_{ij}$ and $\eta Q_1+2Q_2$ have the same distribution, where $Q_1\sim \chi^2_{m}$, $Q_2\sim \chi^2_{m(m-1)/2}$ and $Q_1$ and $Q_2$ are independent. Also $\eta Q_1+2Q_2\leq 2(Q_1+Q_2)\sim 2\cdot \chi^2_{m(m+1)/2}$
	Thus, the last probability in
 \eqref{eq:tail-general-wig} is dominated by
\beaa
			\PP\left(\sum_{i=1}^m\sum_{j=1}^m \tilde{W}^2_{ij}\geq r^2\right)
		%&\leq &\PP\left(
%		\eta Q_1+2Q_2\geq r^2\right)\\
		&\leq & \PP\left( \chi^2_{m(m+1)/2}\geq r^2/2\right)\\
& \leq & \PP\left( \frac{\chi^2_{m(m+1)/2}}{m(m+1)/2}\geq \frac{r^2}{m(m+1)}\right).	
\eeaa
Notice $\frac{r^2}{m(m+1)}\geq 8$ under the given condition $r\geq 4m$. Let $I(x)=\frac{1}{2}(x-1-\log x)$ for $x>0$. It is easy to check that $I(8)=(7-\log 8)/2>2.4$ and that $I(x)=\frac{1}{2}(x-1-\log x)\geq \frac{1}{4}x$ as $x\geq 8.$ By \eqref{tea_coffee}, the last probability above is no more than
\beaa
2\cdot \exp\left(-\frac{m(m+1)}{2}I\Big(\frac{r^2}{m(m+1)}\Big)\right) \leq 2\cdot e^{-r^2/8}.
\eeaa
%	We use the next inequality, which  is due to the tail bound of a chi-squared random variable in Lemma 1 of \cite{laurent2000adaptive}.
%	\begin{equation}
%		\PP\left( Q_1+Q_2-L\geq 2\sqrt{L z}+2z\right)\leq \exp(-z),
%	\end{equation}
%	where $L=m+m(m-1)/2=m(m+1)/2$ and $z>0$. This inequality further implies
%	\begin{equation}
%		\PP\left( Q_1+Q_2\geq \left(\sqrt{L}+\sqrt{2z}\right)^2\right)\leq \exp(-z),
%	\end{equation}
%	for all $z>0$. According to the definition of $L$, we have $\sqrt{L}\leq m\leq \frac{r}{4}$. Thus, the above inequality gives
%	\begin{equation}
%		\PP\left( Q_1+Q_2\geq \left(\frac{r}{4}+\sqrt{2z}\right)^2\right)\leq \exp(-z).
%	\end{equation}
%	Take $z=\frac{9}{32}r^2$ in the above inequality, we arrive at
%% That is, for $r\geq 4m$, we have
%	\begin{equation}
%		\PP\left( Q_1+Q_2\geq r^2\right)\leq \exp\left\{-\frac{9}{32}r^2\right\}\leq \exp\left\{-\frac{1}{64}r^2\right\}.
%	\end{equation}
Hence,
\begin{equation}
 \PP\big(\lambda_{1}(\tilde{W}_{\{1,...,m\}})\geq r\big)\leq 	\PP\left(\sum_{i=1}^m\sum_{j=1}^m \tilde{W}^2_{ij}\geq r^2\right)
		\leq 2\cdot e^{-r^2/8}.
 \end{equation}
 Combining the above display with \eqref{eq:tail-general-wig} and \eqref{eq:first-term-wig-tail}, we have
 \beaa
&& 	\PP\Big(\sup_{\|v\|=1}v^{\intercal} \tilde{W}_{\{1,..,m\}}v\geq x\Big)\\
&\leq & N\cdot\exp\left\{-\frac{(x-2r\delta)^2}{2\left[m^{-1} (\eta-2)+2\right]}\right\} + 2\cdot e^{-r^2/8}.
 \eeaa
The desired conclusion follows since $N\leq m^{1.5}(\log m) \delta^{-m}$ \citep{Rogers:1963tt}.
% Thus, we arrive at
% \begin{equation}
% 	\PP\left(\lambda_{1}(\tilde{W}_{\{1,..,m\}})\geq x\right)
% 	\leq \frac{m^{1.5}\log m \delta^{-m}}{\sqrt{2\pi}}\exp\left\{-\frac{(x-2r\delta)^2}{2(\frac{1}{m} (\eta-2)+2)}\right\} + \exp\left\{-\frac{r^2}{64}\right\}.
% \end{equation}
\end{proof}
\begin{proof}[Proof of Lemma~\ref{lemma:convergence-concepts}]
(i)$\implies$(ii): Easily,
\begin{eqnarray}
		\E\big[e^{\alpha Z_p}\big]
		&=& \E\big[e^{\alpha Z_p}\mathbf{1}_{\{Z_p\geq \delta\} }\big]+
		\E\big[e^{\alpha Z_p}\mathbf{1}_{\{ 0\leq Z_p< \delta \}}\big]\\
		&\leq & \E\big[e^{\alpha Z_p}\mathbf{1}_{\{Z_p\geq \delta\}}\big] + e^{\alpha \delta}.
\end{eqnarray}
Taking $\limsup_{p\to\infty}$ on both sides and then letting $\delta \downarrow 0$, we obtain
\begin{equation}
\limsup_{p\to\infty} \E\big[e^{\alpha Z_p}\big]\leq 1.
\end{equation}
%\begin{equation}
%	\limsup_{p\to\infty} \E[e^{\alpha Z_p}]\leq \lim_{p\to\infty} \E[e^{\alpha Z_p}\mathbf{1}_{\{Z_p\geq \delta\}}]+ e^{\alpha \delta} = e^{\alpha \delta}.
%\end{equation}
%Let $\delta\to 0$ on both sides, we have
%\begin{equation}
%	\limsup_{p\to\infty}\E[e^{\alpha Z_p}]\leq 1.
%\end{equation}
On the other side,  $\liminf_{p\to\infty}\E\big[e^{\alpha Z_p}\big]\geq 1$ since $Z_p\geq 0$. Hence,
$\lim_{p\to\infty} \E\big[e^{\alpha Z_p}\big]=1$.

(ii) $\implies$ (i):
	For each $\beta>0$, we know $\mathbf{1}_{\{Z_p\geq \delta\}}\leq e^{\beta(Z_p-\delta)}
	$. Thus,
	\begin{equation}
		\E[e^{\alpha Z_p}\mathbf{1}_{\{Z_p\geq \delta\}}]
		\leq \E[e^{\alpha Z_p+\beta(Z_p-\delta)}]=e^{-\beta\delta}
		\E[e^{(\alpha+\beta)Z_p}].
	\end{equation}
Taking $\limsup_{p\to\infty}$ on both sides and then letting $\beta\to\infty$, we have
% $\limsup$ on both sides, we have
%	\begin{equation}
%		\limsup_{p\to\infty}\E[e^{\alpha Z_p}\mathbf{1}_{\{Z_p\geq \delta\}}]
%		\leq e^{-\beta\delta}
%		\limsup_{p\to\infty}\E[e^{(\alpha+\beta)Z_p}]
%		= e^{-\beta\delta}.
%	\end{equation}
%	Let $\beta\to\infty$ in the above inequality, we have
	\begin{equation}
		\limsup_{p\to\infty}\E[e^{\alpha Z_p}\mathbf{1}_{\{Z_p\geq \delta\}}]\leq 0.
	\end{equation}
	Thus, $\lim_{p\to\infty}\E[e^{\alpha Z_p}\mathbf{1}_{\{Z_p\geq \delta\}}]=0$.

	(ii)$\implies$ (iii): First, $\E [Z_p^{\alpha}]=\alpha\int_0^{\infty}x^{\alpha-1}P(Z_p\geq x)\,dx$. By the Markov inequality, $P(Z_p\geq x)\leq e^{-\beta x}\E e^{\beta Z_p}$ for all $x>0$ and $\beta>0$. It follows that
\begin{eqnarray*}
\E [Z_p^{\alpha}] \leq \alpha\, \big(\E e^{\beta Z_p}\big)\int_0^{\infty}x^{\alpha-1}e^{-\beta x}\,dx=\frac{\alpha\Gamma(\alpha)}{\beta^{\alpha}}\E e^{\beta Z_p}
\end{eqnarray*}
for all  $\beta>0$. The conclusion then follows by first letting $p\to\infty$ and then sending $\beta\to \infty.$

%Note that $\lim_{z\to\infty}\frac{z^{\alpha}}{e^{z}-1}=0$ for all $\alpha>0$ and $\frac{z^{\alpha}}{e^{z}-1}$ is continuous for $z\geq \delta$. Thus, there is a constant $\kappa$ (only depending on $\alpha$ and $\delta$), such that
%	\begin{equation}
%		z^{\alpha}\leq \kappa e^{z}-1
%	\end{equation}
%	for all $z\geq \delta$. Thus,
%	$
%		\E [Z_p^{\alpha}{\mathbf{1}_{\{Z_p\geq \delta \} }}]\leq \kappa \E[e^{Z_p}\mathbf{1}_{\{
%		Z_p\geq\delta
%		\} }].$
%	Taking $\limsup$ on both sides gives $\lim_{p\to\infty}\E [Z_p^{\alpha}{\mathbf{1}_{\{Z_p\geq \delta \} }}]=0. $
%	Thus, we have
%	\begin{equation}
%		\limsup_{p\to\infty}\E [Z_p^{\alpha}]
%		\leq \lim_{p\to\infty}\E [Z_p^{\alpha}{\mathbf{1}_{\{Z_p\geq \delta \} }}]+\delta^{\alpha}=\delta^{\alpha}
%	\end{equation}
%	for all $\delta>0$. Let $\delta\to 0$ in the above inequality, we arrive at $\lim_{p\to\infty}\E [Z_p^{\alpha}]$.

(iii) $\implies$ (iv): This is a direct consequence of the Chebyshev inequality and the equality $\lim_{p\to\infty}\E(Z_p)=0$.

(iii) $\implies$ (v): Let $\alpha=2$ in (iii), then $\limsup_{p}{\rm Var}(Z_p)\leq \lim_{p\to\infty}\E(Z_p^2)=0.$

\end{proof}

\end{document}